\documentclass{article}
\usepackage[toc,page]{appendix}
\usepackage[utf8]{inputenc}
\usepackage[english]{babel}
\DeclareMathAlphabet{\mathscr}{OT1}{pzc}{m}{it} 
\usepackage{amsmath}
\usepackage{dsfont} 
\usepackage{amssymb}
\usepackage[T1]{fontenc}
\usepackage{mathtools}   % loads »amsmath«
\usepackage{amsfonts}
\usepackage{bbm}
\usepackage{stmaryrd}
\usepackage{mathrsfs}  
\usepackage{graphicx}
\usepackage{bigints}
\usepackage{relsize}
\usepackage[colorinlistoftodos]{todonotes}
\usepackage{amssymb,amsmath,amsthm}
\numberwithin{equation}{section}
\usepackage{dsfont}
\usepackage[shortlabels]{enumitem}
\setlist{labelindent=\parindent,leftmargin=*}
\usepackage{authblk}

\newtheorem{theorem}{Theorem}[section]
\newtheorem{notation}[theorem]{Notation}
\newtheorem{lemma}[theorem]{Lemma}
\newtheorem{proposition}[theorem]{Proposition}
\newtheorem{corollary}[theorem]{Corollary}
\newtheorem{definition}[theorem]{Definition}
\newtheorem{hypothesis}[theorem]{Hypothesis}
\newtheorem{remark}[theorem]{Remark}
\newenvironment{prooff}[1]{\begin{trivlist}
\item {\it \bf Proof}\quad} {\qed\end{trivlist}}
\usepackage{amsmath,amssymb}
\makeatletter
\newsavebox\myboxA
\newsavebox\myboxB
\newlength\mylenA
\newcommand*\xoverline[2][0.75]{%
    \sbox{\myboxA}{$\m@th#2$}%
    \setbox\myboxB\null% Phantom box
    \ht\myboxB=\ht\myboxA%
    \dp\myboxB=\dp\myboxA%
    \wd\myboxB=#1\wd\myboxA% Scale phantom
    \sbox\myboxB{$\m@th\overline{\copy\myboxB}$}%  Overlined phantom
    \setlength\mylenA{\the\wd\myboxA}%   calc width diff
    \addtolength\mylenA{-\the\wd\myboxB}%
    \ifdim\wd\myboxB<\wd\myboxA%
       \rlap{\hskip 0.5\mylenA\usebox\myboxB}{\usebox\myboxA}%
    \else
        \hskip -0.5\mylenA\rlap{\usebox\myboxA}{\hskip 0.5\mylenA\usebox\myboxB}%
    \fi}
\makeatother

\title{Backward Stochastic Differential Equations with no driving martingale, Markov processes and associated Pseudo Partial Differential Equations.}

\author{
Adrien BARRASSO \thanks{ENSTA ParisTech, Unit\'e de Math\'ematiques
 appliqu\'ees, 828, boulevard des Mar\'echaux, F-91120 Palaiseau, France 
 and Ecole Polytechnique,  F-91128 Palaiseau, France.
E-mail: \sf adrien.barrasso@ensta-paristech.fr \\}
\qquad\quad
Francesco RUSSO\thanks{ENSTA ParisTech, Unit\'e de Math\'ematiques appliqu\'ees, 828, boulevard des Mar\'echaux, F-91120 Palaiseau, France. E-mail: \sf francesco.russo@ensta-paristech.fr}}

\date{December 2017}

\begin{document}
\maketitle

{\bf Abstract.}

We discuss a class of Backward Stochastic Differential Equations
(BSDEs) with no driving martingale. When the randomness of the driver 
depends on a general Markov process $X$, those BSDEs are denominated 
Markovian BSDEs and can be associated to a deterministic problem,
called Pseudo-PDE which constitute the natural generalization of a  parabolic
semilinear PDE which naturally appears when the underlying filtration 
is Brownian. We consider two aspects of well-posedness for
the Pseudo-PDEs: {\it classical} and {\it martingale} solutions. 

\bigskip 
{\bf MSC 2010} Classification. 
60H30; 60H10; 35S05; 60J35; 60J60; 60J75.


%\newpage
\section{Introduction}

This paper focuses on 
a new concept of 
  Backward Stochastic Differential  Equation (in short BSDE)
with no driving martingale of the form
\begin{equation}\label{BSDEIntro2}
Y_t = \xi + \int_t^T {\hat f} \left(r,\cdot,Y_r,\sqrt{\frac{d\langle M\rangle}{dV}}(r)\right)dV_r  -(M_T - M_t),
\end{equation}
defined on a fixed stochastic basis fulfilling the usual conditions. 
 $V$ is a given bounded non-decreasing continuous adapted process, 
  $\xi$ (resp. $\hat f$) is a prescribed terminal condition
(resp. driver). The unknown will be a couple of cadlag adapted processes $(Y,M)$
where $M$ is a martingale.
When $V_t=t$  \eqref{BSDEIntro2} is a particular case
of the class of BSDEs introduced and studied by \cite{qian}, for which bring a new light.
%%% 
%that we have only discovered after concluding the paper. 

A special case of such BSDEs are the Markovian BSDEs of the form
\begin{equation}\label{BSDEIntro}
Y^{s,x}_t = g(X_T) + \int_t^T f\left(r,X_r,Y^{s,x}_r,\sqrt{\frac{d\langle M^{s,x}\rangle}{dV}}(r)\right)dV_r  -(M^{s,x}_T - M^{s,x}_t),
\end{equation}
defined in a canonical space $\left(\Omega,\mathcal{F}^{s,x},(X_t)_{t\in[0,T]},(\mathcal{F}^{s,x}_t)_{t\in[0,T]},\mathbbm{P}^{s,x}\right)$ \\
where 
$(\mathbbm{P}^{s,x})_{(s,x)\in[0,T]\times E}$ corresponds to the laws
 (for different starting times $s$ and starting points $x$) of an
 underlying forward  Markov process  with time index $[0,T]$, 
taking values in a 
Polish state space $E$. Indeed this Markov process is 
supposed to solve a  {\it martingale problem} 
with respect to  a given {\it deterministic} operator $a$, 
which is the natural generalization of stochastic differential
equation in law.
% and which is
% characterized as the solution of a martingale problem
% related to a certain operator $a$.
\eqref{BSDEIntro} will be naturally associated with a deterministic 
problem involving $a$, which will be called {\it Pseudo-PDE}, 
being of the type 
\begin{equation}\label{PDEIntro}
\left\{
\begin{array}{rccc}
 a(u)(t,x) + f\left(t,x,u(t,x),\sqrt{\Gamma(u,u)}(t,x)\right)&=&0& \text{ on } [0,T]\times E   \\
 u(T,\cdot)&=&g,& 
\end{array}\right.
\end{equation}
where $\Gamma(u,u)=a(u^2)-2ua(u)$ is a potential theory operator called
 the {\it carr\'e du champs operator}. 
The Markovian BSDE 
\eqref{BSDEIntro} 
  seems to be appropriated  in the case when
 the forward underlying process $X$ is a general Markov process
which does not rely to a fixed reference process or random field
 as  a Brownian motion or a Poisson measure.
\\

The classical notion of Brownian BSDE was introduced  in 1990 by E. Pardoux and
S. Peng in \cite{parpen90}, after an early work
of J.M. Bismut in 1973 in \cite{bismut}. It is  a stochastic differential
equation with prescribed  terminal condition 
$ \xi$ and  driver $\hat f$; the
unknown is a couple $(Y,Z)$ of adapted processes. Of particular interest
is the case when the randomness of the
 driver is expressed through a forward diffusion process $X$
and the terminal condition only depends on $X_T$.
 %  the BSDE is a Forward BSDE (FBSDE).
The solution, when it exists, is usually indexed by the starting time
$s$ and starting point $x$ of the forward diffusion $X= X^{s,x}$, and it 
is expressed by
\begin{equation}\label{BSDEIntroN}
\left\{\begin{array}{rcl}
X^{s,x}_t &=& x+ \int_s^t \mu(r,X^{s,x}_r)dr +\int_s^t \sigma(r,X^{s,x}_r)dB_r\\
Y^{s,x}_t &=& g(X^{s,x}_T) + \int_t^T f\left(r,X^{s,x}_r,Y^{s,x}_r,Z^{s,x}_r\right)dr  -\int_t^T Z^{s,x}_rdB_r,
\end{array}\right.
\end{equation}
where $B$ is a Brownian motion. 
Existence and uniqueness of \eqref{BSDEIntroN}
(that we still indicate with  BSDE)  above was established first
supposing essentially Lipschitz conditions on $f$ with respect to the third
and fourth variable. $\mu$ and $\sigma$ were also supposed to be 
Lipschitz (with respect to $x$). In the sequel those conditions were
considerably relaxed, see \cite{PardouxRascanu}
 and references therein.
\\
 In \cite{peng1991probabilistic} and in \cite{pardoux1992backward} 
previous BSDE was linked 
to the semilinear PDE
\begin{equation}\label{PDEparabolique}
\left\{
\begin{array}{l}
\partial_tu + \frac{1}{2}\underset{i,j\leq d}{\sum} (\sigma\sigma^\intercal)_{i,j}\partial^2_{x_ix_j}u + \underset{i\leq d}{\sum} \mu_i\partial_{x_i}u + f(\cdot,\cdot,u,\sigma\nabla u)=0\quad \text{ on } [0,T[\times\mathbbm{R}^d \\
u(T,\cdot) = g. 
\end{array}\right.
\end{equation}
In particular, if \eqref{PDEparabolique} has a classical smooth solution  $u$ 
then $(Y^{s,x},Z^{s,x}):=(u(\cdot,X^{s,x}_{\cdot}),\sigma\nabla u(\cdot,X^{s,x}_{\cdot}))$ solves the second line of \eqref{BSDEIntroN}. Conversely,  %on $\mu,\sigma,f,g$, 
only under the Lipschitz type  conditions mentioned after \eqref{BSDEIntroN},
the solution of the BSDE can be expressed 
as a function of the forward process
 $(Y^{s,x},Z^{s,x})=(u(\cdot,X^{s,x}_{\cdot}),v(\cdot,X^{s,x}_{\cdot}))$,
see  \cite{el1997backward}. 
%, ici on donne l'impression que $u$ est solution de viscosité dès que la BSDE a une solution mais ça n'est pas complètement vrai. Par exemple si $g,f$ ne sont pas continues en $x$, la BSDE a une solution mais $u$ n'est pas solution de viscosité 
  When $f$ and $g$ are continuous, 
 $u$ is a viscosity solution of \eqref{PDEparabolique}.
%The existence of $v$ was established by  \cite{el1997backward}. 
Excepted in the case when $u$
has some minimal differentiability properties, see e.g.
\cite{fuhrman2005generalized}, it is difficult to say something more
on $v$. 
One major contribution of this paper consists in specifying $v$.
\\

 Since the pioneering work  of \cite{pardoux1992backward}, 
in the Brownian case,
the relations between more general BSDEs and associated
 deterministic problems 
 have been 
studied extensively, and innovations have been made in several directions.
\\ 
In \cite{barles1997backward} 
the authors  introduced 
 a new kind of BSDE including a term with jumps generated by a
 Poisson measure, where an underlying forward process $X$ solves
a jump diffusion equation with Lipschitz type conditions.
 They associated with it  an Integro-Partial Differential
 Equation (in short IPDE) in which some non-local operators are added 
to the classical partial differential maps, and proved that, under some 
continuity conditions on 
the coefficients, the BSDE provides a viscosity solution of the IPDE. 
In chapter 13 of
 \cite{barles1997sde}, under some specific conditions on the coefficients of
 a Brownian BSDE, one produces  a  solution in the sense of distributions
 of the parabolic PDE. Later,   the notion of mild 
solution of the PDE  was used  in \cite{BSDEmildPardouxBally}
 where the authors tackled diffusion operators generating symmetric Dirichlet forms and associated Markov processes thanks to the theory of Fukushima Dirichlet forms, see e.g. \cite{fuosta}. Those results were extended to the case of non symmetric Markov processes in \cite{ZhuMildBSDE}. Infinite dimensional setups were considered  for example  
in  \cite{fuhrman2005generalized} where an infinite
 dimensional BSDE  could produce the mild solution  of a PDE on a 
Hilbert space. 
Concerning  the study of BSDEs driven by more general martingales than
Brownian motion, we have already mentioned BSDEs driven by Poisson measures.
In this respect, more recently, BSDEs driven by marked point processes were introduced in \cite{Confortola}, see also \cite{BandiniBSDE}; in that case
the underlying process does not contain any diffusion term. 
 Brownian BSDEs involving a supplementary orthogonal term were studied in
 \cite{el1997backward}. 
 We can also  mention the study of
 BSDEs driven by a general martingale in \cite{sant}.  
BSDEs  of the same type, but with partial information
have been investigated in \cite{cretarola}.
A first approach to face deterministic problems for those
equations  appears in \cite{laachir};
that paper also contains an application to financial hedging in incomplete
markets. Finally, BSDEs in general filtered space were studied in \cite{qian} as we have already mentioned.
\\

We come back to the motivations of the paper. 
Besides introducing and studying the new class of BSDEs \eqref{BSDEIntro2},
(resp. Markovian BSDEs \eqref{BSDEIntro}),  
we   study  the corresponding 
Pseudo-PDE \eqref{PDEIntro}
and carefully explore their relations 
in the spirit of the existing links between \eqref{BSDEIntroN} and
 \eqref{PDEparabolique}.
For the Pseudo-PDE, we analyze well-posedness at two different levels:
 {\it classical} 
 solutions, which generalize the $C^{1,2}$-solutions of
 \eqref{PDEparabolique} and the so called {\it martingale solutions}.
In the companion paper (see \cite{paper2}), we  also discuss 
 other (analytical)  solutions, that we denominate as {\it decoupled mild} solutions.
The main contributions  of the paper are essentially the following.
In Section \ref{S2} we introduce the notion of BSDE with no driving martingale 
\eqref{BSDEIntro2}.  Theorem \ref{uniquenessBSDE}
 states existence and uniqueness of a solution for that BSDE,
when the final condition $\xi$ is square integrable and the driver $\hat f$
 verifies some integrability and Lipschitz conditions.
For technical reasons we have decided to provide an independent 
constructive proof from the one of \cite{qian}. 
Indeed we need that construction 
for the sequel of the paper. On the other hand, the particular form of our
BSDE allows a simple and direct proof.

%Moreover we study the space of square integrable martingales whose angular brackets are absolutely continuous with respect to the 
%reference  increasing process $V$. 
 In Section \ref{SecMProcess}, we consider an operator and its 
domain $(a,\mathcal{D}(a))$; $V$ will be a  continuous non-decreasing function.
That section is devoted to the formulation of 
the martingale problem concerning our underlying process $X$.
%and the  related  maps $a$ and $\Gamma$.
For each initial time $s$ and initial point $x$
the solution will be a probability 
  $\mathbbm{P}^{s,x}$ under which for any $\phi\in\mathcal{D}(a)$,
\begin{equation*}
\phi( \cdot,X_{\cdot}) - \phi(s,x)-\int_s^{\cdot}a(\phi)(r,X_r)dV_r
\end{equation*}
is a local martingale starting in zero at time $s$. 
We will then assume that this martingale problem is well-posed and  that its solution
 $(\mathbbm{P}^{s,x})_{(s,x)\in[0,T]\times E}$ defines a Markov process. 
In Proposition \ref{H2V=H2}, we prove that, under 
each one of those probabilities, 
 the   angular bracket of
every square integrable martingale is absolutely 
continuous with respect to $dV$. In Definition \ref{domainextended}, we 
suitably define some
 extended domains for the operators $a$ and $\Gamma$, using some locally convex topology.
 In Section \ref{SecPDE}  
 we introduce the Pseudo-PDE \eqref{PDEIntro}
 to which we associate the Markovian BSDE \eqref{BSDEIntro}, considered under
  every $\mathbbm{P}^{s,x}$. We also introduce the  notions of  
{\it classical  solution} in Definition \ref{MarkovPDE}, and of 
{\it martingale solution}
 in Definition \ref{D417}, which is fully probabilistic.
Proposition \ref{CoroClassic} says the following.
Classical solutions of  \eqref{PDEIntro}
 typically belong to the domain $\mathcal{D}(a)$
and are shown also to be essentially martingale solutions.
Conversely  a martingale solution belonging 
 to $\mathcal{D}(a)$ is  a classical solution,
up to so called zero potential sets,  see Definition \ref{zeropotential}.
 Proposition \ref{classicimpliesBSDE} asserts that, given  a
 classical solution $u\in\mathcal{D}(a)$, then for any $(s,x)$ the processes $Y^{s,x} = u(\cdot,X_{\cdot})$ and 
 \\
$M^{s,x} = u(\cdot,X_{\cdot})-u(s,x)-\int_s^{\cdot}f(\cdot,\cdot,u,\sqrt{\Gamma(u,u)})(r,X_r)dV_r$ solve \eqref{BSDEIntro} under the probability 
$\mathbbm{P}^{s,x}$.

 Theorems \ref{RMartExistence} and \ref{P418} state that the function $u$ is 
the unique martingale solution  of \eqref{PDEIntro}.
Moreover $v$ is also identified as a function of $u$
through the an extesion of the carr\'e du champs operator.
This is the consequence of 
% JE VOUDRAI APPELER PROPOSITION  LE THEOREME EN QUESTION
Theorem \ref{Defuv}, which states that, without any assumptions of regularity, 
there exist Borel functions $u$ and $v$ such that for any
 $(s,x)\in[0,T]\times E$,
 the solution of \eqref{BSDEIntro} verifies
\begin{equation*}
\left\{\begin{array}{l}
     \forall t\geq s: Y^{s,x}_t = u(t,X_t)\quad \mathbbm{P}^{s,x}\text{ a.s.} \\
     \frac{d\langle M^{s,x}\rangle}{dV}(t)=v^2(t,X_t)\quad dV\otimes d\mathbbm{P}^{s,x}\text{ a.e.}
\end{array}\right.
\end{equation*}
%%%  EXPLIQUER LE ROLE DE $v$

 % In particular a classical solution is essentially 
% a martingale solution and there is 
% at most one martingale solution.
% in the classical sense 
%up to  a so called  zero potential set, see Definition \ref{zeropotential}.
%The function $u$ 
% will always be  a mild solution of \eqref{PDEIntro}. 
%whereas several assumptions will have to be strengthened 
% in order for  $u$  to be a viscosity solution of \eqref{PDEIntro}.
%For that reason, the notion of mild solution will appear to be the 
% most natural one  at the analytical level.
In Section \ref{SUpcoming} we list some examples which are developed in
 \cite{paper2}. These include Markov processes defined as weak solutions of 
Stochastic Differential Equations (in short SDEs)  including  possible 
jump terms, $\alpha$-stable L\'evy processes associated to 
fractional Laplace operators, solutions of SDEs with distributional drift and diffusions on compact manifolds.

\section{Preliminaries}

In the whole paper we will use the following notions, notations and vocabulary.
\\
\\
A topological space $E$ will always be considered as a measurable space with its Borel $\sigma$-field which shall be denoted $\mathcal{B}(E)$ and if $(F,d_F)$ is a metric space, $\mathcal{C}(E,F)$ (respectively $\mathcal{C}_b(E,F)$, $\mathcal{B}(E,F)$, $\mathcal{B}_b(E,F)$)  will denote the set of functions from $E$ to $F$ which are continuous (respectively bounded continuous, Borel, bounded Borel).
\\
%Let $(\Omega,\mathcal{F})$, $(E,\mathcal{E})$ be two measurable spaces. A measurable mapping from $(\Omega,\mathcal{F})$ to $(E,\mathcal{E})$ shall often be called a \textbf{random variable} (with values in $E$), or in short r.v. If $\mathbbm{T}$ is some set, an indexed set of r.v. with values in $E$, $(X_t)_{t\in \mathbbm{T}}$ will be called a \textbf{random field} (indexed by $\mathbbm{T}$ with values in $E$). In particular, if $\mathbbm{T}$ is an interval included in $\mathbbm{R}_+$, $(X_t)_{t\in \mathbbm{T}}$ will be called a \textbf{stochastic process} (indexed by $\mathbbm{T}$ with values in $E$). Given a stochastic process, if the mapping 
%\begin{equation*}
%	\begin{array}{rcl}
%	(t,\omega)&\longmapsto&X_t(\omega)\\
%	(\mathbbm{T}\times\Omega,\mathcal{B}(\mathbbm{T})\otimes\mathcal{F})&
%\longrightarrow&(E,\mathcal{E})
%	\end{array}
%\end{equation*}
%is measurable, then the process $(X_t)_{t\in \mathbbm{T}}$ will be called a \textbf{measurable process} (indexed by $\mathbbm{T}$ with values in $E$).
\\
On a fixed probability space $\left(\Omega,\mathcal{F},\mathbbm{P}\right)$, for any $p \ge 1$,
 $L^p$ will denote the set of random variables with finite $p$-th moment.
%Two random fields (or stochastic processes) $(X_t)_{t\in \mathbbm{T}}$, $(Y_t)_{t\in \mathbbm{T}}$ indexed by the same set and with values in the same space will be said to be \textbf{modifications (or versions) of each other} if for every $t\in\mathbbm{T}$, $\mathbbm{P}(X_t=Y_t)=1$.
A measurable space equipped with a right-continuous filtration $\left(\Omega,\mathcal{F},(\mathcal{F}_t)_{t\in\mathbbm{T}}\right)$ (where $\mathbbm{T}$ is equal to $\mathbbm{R}_+$ or to $[0,T]$ for some $T\in\mathbbm{R}_+^*$) will be called a \textbf{filtered space}. 
A probability space equipped with a right-continuous filtration $\left(\Omega,\mathcal{F},(\mathcal{F}_t)_{t\in\mathbbm{T}},\mathbbm{P}\right)$  will be called called a \textbf{stochastic basis} and will be said to \textbf{fulfill the usual conditions} if the probability space is complete and if $\mathcal{F}_0$ contains all the $\mathbbm{P}$-negligible sets.
We introduce now some notations and vocabulary about 
 spaces of stochastic processes, on
 a fixed stochastic basis $\left(\Omega,\mathcal{F},(\mathcal{F}_t)_{t\in\mathbbm{T}},\mathbbm{P}\right)$.
% we will use the following notations and vocabulary,
 Most of them are taken or adapted from \cite{jacod79} or \cite{jacod}.
A process $(X_t)_{t \in \mathbbm{T}}  $ is said to be {\bf integrable} if $X_t$ is an integrable r.v.
for any $t$. 
We will denote $\mathcal{V}$ (resp $\mathcal{V}^+$)  the set of adapted, bounded variation (resp non-decreasing) processes starting at 0; $\mathcal{V}^p$ (resp $\mathcal{V}^{p,+}$) the elements of $\mathcal{V}$ (resp $\mathcal{V}^+$) which are predictable, and $\mathcal{V}^c$ (resp $\mathcal{V}^{c,+}$) the elements of $\mathcal{V}$ (resp $\mathcal{V}^+$) which are continuous. If $A \in \mathcal{V}$, we will denote  $Pos(A)$ and $Neg(A)$ the positive variation and negative variation parts of $A$, meaning the unique pair of elements $\mathcal{V}^+$ such that $A=Pos(A)-Neg(A)$  (see Proposition I.3.3 in \cite{jacod} for their existence) and $Var(A)=Pos(A)+Neg(A)$ its total variation. $\mathcal{M}$ 
will be the space of cadlag martingales.  
For any $p\in[1,\infty]$  $\mathcal{H}^p$ will denote the Banach space of elements of $\mathcal{M}$ for which $\| M\|_{\mathcal{H}^p}:=\mathbbm{E}[|\underset{t\in \mathbbm{T}}{\text{sup }}M_t|^p]^{\frac{1}{p}}<\infty$ and in this set we identify indistinguishable elements. $\mathcal{H}^p_0$ will denote the Banach subspace of $\mathcal{H}^p$
of elements vanishing at zero.
\\
If $\mathbbm{T}=[0,T]$ for some $T\in\mathbbm{R}_+^*$, a stopping time
 will take values in
% be defined as a random variable with values in 
$[0,T]\cup\{+\infty\}$.
% such that for any $t\in[0,T]$, $\{\tau\leq t\}\in \mathcal{F}_t$. 
We define a \textbf{localizing sequence of stopping times} as an a.s. increasing sequence of stopping times $(\tau_n)_{n\geq 0}$ such that there a.s. exists $N\in\mathbbm{N}$ for which $\tau_N=+\infty$. Let $Y$ be a process and $\tau$ a stopping time, we denote by  $Y^{\tau}$ the \textbf{stopped process}
 $t\mapsto Y_{t\wedge\tau}$.
% which we call \textbf{stopped process}.  
If $\mathcal{C}$ is a set of processes, we define its \textbf{localized class} $\mathcal{C}_{loc}$ as the set of processes $Y$ such that there exists a localizing sequence $(\tau_n)_{n\geq 0}$ such that for every $n$, the stopped process $Y^{\tau_n}$ belongs to $\mathcal{C}$. In particular a process $X$ is said to
be locally integrable (resp. locally square integrable) 
if there is a localizing sequence $(\tau_n)_{n\geq 0}$ such that for every $n$, $X_t^{\tau_n}$
is integrable (resp. square integrable) for every $t$.
\\
For any $M\in  \mathcal{M}_{loc}$, we denote $[M]$ its \textbf{quadratic variation} and if moreover  $M\in\mathcal{H}^2_{loc}$, $\langle M\rangle$ will denote its (predictable) \textbf{angular bracket}.
$\mathcal{H}_0^2$ will be equipped with scalar product defined by $(M,N)_{\mathcal{H}_0^2}=\mathbbm{E}[M_TN_T] 
=\mathbbm{E}[\langle M, N\rangle_T] $ which makes it a Hilbert space.
Two local martingales $M,N$ will be said to be \textbf{strongly orthogonal} if $MN$ is a local martingale starting in 0 at time 0. In $\mathcal{H}^2_{0,loc}$ this notion is equivalent to $\langle M,N\rangle=0$.
%\\
%If $M\in  \mathcal{M}_{loc}$, and $p\in[1,\infty]$. We denote $L^p(M)$ the set of predictable processes $H$ such that $\mathbbm{E}\left[\left(\int_0^TH^2_rd[M]_r\right)^{\frac{p}{2}}\right]<\infty$.  This implies that $\int_0^{\cdot}H_rM_r$ belongs to  $\mathcal{H}^p$.

\section{BSDEs without driving martingale}\label{S2}

In the whole present section we are given $T\in\mathbbm{R}_+^*$, and a stochastic basis $\left(\Omega,\mathcal{F},(\mathcal{F}_t)_{t\in[0,T]},\mathbbm{P}\right)$ fulfilling the usual conditions.  
Some proofs and intermediary results of the first part of this
section are postponed  to Appendix \ref{SC}.

\begin{definition} \label{StochMeasures}
Let $A$ and $B$ be in $\mathcal{V}^+$. We will say that $dB$ dominates $dA$
{\bf in the sense of stochastic measures} (written $dA\ll dB$) if for almost all $\omega$, $dA(\omega)\ll dB(\omega)$ as Borel measures on $[0,T]$.
\\
\\
We will say that $dB$ and $dA$ are mutually singular {\bf in the sense of stochastic measures} (written $dA \bot dB$) if for almost all $\omega$,  the Borel measures
$dA(\omega)$ and $dB(\omega)$ are mutually singular.
\\
\\
Let $B\in \mathcal{V}^+$.  $dB\otimes d\mathbbm{P}$ will denote the positive measure on 
\\
$(\Omega\times [0,T],\mathcal{F}\otimes\mathcal{B}([0,T]))$ defined for any $F\in\mathcal{F}\otimes\mathcal{B}([0,T])$ by 
\\
$dB\otimes d\mathbbm{P}( F) = \mathbbm{E}\left[\int_0^{T}\mathds{1}_F(r,\omega)dB_r(\omega)\right]$.
A  property which holds  true everywhere except on a null set for
 this measure will be said to be true $dB\otimes d\mathbbm{P}$ almost 
everywhere (a.e).
\end{definition}
%Proposition below admits a straightforward proof.
%\begin{proposition}\label{EqualdVdP}
%Let $\phi$, $\psi$ be two measurable mappings from 
%\\
%$(\Omega\times [0,T],\mathcal{F}\otimes\mathcal{B}([0,T]))$ to $(\mathbbm{R},\mathcal{B}(\mathbbm{R}))$, then if $\phi=\psi$ $dB\otimes d\mathbbm{P}$ a.e, we have for $\mathbbm{P}$ almost all $\omega$ that $(\phi(\omega)=\psi(\omega)\, dB(\omega)$ a.e.)
%\end{proposition}
% We set $F=\{(\omega,t)\in\Omega\times [0,T]:\, \phi(\omega,t)\neq \psi(\omega,t)\}$ which is measurable. By definition, if $\phi=\psi$ $dB\otimes d\mathbbm{P}$ a.e, then  $dB\otimes d\mathbbm{P}(F)=\mathbbm{E}\left[\int_0^{T}\mathds{1}_{\{\phi\neq \psi\}}dB\right]=0$, so $\int_0^{T}\mathds{1}_{\{\phi\neq \psi\}}dB=0$ $\mathbbm{P}$ a.s.
% \end{proof}

The proof of Proposition below is in Appendix \ref{SC}.
\begin{proposition}\label{Decomposition}
For any $A$ and $B$ in $\mathcal{V}^{p,+}$, there exists a  
(non-negative 
\\
$dB\otimes d\mathbbm{P}$ a.e.)
predictable process $\frac{dA}{dB}$  
%unique up to a null set for $dB\otimes d\mathbbm{P}$,
 and a  process in $\mathcal{V}^{p,+}$ $A^{\bot B}$ such that 
\begin{equation*}\label{eqDecompo}
dA^{\bot B}\bot\, dB \text{    and    } A = 
 A^{ B} 
%\int_0^{\cdot}\frac{dA}{dB}(r)dB_r
 + A^{\bot B} \text{  a.s.}
\end{equation*}
where $ A^{B} = \int_0^{\cdot}\frac{dA}{dB}(r)dB_r.$
The process $A^{\bot B}$ is  unique and
the process  $\frac{dA}{dB}$  is 
unique $dB\otimes d\mathbbm{P}$ a.e. \\
%The process $\int_0^{\cdot}\frac{dA}{dB}(s)dB_s$ will also be denoted by $A^B$.
Moreover, 
%$\int_0^{\cdot}\frac{dA}{dB}(s)dB_s$, also  denoted by $A^B$, and 
there exists a predictable process $K$  with values in  $[0,1]
$ (for every $(\omega,t)$), such that $A^B=\int_0^{\cdot} \mathds{1}_{\{K_r < 1\}}dA_r$ and $A^{\perp B}=\int_0^{\cdot} \mathds{1}_{\{K_r = 1\}}dA_r$.
\end{proposition}

The predictable process 
$\frac{dA}{dB}$ appearing in the statement of Proposition \ref{Decomposition}
will beF the \textbf{Radon-Nikodym derivative} of $A$ by $B$.

\begin{remark}\label{RemIncr}
Since for any $s<t$ 
$A_t - A_s = \int_s^t\frac{dA}{dB}(r)dB_r + A^{\bot B}_t - A^{\bot B}_s$ a.s.
where $A^{\bot B}$ is increasing, it is clear that for any $s<t$, 
\\
$\int_s^t\frac{dA}{dB}(r)dB_r \leq A_t - A_s$   a.s. and therefore that for any positive measurable process $\phi$ we have 
$\int_0^{T}\phi_r\frac{dA}{dB}(r)dB_r \leq \int_0^{T}\phi_rdA_r$  a.s.
\end{remark}
%If $A \in \mathcal{V}$, we will denote $Pos(A)$ and $Neg(A)$ the positive variation and negative variation parts of $A$, meaning the unique pair of elements $\mathcal{V}^+$ such that $A=Pos(A)-Neg(A)$, see Proposition I.3.3 in \cite{jacod} for their existence.
If $A$ is in $\mathcal{V}^p$, and $B\in\mathcal{V}^{p,+}$. We set $\frac{dA}{dB} := \frac{dPos(A)}{dB}-\frac{dPos(A)}{dB}$ and $A^{\bot B}:=(Pos(A))^{\bot B}-(Neg(A))^{\bot B}$.

\begin{proposition}\label{linearity}
	Let $A_1$ and $A_2$ be in $\mathcal{V}^p$, and $B\in\mathcal{V}^{p,+}$. Then, 
	\\
	$ \frac{d(A_1+A_2)}{dB}=\frac{dA_1}{dB}+\frac{dA_2}{dB}$ $dV\otimes d\mathbbm{P}$ a.e. and $(A_1+A_2)^{\perp B}=A_1^{\perp B}+A_2^{\perp B}$.
\end{proposition}
\begin{proof}
The proof is an immediate consequence of the uniqueness of the decomposition \eqref{eqDecompo}.
\end{proof}

Let $V\in\mathcal{V}^{p,+}$. We introduce two significant spaces related to $V$.
\\
 $\mathcal{H}^{2,V} := \{M\in\mathcal{H}^2_0|d\langle M\rangle \ll dV\}$ and
$\mathcal{H}^{2,\perp V} := \{M\in\mathcal{H}^2_0|d\langle M\rangle \perp dV\}$.  
\\
\\
The proof of the two propositions below are in Appendix \ref{SC}.
\begin{proposition}\label{DecompoMart}
Let $M\in\mathcal{H}_0^2$, and let $V\in\mathcal{V}^{p,+}$. There exists a pair $(M^V,M^{\bot V})$ in $\mathcal{H}^{2,V}\times\mathcal{H}^{2,\perp V}$ such that $M = M^V+M^{\bot V}$ and $\langle M^V, M^{\bot V}\rangle = 0$.
%\begin{enumerate}\label{E24}
%\item $M = M^V+M^{\bot V}$;
%\item $d\langle M^V \rangle \ll dV$;
%\item $d\langle M^{\bot V} \rangle \bot dV$; 
%\item $\langle M^V, M^{\bot V}\rangle = 0$.
%\end{enumerate} 
  
Moreover, we have  $\langle M^V \rangle = \langle M \rangle^V = \int_0^{\cdot} \frac{d\langle M\rangle}{dV}(r)dV_r$ and $\langle M^{\bot V} \rangle = \langle M \rangle^{\bot V}$ and  there exists a predictable process $K$ with values in $[0,1]$ such that 
\\
$M^V=\int_0^{\cdot} \mathds{1}_{\{K_r < 1\}}dM_r$ and $M^{\bot V}=\int_0^{\cdot} \mathds{1}_{\{K_r = 1\}}dM_r$.
\end{proposition}
%The proof of the proposition below  is in Appendix \ref{SC}.
\begin{proposition}\label{OrthogonalSpaces}
$\mathcal{H}^{2,V}$ and $\mathcal{H}^{2,\perp V}$ are orthogonal sub-Hilbert spaces of $\mathcal{H}^2_0$ and $\mathcal{H}^2_0 = \mathcal{H}^{2,V}\oplus^{\perp}\mathcal{H}^{2,\perp V}$. Moreover, any element of $\mathcal{H}^{2,V}_{loc} $ 
 is strongly orthogonal to any element of $\mathcal{H}^{2,\perp V}_{loc}$.
\end{proposition}

\begin{remark}
All previous results extend when the filtration is indexed by $\mathbbm{R}_+$.
\end{remark}

We are going to introduce here 
our Backward Stochastic Differential Equation (BSDE) for which there is no need for
having a particular martingale of reference.
\\
\\
We will denote $\mathcal{P}ro$ 
the  $\sigma$-field  generated by progressively measurable processes 
defined on $[0,T]\times \Omega$.
\\
Given some $V\in\mathcal{V}^{c,+}$, we will indicate by $\mathcal{L}^2(dV\otimes d\mathbbm{P})$ (resp. $\mathcal{L}^0(dV\otimes d\mathbbm{P})$) the set of (up to indistinguishability)
 progressively measurable processes $\phi$ such that $\mathbbm{E}[\int_0^T \phi^2_rdV_r]<\infty$ (resp. $\int_0^T |\phi_r|dV_r<\infty$ $\mathbbm{P}$ a.s.) and $L^2(dV\otimes d\mathbbm{P})$ the quotient space of 
$\mathcal{L}^2(dV\otimes d\mathbbm{P})$ with respect to the subspace of processes equal to zero $dV\otimes d\mathbbm{P}$ a.e.
 More formally,  $L^2(dV\otimes d\mathbbm{P})$ corresponds to the classical $L^2$ space 
$L^2([0,T]\times\Omega,\mathcal{P}ro,dV\otimes d\mathbbm{P})$ and is therefore complete for its usual norm.
 \\
$\mathcal{L}^{2,cadlag}(dV\otimes d\mathbbm{P})$ (resp. $L^{2,cadlag}(dV\otimes d\mathbbm{P})$) will denote the subspace of $\mathcal{L}^{2}(dV\otimes d\mathbbm{P})$ (resp. $L^{2}(dV\otimes d\mathbbm{P})$) of cadlag elements (resp. of elements having a cadlag representative). We emphasize that  $L^{2,cadlag}(dV\otimes d\mathbbm{P})$ is not a closed subspace of $L^{2}(dV\otimes d\mathbbm{P})$.
\\
The application which associates to a process its corresponding 
class will be denoted $\phi\mapsto\dot{\phi}$.
\\
\\
The aforementioned BSDE will depend on a triple $(V,\xi,f)$ of coefficients:
$V$ is an integrator process, $\xi$ is the {\bf final} condition,
$f$ is the {\bf driver}.
% We will now consider some adapted non-decreasing process $V$, an $\mathcal{F}_T$-measurable random variable $\xi$ called the \textbf{final condition} and a \textbf{driver} $\hat{f}:\left([0,T]\times\Omega\right)\times\mathbbm{R}\times\mathbbm{R}\longrightarrow\mathbbm{R}$, measurable with respect to  $\mathcal{P}ro\otimes \mathcal{B}(\mathbbm{R})\otimes \mathcal{B}(\mathbbm{R})$.
% We will assume in all this section that $(\xi,\hat{f},V)$ verify the following hypothesis.
\begin{hypothesis}\label{HypBSDE}
\begin{enumerate}
	\item $V$ is bounded continuous non-decreasing adapted process;
    \item $\xi$ is a square integrable ${\mathcal F_T}$-measurable r.v.
\item $\hat{f}:\left([0,T]\times\Omega\right)\times\mathbbm{R}\times\mathbbm{R}\longrightarrow\mathbbm{R}$, measurable with respect to  $\mathcal{P}ro\otimes \mathcal{B}(\mathbbm{R})\otimes \mathcal{B}(\mathbbm{R})$.
    \item $\hat{f}(\cdot,\cdot,0,0)\in\mathcal{L}^2(dV\otimes d\mathbbm{P})$. 
    \item There exist positive constants $K^Y,K^Z$ such that, $\mathbbm{P}$ a.s. we have for all $t,y,y',z,z'$,
    \begin{equation}
        |\hat{f}(t,\cdot,y,z)-\hat{f}(t,\cdot,y',z')|\leq K^Y|y-y'|+K^Z|z-z'|.
    \end{equation}
\end{enumerate}
\end{hypothesis}

We start with a lemma.
\begin{lemma}\label{classdV}
Let $U_1$ and $U_2$ be in $\mathcal{L}^2(dV\otimes d\mathbbm{P})$ and such that $\dot{U}_1=\dot{U}_2$. Let $F:[0,T] \times
 \Omega \times \mathbbm{R}\longrightarrow\mathbbm{R}$ be such that $F(\cdot,\cdot,U_1)$ and  $F(\cdot,\cdot,U_2)$ are in $\mathcal{L}^0(dV\otimes d\mathbbm{P})$, then the processes  $\int_0^{\cdot} F(r,\cdot,U^1_r)dV_r$ and  $\int_0^{\cdot} F(r,\cdot,U^2_r)dV_r$  are indistinguishable. 
\end{lemma}
\begin{proof}
There exists a $\mathbbm{P}$-null set $\mathcal{N}$ such that for any $\omega\in\mathcal{N}^c$,  $U^1(\omega)=U^2(\omega)$ $dV(\omega)$ a.e. So for any $\omega\in \mathcal{N}^c$,  $F(\cdot,\omega,U^1(\omega))= F(\cdot,\omega,U^2(\omega))$ $dV(\omega)$ a.e. implying $\int_0^{\cdot} F(r,\omega,U^1_r(\omega))dV_r(\omega)= \int_0^{\cdot} F(r,\omega,U^2_r(\omega))dV_r(\omega)$.
So $\int_0^{\cdot} F(r,\cdot,U^1_r)dV_r$ and  $\int_0^{\cdot} F(r,\cdot,U^2_r)dV_r$  are indistinguishable processes.
\end{proof}
In some of the following proofs, we will have to work with classes 
of processes. 
According to Lemma \ref{classdV},  if $\dot{U}$ is an element of $L^2(dV\otimes d\mathbbm{P})$, 
%for any process  $U$ being a representative of $\dot{U}$
 the integral $\int_0^{\cdot} F(r,\omega,U_r)dV_r$  will not depend on the 
representantive process $U$  that we have chosen.
%%% OLD
%\begin{remark}\label{RclassdV}
%In some of the following proofs, we will have to work with classes 
%of processes. 
%%Previous lemma  permits us to remark the following. 
%According to Lemma \ref{classdV},  if $\dot{U}$ is an element of 
%\\
%$L^2(dV\otimes d\mathbbm{P})$ then  we could define the integral process 
% $\int_0^{\cdot} F(r,\omega,\dot{U}_r)dV_r$ 
%as  $\int_0^{\cdot} F(r,\omega,U_r)dV_r$, where $U$ is
%a representative $\dot{U}$.
%%is uniquely
%% (up to indistinguishability) defined in the sense that it does not
%% depend on the representative of $\dot{U}$.
%Nevertheless we will rarely use the dot notation in the integral.
%%So in all what follows if $\dot{U}$ only appears in integrals 
%%of previous type driven by $dV$, we can forget that it is a class
%% of processes and we will drop the dot above $U$.
%\end{remark}
\\
\\
We will now give the formulation of our BSDE.
\begin{definition}\label{firstdefBSDE}
We say that a couple
$(Y,M)\in \mathcal{L}^{2,cadlag}(dV\otimes d\mathbbm{P})\times \mathcal{H}^2_0$ is a
solution of $BSDE(\xi,\hat{f}, V)$ if it verifies
\begin{equation}\label{BSDEcadlag}
Y=\xi +\int_{\cdot}^T\hat{f}\left(r,\cdot,Y_r,\sqrt{\frac{d\langle M \rangle}{dV}}(r)\right)dV_r - (M_T-M_{\cdot})
\end{equation}
in the sense of indistinguishability. 
\end{definition}

%\begin{definition}\label{firstdefBSDE}
%We say that a couple
%$(\dot{Y},M)\in L^2(dV\otimes d\mathbbm{P})\times \mathcal{H}^2_0$ is a
%solution of $\overline{BSDE}(\xi,\hat{f}, V)$ if there exists a cadlag 
%representative $Y$ of $\dot{Y}$ which verifies
%\begin{equation}\label{BSDEcadlag}
%Y=\xi +\int_{\cdot}^T\hat{f}\left(r,\cdot,Y_r,\sqrt{\frac{d\langle M \rangle}{dV}}(r)\right)dV_r - (M_T-M_{\cdot})
%\end{equation}
%in the sense of indistinguishability. 
%\\
%A couple $(Y,M)\in\mathcal{L}^2(dV\otimes d\mathbbm{P})\times \mathcal{H}^2_0$ verifying \eqref{BSDEcadlag} will be said to be a solution of
%  $BSDE(\xi,\hat{f}, V)$. 
%\end{definition}
%
%\begin{proposition}\label{EquivalenceBSDE}
%The mapping $(Y,M)\mapsto(\dot{Y},M)$ induces a bijection between the set of solutions of $BSDE(\xi,\hat{f}, V)$ and the set of solutions of $\overline{BSDE}(\xi,\hat{f}, V)$.
%\end{proposition}
%\begin{proof}
%The set of solutions of $\overline{BSDE}(\xi,\hat{f}, V)$ is by Definition \ref{firstdefBSDE} the image by $(Y,M)\mapsto(\dot{Y},M)$ of the set of solutions of $BSDE(\xi,\hat{f}, V)$, so we only have to show injectivity.
%\\
%Let $(Y,M)$ and $(Y',M')$ be two solutions of $BSDE(\xi,\hat{f}, V)$ such that
% \\
% $(\dot Y,M)=(\dot Y',M')$. By Lemma \ref{classdV}, we have that $\int_0^{\cdot}\hat{f}\left(r,\cdot,Y_r,\sqrt{\frac{d\langle M \rangle}{dV}}(r)\right)dV_r$ and $\int_0^{\cdot}\hat{f}\left(r,\cdot,Y'_r,\sqrt{\frac{d\langle M \rangle}{dV}}(r)\right)dV_r$ are indistinguishable, so by \eqref{BSDEcadlag}, $Y$ and $Y'$ are indistinguishable. 
%\end{proof}

\begin{proposition}\label{BSDEexpectations}
If $(Y,M)$ solves $BSDE(\xi,\hat{f}, V)$, and if we denote
\\
 $\hat{f}\left(r,\cdot,Y_r,\sqrt{\frac{d\langle M \rangle}{dV}}(r)\right)$ by $\hat{f}_r$, then for any $t\in[0,T]$, a.s. we have 
\begin{equation} \label{E32bis}
\left\{\begin{array}{rcl}
Y_t &=& \mathbbm{E}\left[\xi+\int_t^T\hat{f}_rdV_r\middle|\mathcal{F}_t\right] \\
M_t &=& \mathbbm{E}\left[\xi+\int_0^T\hat{f}_rdV_r\middle|\mathcal{F}_t\right]-\mathbbm{E}\left[\xi+\int_0^T\hat{f}_rdV_r\middle|\mathcal{F}_0\right].
\end{array}\right.
\end{equation}
\end{proposition}
\begin{proof}
Since $Y_t=\xi +\int_t^T\hat{f}_rdV_r - (M_T-M_t)$ a.s.,   $Y$ being
an adapted process and $M$ a martingale, taking the expectation in \eqref{BSDEcadlag} at time $t$, we directly
 get  $Y_t = \mathbbm{E}\left[\xi+\int_t^T\hat{f}_rdV_r\middle|\mathcal{F}_t\right]$ and in particular that $Y_0 = \mathbbm{E}\left[\xi+\int_0^T\hat{f}_rdV_r\middle|\mathcal{F}_0\right]$. Since $M_0=0$, looking at the BSDE at time 0 we get
$M_T = \xi +\int_0^T\hat{f}_rdV_r - Y_0
= \xi +\int_0^T\hat{f}_rdV_r -\mathbbm{E}\left[\xi+\int_0^T\hat{f}_rdV_r\middle|\mathcal{F}_0\right]$.
Taking the expectation with respect to $\mathcal{F}_t$ in the above inequality
gives the second line of \eqref{E32bis}.

\end{proof}

We will proceed showing that $BSDE(\xi,\hat{f},V)$ has a unique solution. At this point we introduce a significant map
$\Phi$ which will map $L^2(dV\otimes d\mathbbm{P})\times \mathcal{H}^2_0$
into its subspace $L^{2,cadlag}(dV\otimes d\mathbbm{P})\times \mathcal{H}^2_0$.
From now on, until Notation \ref{contraction}, we fix a couple $(\dot{U},N)\in L^2(dV\otimes d\mathbbm{P})\times \mathcal{H}^2_0$ to which  we will associate $(\dot{Y},M)$ which, as we will show, will belong to  $L^{2,cadlag}(dV\otimes d\mathbbm{P})\times \mathcal{H}^2_0$. We will show that $(\dot{U},N)\mapsto (\dot{Y},M)$ is a contraction for a certain norm.
 In all the  proofs below, $\dot{U}$ will only appear in integrals 
driven by $dV$ through a representative $U$.
%, so we can consider that we
% are working with any element $U$ of the class $\dot{U}$.
% This will however not be the case for $\dot{Y}$ for which we will have to pick a specific 
%representative. 
%Our strategy consists in starting by defining through Definition \ref{defYM} a
%cadlag process $Y$, which will be said to be the {\it cadlag reference process},
%associated with $(\dot U, N)$. Then
%we define  $\dot{Y}$.

\begin{proposition}\label{L1}
For any $t\in[0,T]$, $\int_t^T\hat{f}^2\left(r,\cdot,U_r,\sqrt{\frac{d\langle N \rangle}{dV}}(r)\right)dV_r$ is in $L^1$ and $\left(\int_t^T\hat{f}\left(r,\cdot,U_r,\sqrt{\frac{d\langle N \rangle}{dV}}(r)\right)dV_r\right)$ is in $L^2$.
\end{proposition}

\begin{proof}
By Cauchy-Schwarz inequality and thanks to  the boundedness of $V$ together
the Lipschitz conditions on $f$ in 
Hypothesis \ref{HypBSDE},
 there exist a positive constant
 $C$ 
such that, for any $t\in[0,T]$, we have 
\\
\begin{equation}
\begin{array}{l}
\left(\int_t^T\hat{f}\left(r,\cdot,U_r,\sqrt{\frac{d\langle N \rangle}{dV}}(r)\right)dV_r\right)^2
      \leq V_T^2\int_t^T\hat{f}^2\left(r,\cdot,U_r,\sqrt{\frac{d\langle N \rangle}{dV}}(r)\right)dV_r\\
      \leq C\left(\int_t^T\hat{f}^2\left(r,\cdot,0,0\right)dV_r + \int_t^TU_r^2dV_r + \int_t^T\frac{d\langle N \rangle}{dV}(r)dV_r\right).
\end{array}
\end{equation}
The three terms on the right are in $L^1$.
Indeed, by Remark \ref{RemIncr} 
\\
$\int_t^T\frac{d\langle N \rangle}{dV}(r)dV_r \leq (\langle N \rangle_T-\langle N \rangle_t)$ which belongs to $L^1$
 since $N$ is taken in $\mathcal{H}^2$. By Hypothesis \ref{HypBSDE}, $f(\cdot,\cdot,0,0)$ is in ${\mathcal L}^2(dV\otimes d\mathbbm{P})$, and
 $\dot{U}$ was also taken in $L^2(dV\otimes d\mathbbm{P})$.
 This concludes the proof.
\end{proof}

We can therefore state
the following definition.

\begin{definition}\label{defYM} Setting 
$\hat{f}_r=\hat{f}\left(r,\cdot,U_r,\sqrt{\frac{d\langle N \rangle}{dV}}(r)\right)$, we define  $M$ 
as the cadlag version  of the martingale $t \mapsto  \mathbbm{E}\left[\xi + \int_0^T \hat{f}_rdV_r\middle|\mathcal{F}_t\right]-\mathbbm{E}\left[\xi + \int_0^T \hat{f}_rdV_r\middle|\mathcal{F}_0\right]$.
 
It admits a cadlag version taking into account
 Theorem 4 in Chapter IV of \cite{dellmeyerB}, since the stochastic basis fulfills the usual conditions.
We denote by $Y$ the cadlag process defined by 
$Y_t = \xi+\int_t^T\hat{f}\left(r,\cdot,U_r,\sqrt{\frac{d\langle N \rangle}{dV}}(r)\right)dV_r - (M_T - M_t)$.
This will be called the {\bf cadlag reference process} and we will
 often omit its
dependence to $(\dot U,N)$.
\end{definition}

According to previous definition,  it is not clear whether
 $Y$ is adapted, however, we have the almost sure equalities
\begin{equation}\label{Ytexpectation}
\begin{array}{rcl}
Y_t &=& \xi+\int_t^T\hat{f}_rdV_r - (M_T - M_t)\\
&=& \xi+\int_t^T\hat{f}_rdV_r - \left(\xi+\int_0^T\hat{f}_rdV_r -\mathbbm{E}\left[\xi + \int_0^T \hat{f}_rdV_r\middle|\mathcal{F}_t\right]\right)\\
&=& \mathbbm{E}\left[\xi + \int_0^T \hat{f}_rdV_r\middle|\mathcal{F}_t\right] -\int_0^t\hat{f}_rdV_r\\
&=& \mathbbm{E}\left[\xi + \int_t^T \hat{f}_rdV_r\middle|\mathcal{F}_t\right].
\end{array}
\end{equation}
Since $Y$ is cadlag and adapted,  
by Theorem 15 Chapter IV of \cite{dellmeyer75}, it is progressively measurable.

\begin{proposition}\label{supY}
$M$ belongs to $\mathcal{H}^2_0$ and  
$\underset{t\in[0,T]}{\text{sup }}|Y_t|\in L^2$.
\end{proposition}
\begin{proof}
$M$ is square integrable and vanishes at $0$ by Definition \ref{defYM} and Proposition \ref{L1}. A consequence of Definition \ref{defYM}, of Cauchy-Schwarz inequality and of the boundedness of $V$ is the existence of some $C,C'>0$ such that, a.s., 
\begin{equation}
\begin{array}{rcl}
\underset{t\in[0,T]}{\text{sup }}Y^2_t&\leq& C\left(\xi^2 + \underset{t\in[0,T]}{\text{sup }}\left(\int_t^T\hat{f}_rdV_r\right)^2+\underset{t\in[0,T]}{\text{sup }}(M_T-M_t)^2\right)\\
&\leq& C'\left(\xi^2 + \int_0^T\hat{f}^2_rdV_r+\underset{t\in[0,T]}{\text{sup }} M^2_t\right)
\end{array}
\end{equation}
which belongs to $L^1$ by Proposition \ref{L1} and
the fact that
 $\xi$ and $M$ are  square integrable.
\end{proof}

%\begin{proposition}\label{L2}
%$Y$ and $M$  are square integrable processes.
%\end{proposition}
%
%\begin{proof}
%We already know that $M$ is a square integrable martingale. 
%As we have seen in Proposition \ref{L1}, $\int_t^T\hat{f}\left(r,\cdot,U_r,\sqrt{\frac{d\langle N \rangle}{dV}}(r)\right)dV_r$ belongs to $L^2$ for any $t\in[0,T]$ and by Hypothesis \ref{HypBSDE}, $\xi\in L^2$. So by \eqref{Ytexpectation} and   Jensen's inequality for conditional expectation we have
%\begin{equation*}
%    \begin{array}{rcl}
%         \mathbbm{E}\left[Y_t^2\right] &=& \mathbbm{E}\left[\mathbbm{E}\left[\xi+\int_t^T\hat{f}\left(r,\cdot,U_r,\sqrt{\frac{d\langle N \rangle}{dV}}(r)\right)dV_r\middle|\mathcal{F}_t\right]^2\right]\\
%         & \leq & \mathbbm{E}\left[\mathbbm{E}\left[\left(\xi+\int_t^T\hat{f}\left(r,\cdot,U_r,\sqrt{\frac{d\langle N \rangle}{dV}}(r)\right)dV_r\right)^2\middle|\mathcal{F}_t\right]\right] \\
%         & = &\mathbbm{E}\left[\left(\xi+\int_t^T\hat{f}\left(r,\cdot,U_r,\sqrt{\frac{d\langle N \rangle}{dV}}(r)\right)dV_r\right)^2\right],
%    \end{array}
%\end{equation*}
%which is finite.
%\end{proof}
%
%

Since $Y$ is cadlag progressively measurable, $\underset{t\in[0,T]}{\text{sup }}|Y_t|\in L^2$ and since $V$ is bounded, it is clear that $Y\in \mathcal{L}^{2,cadlag}(dV\otimes d\mathbbm{P})$ and the corresponding class $\dot{Y}$ belongs to  
$L^{2,cadlag}(dV\otimes d\mathbbm{P})$.
% We recall that $M\in \mathcal{H}^2_0$ thanks to Proposition \ref{L2}.

\begin{notation}\label{contraction}
We denote by $\Phi$ the operator which associates to a couple 
$(\dot{U},N)$ the couple $(\dot{Y}, M)$.
\begin{equation*}
\Phi: \begin{array}{rcl}
L^2(dV\otimes d\mathbbm{P})\times \mathcal{H}^2_0 &\longrightarrow& L^{2,cadlag}(dV\otimes d\mathbbm{P})\times \mathcal{H}^2_0\\
(\dot{U},N) &\longmapsto& (\dot{Y},M).
\end{array}
\end{equation*}
\end{notation}

\begin{proposition}\label{FixedPoint}
The mapping $(Y,M)\longmapsto(\dot{Y},M)$ induces a bijection between the set of solutions of $BSDE(\xi,\hat{f},V)$ and the set of fixed points of $\Phi$.
\end{proposition}

\begin{proof}
First, let $(U,N)$ be a solution of $BSDE(\xi,\hat{f},V)$, let $(\dot{Y},M):=\Phi(\dot{U},N)$ and let $Y$ be the reference cadlag process associated to $U$ as in Definition \ref{defYM}. By this same definition,  $M$ is the cadlag version of 
\\
$t\mapsto \mathbbm{E}\left[\xi+\int_0^T\hat{f}\left(r,\cdot,U_r,\sqrt{\frac{d\langle N \rangle}{dV}}(r)\right)dV_r\middle|\mathcal{F}_t\right]-\mathbbm{E}\left[\xi+\int_0^T\hat{f}\left(r,\cdot,U_r,\sqrt{\frac{d\langle N \rangle}{dV}}(r)\right)dV_r\middle|\mathcal{F}_0\right]$, but by Proposition \ref{BSDEexpectations}, so is $N$, meaning $M=N$. Again by Definition \ref{defYM}, 
$Y =\xi + \int_{\cdot}^T \hat{f}\left(r,\cdot,U_r,\sqrt{\frac{d\langle N \rangle}{dV}}(r)\right)dV_r -(N_T-N_{\cdot})$ which is equal to $U$ thanks to \eqref{BSDEcadlag}, so $Y=U$
in the sense of indistinguishability, and in particular, $\dot{U}=\dot{Y}$, implying $(\dot{U},N)=(\dot{Y},M)=\Phi(\dot{U},N)$. The mapping $(Y,M)\longmapsto(\dot{Y},M)$ therefore does indeed map the set of solutions of $BSDE(\xi,\hat{f},V)$ into the set of fix points of $\Phi$.
\\
\\
The map is surjective. Indeed let $(\dot{U},N)$ be a fixed point of $\Phi$, the couple $(Y,M)$ of Definition \ref{defYM} verifies  
$Y =\xi + \int_{\cdot}^T \hat{f}\left(r,\cdot,U_r,\sqrt{\frac{d\langle N \rangle}{dV}}(r)\right)dV_r -(M_T-M_{\cdot})$
in the sense of indistinguishability, and $(\dot{Y},M)=\Phi(\dot{U},N)=(\dot{U},N)$, so by Lemma \ref{classdV}, $\int_{\cdot}^T \hat{f}\left(r,\cdot,Y_r,\sqrt{\frac{d\langle M \rangle}{dV}}(r)\right)dV_r$ and $\int_{\cdot}^T \hat{f}\left(r,\cdot,U_r,\sqrt{\frac{d\langle N \rangle}{dV}}(r)\right)dV_r$ are indistinguishable  and
$Y =\xi + \int_{\cdot}^T \hat{f}\left(r,\cdot,Y_r,\sqrt{\frac{d\langle M \rangle}{dV}}(r)\right)dV_r -(M_T-M_{\cdot})$, meaning that
$(Y,M)$ solves $BSDE(\xi,\hat{f},V)$.
\\
\\
We finally show that it is injective. Let us consider two solutions $(Y^1,M)$ and $(Y^2,M)$ of $BSDE(\xi,\hat{f},V)$ with  $\dot{Y^1}=\dot{Y^2}$. By Lemma \ref{classdV}, the processes $\int_{\cdot}^T\hat{f}\left(r,\cdot,Y^1_r,\sqrt{\frac{d\langle M \rangle}{dV}}(r)\right)dV_r$ and $\int_{\cdot}^T\hat{f}\left(r,\cdot,Y^2_r,\sqrt{\frac{d\langle M \rangle}{dV}}(r)\right)dV_r$ are indistinguishable, so taking \eqref{BSDEcadlag} into account, we have $Y^1=Y^2$.

\end{proof}

From now on, if $(\dot{Y},M)$ is the image by $\Phi$ of a couple 
\\
$(\dot{U},N)\in L^2(dV\otimes d\mathbbm{P})\times \mathcal{H}^2_0$,
 by default, we will always refer to the cadlag reference process 
$Y$ of $\dot{Y}$ defined in Definition \ref{defYM}.
 
\begin{lemma}\label{realmartLemma}
Let $Y$ be a cadlag adapted process satisfying 
$\mathbbm{E}\left[\underset{t\in[0,T]}{\text{sup }} Y_t^2\right]<\infty$ and $M$ be a square integrable martingale. Then there exists a constant $C>0$ such that for any $\epsilon >0$ we have 
$$\mathbbm{E}\left[\underset{t\in[0,T]}{\text{sup }}\left|\int_0^tY_{r^-}dM_r\right|\right]\leq  C\left( \frac{\epsilon}{2}\mathbbm{E}\left[\underset{t\in[0,T]}{\text{sup }} Y_t^2\right] + \frac{1}{2\epsilon}\mathbbm{E}\left[[ M]_T\right]\right).$$ In particular, $\int_0^{\cdot}Y_{r^-}dM_r$ is a uniformly integrable martingale.

\end{lemma}
\begin{proof}
By Burkholder-Davis-Gundy 
%(shortened by BDG)
 and Cauchy-Schwarz 
%% (CS) jamais utilise
  inequalities, there exists $C>0$ such that 
\begin{equation*}
\begin{array}{rcccl}
    &&\mathbbm{E}\left[\underset{t\in[0,T]}{\text{sup }}\left|\int_0^tY_{r^-}dM_r\right|\right] & \leq & C\mathbbm{E}\left[\sqrt{\int_0^TY^2_{r^-}d[M]_r}\right]\\
    &\leq & C\mathbbm{E}\left[\sqrt{\underset{t\in[0,T]}{\text{sup }}Y^2_t[M]_T}\right]
    &\leq & C\sqrt{\mathbbm{E}\left[\underset{t\in[0,T]}{\text{sup }}Y^2_t\right]\mathbbm{E}[[M]]_T}\\
     & \leq & C\left(\frac{\epsilon}{2}\mathbbm{E}\left[\underset{t\in[0,T]}{\text{sup }} Y_t^2\right] + \frac{1}{2\epsilon}\mathbbm{E}\left[[ M]_T\right]\right)
     & < & +\infty.
\end{array}
\end{equation*}
So $\int_0^{\cdot}Y_{r^-}dM_r$ is a uniformly integrable local martingale, and therefore a martingale.
\end{proof}

\begin{lemma}\label{LsupY}
	Let $Y$ be a cadlag adapted process and $M\in\mathcal{H}^2$. Assume the existence of a constant $C>0$ and an $L^1$-random variable $Z$ such that for any $t\in[0,T]$, $Y^2_t \leq C\left(Z+\left|\int_0^tY_{r^-}dM_r\right|\right)$.
	Then $\underset{t\in[0,T]}{\text{sup }}|Y_t|\in L^2$.
\end{lemma}
\begin{proof}
	For any stopping time $\tau$ we have 
	\begin{equation} \label{localsup1}
	\underset{t\in[0,\tau]}{\text{sup }} Y_t^2 \leq C\left(Z +\underset{t\in[0,\tau]}{\text{sup }}\left|\int_0^tY_{r^-}dM_r\right|\right).
	\end{equation}
	Since $Y_{t^-}$ is caglad and therefore locally bounded, (see Definition p164 in \cite{protter}) 
	we define $\tau_n = \text{inf }\{t>0: Y_{t^-}\geq n\}$. 
	It yields $\int_0^{\cdot\wedge\tau_n}Y_{r^-}dM_r$ is in $\mathcal{H}^2$ since its 
	angular bracket is equal to $\int_0^{\cdot\wedge\tau_n}Y^2_{r^-}d\langle M\rangle_r$ which is inferior to $n^2\langle M\rangle_T\in L^1$. By Doob's inequality we know that $\underset{t\in[0,\tau_n]}{\text{sup }}\left|\int_0^tY_{r^-}dM_r\right|$ is $L^2$ and  using
	\eqref{localsup1}, we get that  $\underset{t\in[0,\tau_n]}{\text{sup }} Y_t^2$  is $L^1$.
	By \eqref{localsup1} applied with $\tau_n$ and taking expectation, we get
	$\mathbbm{E}\left[\underset{t\in[0,\tau_n]}{\text{sup }} Y_t^2\right] \leq C'\left(1 +\mathbbm{E}\left[\underset{t\in[0,\tau_n]}{\text{sup }}\left|\int_0^tY_{r^-}dM_r\right|\right]\right)$,
	for some $C'$ which does not depend on $n$.
	By Lemma \ref{realmartLemma} applied to $(Y^{\tau_n},M)$ there exists $C''>0$ such that for any $n\in\mathbbm{N}^*$ and $\epsilon>0$, 
	\\
	$\mathbbm{E}\left[\underset{t\in[0,\tau_n]}{\text{sup }} Y_t^2\right] \leq C''\left(1 + \frac{\epsilon}{2}\mathbbm{E}\left[\underset{t\in[0,\tau_n]}{\text{sup }} Y_t^2\right] + \frac{1}{2\epsilon}\mathbbm{E}\left[[ M]_T\right]\right)$. Choosing $\epsilon = \frac{1}{C''}$, it follows that there exists $C_3>0$ such that for any $n>0$,
	\\
	$\frac{1}{2}\mathbbm{E}\left[\underset{t\in[0,\tau_n]}{\text{sup }} Y_t^2\right] \leq C_3\left(1 +  \mathbbm{E}\left[[ M]_T\right]\right)<\infty$.
	By monotone convergence theorem,
	taking the limit in $n$ we get the result.
\end{proof}

\begin{proposition}\label{realmart}
Let $\lambda\in\mathbbm{R}$, let $(\dot U,N)$, $(\dot U',N')$ be in 
\\
$L^2(dV\otimes d\mathbbm{P})\times \mathcal{H}^2_0$, let $(\dot Y,M)$, $(\dot Y',M')$ be their images by $\Phi$ and let $Y,Y'$ be the cadlag representatives of $\dot Y$, $\dot Y'$ introduced in Definition \ref{defYM}. Then $\int_0^{\cdot}e^{\lambda V_r}Y_{r^-}dM_r$, $\int_0^{\cdot}e^{\lambda V_r}Y'_{r^-}dM'_r$, $\int_0^{\cdot}e^{\lambda V_r}Y_{r^-}dM'_r$ and $\int_0^{\cdot}e^{\lambda V_r}Y'_{r^-}dM_r$ are martingales.
\end{proposition}

\begin{proof}
Thanks to Proposition \ref{supY} we know that $\underset{t\in[0,T]}{\text{sup }} |Y_t|$ and $\underset{t\in[0,T]}{\text{sup }} |Y'_t|$ are $L^2$. Moreover since $M$ and $M'$ are square integrable, the statement
 yields therefore as a consequence of  Lemma \ref{realmartLemma}
and the fact that $V$ is bounded.
\end{proof}

We will now show that $\Phi$ is a contraction for a certain norm.
 This will imply that it has a unique fixed point in $L^2(dV\otimes d\mathbbm{P})\times\mathcal{H}^2_0$ since this space is complete and therefore that  $BSDE(\xi,\hat{f},V)$ has a unique solution thanks to Proposition \ref{FixedPoint}.
\\
For any $\lambda>0$, 
%we define the following norm,
 on $L^2(dV\otimes d\mathbbm{P})\times\mathcal{H}^2_0$
we define the  norm
\\
$\|(\dot Y,M)\|_{\lambda}^2 :=\mathbbm{E}\left[\int_0^T e^{\lambda V_r}Y_r^2dV_r\right] + \mathbbm{E}\left[\int_0^T e^{\lambda V_r}d\langle M\rangle_r\right]$.
Since $V$ is bounded, these norms are all equivalent to the usual one 
of this space, which corresponds to $\lambda=0$.

\begin{proposition}\label{ProofContraction}
There exists  $\lambda>0$ such that for any 
\\
$(\dot U,N)\in L^2(dV\otimes d\mathbbm{P})\times\mathcal{H}^2_0$, $\left\|\Phi(\dot U,N)\right\|^2_{\lambda}\leq \frac{1}{2}\left\|(\dot U,N)\right\|^2_{\lambda}$. In particular, $\Phi$ is a contraction in 
$L^2(dV\otimes d\mathbbm{P})\times\mathcal{H}^2_0$ for the norm $\|\cdot\|_{\lambda}$. 
\end{proposition}

\begin{proof}
Let $(\dot U,N)$ and $(\dot U',N')$ be two couples of $L^2(dV\otimes d\mathbbm{P})\times\mathcal{H}^2_0$, let $(\dot Y,M)$ and $(\dot Y',M')$ be their images via $\Phi$ and let $Y,Y'$ be the cadlag reference process of $\dot Y$, $\dot Y'$ introduced in Definition \ref{defYM}. We will write $\bar{Y}$ for $Y-Y'$ and we adopt a similar notation  for  other processes. We will also write  
\\
$ \bar{f}_t := \hat{f}\left(t,\cdot,U_t,\sqrt{\frac{d\langle N \rangle}{dV}}(t)\right) -\hat{f}\left(t,\cdot,U'_t,\sqrt{\frac{d\langle N' \rangle}{dV}}(t)\right).$
\\
\\
By additivity,
% Definition \ref{defYM}, 
we have $d\bar Y_t=-\bar{f}_tdV_t +d\bar M_t$.
Since $\bar{Y}_T = \xi - \xi = 0$, applying the integration by parts formula to $\bar{Y}_t^2e^{\lambda V_t}$ between $0$ and $T$ we get
\begin{equation*}
\bar{Y}_0^2 - 2\int_0^T e^{\lambda V_r}\bar{Y}_r\bar{f}_rdV_r + 2\int_0^T e^{\lambda V_r}\bar{Y}_{r^-}d\bar{M}_r +\int_0^T e^{\lambda V_r}d[\bar{M}]_r + \lambda\int_0^T e^{\lambda V_r}\bar{Y}^2_rdV_r=0.
\end{equation*}

Since, by Proposition \ref{realmart}, the stochastic integral with respect
to $\bar M$ is a real martingale,
 by taking the expectations we get

\begin{equation*}
\mathbbm{E}\left[\bar{Y}_0^2\right] - 2 \mathbbm{E}\left[\int_0^T e^{\lambda V_r}\bar{Y}_r\bar{f}_rdV_r\right] + \mathbbm{E}\left[\int_0^T e^{\lambda V_r}d\langle \bar{M}\rangle_r\right] + \lambda\mathbbm{E}\left[\int_0^T e^{\lambda V_r}\bar{Y}^2_rdV_r\right]=0.
\end{equation*}
So by re-arranging and by using the Lipschitz condition on $f$ stated in Hypothesis
 \ref{HypBSDE}, we get 

\begin{equation*}
    \begin{array}{lll}
        & &\lambda \mathbbm{E}\left[\int_0^T e^{\lambda V_r}\bar{Y}^2_rdV_r\right] + \mathbbm{E}\left[\int_0^T e^{\lambda V_r}d\langle \bar{M}\rangle_r\right]\\
         & \leq & 2K^Y\mathbbm{E}\left[\int_0^T e^{\lambda V_r}|\bar{Y}_r||\bar{U}_r|dV_r\right]\\
         &&+2K^Z\mathbbm{E}\left[\int_0^T e^{\lambda V_r}|\bar{Y}_r|\left|\sqrt{\frac{d\langle N\rangle }{dV}}(r)-\sqrt{\frac{d\langle N'\rangle}{dV}}(r)\right|dV_r\right] \\
         &\leq& (K^Y\alpha + K^Z\beta)\mathbbm{E}\left[\int_0^T e^{\lambda V_r}|\bar{Y}_r|^2dV_r\right] + \frac{K^Y}{\alpha}\mathbbm{E}\left[\int_0^T e^{\lambda V_r}|\bar{U}_r|^2dV_r\right] \\
         &&+ \frac{K^Z}{\beta}\mathbbm{E}\left[\int_0^T e^{\lambda V_r}\left|\sqrt{\frac{d\langle N\rangle}{dV}}(r)-\sqrt{\frac{d\langle N'\rangle}{dV}}(r)\right|^2dV_r\right],
    \end{array}
\end{equation*}
for any positive $\alpha$ and $\beta$. Then we pick $\alpha = 2K^Y$ and $\beta = 2K^Z$, which gives us

\begin{equation*} 
    \begin{array}{rcl}
        &&\lambda \mathbbm{E}\left[\int_0^T e^{\lambda V_r}\bar{Y}^2_rdV_r\right] + \mathbbm{E}\left[\int_0^T e^{\lambda V_r}d\langle \bar{M}\rangle_r\right] \\
        &\leq& 2((K^Y)^2 + (K^Z)^2)\mathbbm{E}\left[\int_0^T e^{\lambda V_r}|\bar{Y}_r|^2dV_r\right] \\
        &+& \frac{1}{2}\mathbbm{E}\left[\int_0^T e^{\lambda V_r}|\bar{U}_r|^2dV_r\right] 
        + \frac{1}{2}\mathbbm{E}\left[\int_0^T e^{\lambda V_r}\middle|\sqrt{\frac{d\langle N\rangle}{dV}}(r)
        -\sqrt{\frac{d\langle N'\rangle}{dV}}(r)\middle|^2dV_r\right].
    \end{array}
\end{equation*}
We choose now $\lambda = 1 + 2((K^Y)^2+(K^Z)^2)$ and we get
\begin{equation} \label{E152}
\begin{array}{rcl}
&&\mathbbm{E}\left[\int_0^T e^{\lambda V_r}\bar{Y}^2_rdV_r\right] + \mathbbm{E}\left[\int_0^T e^{\lambda V_r}d\langle\bar{M}\rangle_r\right]\\
&\leq& \frac{1}{2}\mathbbm{E}\left[\int_0^T e^{\lambda V_r}|\bar{U}_r|^2dV_r\right]
+ \frac{1}{2}\mathbbm{E}\left[\int_0^T e^{\lambda V_r}\left|\sqrt{\frac{d\langle N\rangle}{dV}}(r)-\sqrt{\frac{d\langle N'\rangle}{dV}}(r)\right|^2dV_r\right].
\end{array}
\end{equation}
On the other hand, since by Proposition \ref{CS} we know that $\frac{d\langle N \rangle}{dV}\frac{d\langle N' \rangle}{dV} - \left(\frac{d\langle N,N' \rangle}{dV}\right)^2$ is a positive process, we have 
\begin{equation}\label{CSRadonNikodymDeriv}
    \begin{array}{rcl}
        \left|\sqrt{\frac{d\langle N\rangle}{dV}}-\sqrt{\frac{d\langle N'\rangle}{dV}}\right|^2 &=& \frac{d\langle N\rangle}{dV} - 2\sqrt{\frac{d\langle N\rangle}{dV}}\sqrt{\frac{d\langle N'\rangle}{dV}} +  \frac{d\langle N'\rangle}{dV}\\
         & \leq & \frac{d\langle N\rangle}{dV} - 2\frac{d\langle N, N'\rangle}{dV} +  \frac{d\langle N'\rangle}{dV}\\
         &=& \frac{d\langle \bar{N}\rangle}{dV} \text{  }dV\otimes d\mathbbm{P} \text{ a.e.}
    \end{array}
\end{equation}
Therefore, since by Remark \ref{RemIncr} we have $\int_0^{\cdot}e^{\lambda V_r}\frac{d\langle \bar{N}\rangle}{dV}(r)dV_r \leq \int_0^{\cdot} e^{\lambda V_r}d\langle \bar{N}\rangle_r$, then expression \eqref{E152} implies
\\
$\mathbbm{E}\left[\int_0^T e^{\lambda V_r}\bar{Y}^2_rdV_r+\int_0^T e^{\lambda V_r}d\langle \bar{M}\rangle_r\right] \leq \frac{1}{2}\mathbbm{E}\left[\int_0^T e^{\lambda V_r}|\bar{U}_r|^2dV_r+\int_0^T e^{\lambda V_r}d\langle \bar{N}\rangle_r\right]$,
which proves the contraction for the norm  $\|\cdot\|_{\lambda}$.
\end{proof}

\begin{theorem}\label{uniquenessBSDE}
If $(\xi,\hat{f})$ verifies Hypothesis \ref{HypBSDE} then $BSDE(\xi,\hat{f},V)$ has a unique solution.
\end{theorem}

\begin{proof}
The space $L^2(dV\otimes d\mathbbm{P})\times \mathcal{H}^2_0$ is complete and $\Phi$ defines on it a contraction for the norm $\|(\cdot,\cdot)\|_{\lambda}$ for some $\lambda>0$, so $\Phi$ has a unique fixed point in 
\\
$L^2(dV\otimes d\mathbbm{P})\times \mathcal{H}^2_0$. Then by Proposition \ref{FixedPoint}, $BSDE(\xi,\hat{f},V)$ has a unique solution. 
\end{proof}

\begin{remark}\label{RealMart}
Let $(Y,M)$ be the solution of $BSDE(\xi,\hat{f},V)$ and $\dot{Y}$ the class of $Y$ in $L^2(dV\otimes d\mathbbm{P})$. Thanks to Proposition \ref{FixedPoint}, we know  that 
\\
$(\dot{Y},M)=\Phi(\dot{Y},M)$ and therefore by Propositions \ref{supY} and \ref{realmart} that $\underset{t\in[0,T]}{\text{ sup }}|Y_t|$ is $L^2$ and that $\int_0^{\cdot}Y_{r^-}dM_r$ is a real martingale.
\end{remark}

\begin{remark}\label{BSDESmallInt} Let $(\xi,\hat{f},V)$ satisfying 
Hypothesis \ref{HypBSDE}.
Until now we have considered the related BSDE  on the interval $[0,T]$.
Without restriction of generality we can consider a
BSDE  on a restricted interval $[s,T]$ for some $s \in [0,T[$.
The results and comments of this section immediately
%% 
%The whole previous discussion and all the results expressed above
%trivially 
extend to this case.
%What has been done in this section on the interval $[0,T]$ could also have 
%been done on $[s,T]$. We can therefore note that
In particular there exists a unique couple of processes 
 $(Y^s,M^s)$, indexed by  $[s,T]$
 such that $Y^s$ is adapted, cadlag and verifies 
$\mathbbm{E}[\int_s^T(Y^s_r)^2dV_r]<\infty$, such that $M^s$ is a martingale
 starting at 0 in $s$ and such that
$Y^s_{\cdot} = \xi +\int_{\cdot}^T \hat{f}\left(r,\cdot,Y^s_r,\sqrt{\frac{d\langle M\rangle}{dV}}(r)\right)dV_r -(M^s_T-M^s_{\cdot})$
in the sense of indistinguishability on $[s,T]$. \\ 
 Moreover, if $(Y,M)$ 
denotes the solution of $BSDE(\xi,\hat{f},V)$ then $(Y,M_{\cdot}-M_s)$ and $(Y^s,M^s)$ coincide on $[s,T]$. This follows by the uniqueness argument for the restricted BSDE to $[s,T]$.
\end{remark}

The  lemma below  shows that, in order to verify that a couple
 $(Y,M)$ is the solution of $BSDE(\xi,\hat{f},V)$, it is not necessary 
to verify the square integrability of $Y$ since
it will be automatically fulfilled.
% to relax integrability
% conditions 
% when one wants to show that a couple $(Y,M)$ is the unique solution of $BSDE(\xi,\hat{f},V)$.
\begin{lemma}\label{LED+Pext}
Let $(\xi,\hat{f},V)$ verify Hypothesis \ref{HypBSDE} and consider $BSDE(\xi,\hat{f},V)$ defined in Definition \ref{firstdefBSDE}. Assume that there exists a cadlag adapted process $Y$ with $Y_0\in L^2$ , and $M\in\mathcal{H}^2_0$ such that 
\begin{equation}\label{EqLEDPext}
Y = \xi +\int_{\cdot}^T
 \hat{f}\left(r,\cdot,Y_r,\sqrt{\frac{d\langle M\rangle}{dV}}(r)\right) dV_r - (M_T-M_{\cdot}),
\end{equation}
in the sense of indistinguishability. Then
 $\underset{t\in[0,T]}{\text{sup }}|Y_t| \in L^2$. In particular, 
\\
$Y\in \mathcal{L}^2(dV\otimes d\mathbbm{P})$ and $(Y,M)$ is the unique solution of $BSDE(\xi,\hat{f},V)$ . 
\\
\\
On the other hand if $(Y,M)$ verifies \eqref{EqLEDPext} 
 on $[s,T]$ with $s<T$, if $Y_s\in L^2$,  $M_s=0$ and if we denote $(U,N)$ the unique solution of $BSDE(\xi,\hat{f},V)$, then $(Y,M)$ and $(U,N_{\cdot}-N_s)$ are indistinguishable on $[s,T]$.
\end{lemma}
\begin{proof}
%We denote $Z:=\sqrt{\frac{d\langle M\rangle}{dV}}$.
Let $\lambda>0$ and $t\in[0,T]$. By integration by parts formula applied to $Y^2e^{-\lambda V}$ between $0$ and $t$ we get
\begin{equation*}
\begin{array}{rcl}
Y^2_te^{-\lambda V_t}-Y_0^2 &=& -2\int_0^te^{-\lambda V_r}Y_r\hat{f}\left(r,\cdot,Y_r,\sqrt{\frac{d\langle M\rangle}{dV}}(r)\right)dV_r +2\int_0^te^{-\lambda V_r}Y_{r^-}dM_r \\
& & +\int_0^te^{-\lambda V_r}d[M]_r-\lambda\int_0^te^{-\lambda V_r}Y_r^2dM_r.
\end{array}
\end{equation*}

By re-arranging the terms  and using the Lipschitz conditions in Hypothesis \ref{HypBSDE}, we get
\begin{equation*} 
\begin{array}{rcl}
&& Y^2_te^{-\lambda V_t}+\lambda\int_0^te^{-\lambda V_r}Y_r^2dV_r\\
&\leq& Y_0^2 + 2\int_0^te^{-\lambda V_r}|Y_r||\hat{f}|\left(r,\cdot,Y_r,\sqrt{\frac{d\langle M\rangle}{dV}}(r)\right)dV_r+2\left|\int_0^te^{-\lambda V_r}Y_{r^-}dM_r\right|\\
&& +\int_0^te^{-\lambda V_r}d[M]_r\\
& \leq& Y_0^2 + \int_0^te^{-\lambda V_r}|\hat{f}|^2(r,\cdot,0,0)dV_r+(2K^Y+1+K^Z)\int_0^te^{-\lambda V_r}|Y_r|^2dV_r\\
& &+2\left|\int_0^te^{-\lambda V_r}Y_{r^-}dM_r\right| +\int_0^te^{-\lambda V_r}d[M]_r.
\end{array}
\end{equation*}
Choosing $\lambda = 2K^Y+1+K^Z$ this gives
\begin{equation*}
\begin{array}{rcl}
Y^2_te^{-\lambda V_t} &\leq& Y_0^2 + \int_0^te^{-\lambda V_r}|\hat{f}|^2(r,\cdot,0,0)dV_r+K^Z\int_0^te^{-\lambda V_r}\frac{d\langle M\rangle}{dV}(r)dV_r\\
&&+2\left|\int_0^te^{-\lambda V_r}Y_{r^-}dM_r\right| +\int_0^te^{-\lambda V_r}d[M]_r.
\end{array}
\end{equation*}
Since $V$ is bounded, there is a constant $C>0$, such that for any $t\in[0,T]$
\begin{equation*}
Y^2_t \leq C\left(Y_0^2 + \int_0^T|\hat{f}|^2(r,\cdot,0,0)dV_r+\int_0^T\frac{d\langle M\rangle}{dV}(r)dV_r +[M]_T+\left|\int_0^tY_{r^-}dM_r\right|\right).
\end{equation*}
By Hypothesis \ref{HypBSDE} and since we assumed $Y_0\in L^2$ and $M\in\mathcal{H}^2$, the first four terms on the right hand side are integrable and we can conclude by Lemma \ref{LsupY}. 
\\
\\
An analogous proof also holds on the interval $[s,T]$ 
 taking into account  Remark \ref{BSDESmallInt}.
\end{proof}

If the underlying filtration is Brownian and $V_t = t$,
 we can identify the solution of the BSDE with no
 driving martingale to the solution of a Brownian BSDE. 

Let $B$ be a $1$-dimensional Brownian motion defined on a complete probability 
space $(\Omega,\mathcal{F},\mathbbm{P})$. Let $T\in\mathbbm{R}_+^*$ and for any $t\in[0,T]$, let $\mathcal{F}^B_t$ denote the $\sigma$-field $\sigma(B_r|r\in[0,t])$ augmented with the $\mathbbm{P}$-negligible sets.
\\
In the stochastic basis $(\Omega,\mathcal{F},\mathcal{F}^B,\mathbbm{P})$, let $V_t=t$ and $(\xi,\hat{f})$ satisfy Hypothesis \ref{HypBSDE}.
Let $(Y,M)$ be the unique solution of $BSDE(\xi,\hat{f},V)$, see
 Theorem \ref{uniquenessBSDE}.
% $BSDE(\xi,\hat{f},V)$ has a unique solution $(Y,M)$ in $\mathcal{L}^2(dt\otimes d\mathbbm{P})\times \mathcal{H}^2_0$ which satisfies 
%\begin{equation}
%Y=\xi +\int_{\cdot}^T\hat{f}\left(r,\cdot,Y_r,\sqrt{\frac{d\langle M \rangle_r}{dr}}\right)dr - (M_T-M_{\cdot})
%\end{equation}
%in the sense of indistinguishability.
%Moreover, the Brownian BSDE 
%has a unique a unique solution $(U,Z)$ in $\mathcal{L}^2(dt\otimes d\mathbbm{P})\times L^2(dt\otimes d\mathbbm{P})$, see Theorem 1.2 in \cite{Pardoux}.

\begin{proposition} 
We have  $Y=U$, $M=\int_0^{\cdot}Z_rdB_r$, 
where $(U,Z)$ is the unique solution of the Brownian BSDE
\begin{equation}\label{BSDEBrownian}
U=\xi +\int_{\cdot}^T\hat{f}\left(r,\cdot,U_r,|Z_r|\right)dr - \int_{\cdot}^TZ_rdB_r.
\end{equation}
% in $\mathcal{L}^2(dt\otimes d\mathbbm{P})\times L^2(dt\otimes d\mathbbm{P})$,
\end{proposition}
\begin{proof}
By Theorem 1.2 in \cite{Pardoux},  \eqref{BSDEBrownian} admits a 
unique solution  $(U,Z)$ of progressively measurable processes 
such that $Z \in L^2(dt\otimes d\mathbbm{P})$. It is 
known that  $\underset{t\in[0,T]}{\text{sup }}|U_t|\in L^2$ and therefore that  $U \in \mathcal{L}^2(dt\otimes d\mathbbm{P})$, see Proposition 1.1 in \cite{Pardoux} for instance.
We define  
 $N=\int_0^{\cdot}Z_rdB_r$. The couple $(U,N)$ belongs to $\mathcal{L}^2(dt\otimes d\mathbbm{P})\times \mathcal{H}^2_0$. $N$ verifies $\frac{d\langle N \rangle_r}{dr}=Z^2_r$ $dt\otimes d\mathbbm{P}$ a.e. So by \eqref{BSDEBrownian}, the couple $(U,N)$ verifies $U=\xi +\int_{\cdot}^T\hat{f}\left(r,\cdot,U_r,\sqrt{\frac{d\langle N \rangle_r}{dr}}\right)dr - (N_T-N_{\cdot})$
in the sense of indistinguishability. It therefore solves $BSDE(\xi,\hat{f},V)$ and the assertion yields by uniqueness of the solution.
\end{proof}

\section{Martingale Problem and Markov classes}\label{SecMProcess}

In this section, we introduce the Markov process which will later 
be the forward process which will be coupled to a BSDE 
in order to constitute  Markovian BSDEs with no driving martingales. For  details about the exact mathematical background that we use to define our Markov process, one can consult the Section \ref{A1} of the
 Appendix. We also introduce the martingale problem related to this
 Markov process.
\\
\\
Let $E$ be a Polish space and $T\in\mathbbm{R}_+^*$ be  a fixed horizon.
From now on, $\left(\Omega,\mathcal{F},(X_t)_{t\in[0,T]},(\mathcal{F}_t)_{t\in[0,T]}\right)$ denotes the canonical space defined in Definition \ref{canonicalspace}.  We consider a  canonical Markov class $(\mathbbm{P}^{s,x})_{(s,x)\in[0,T]\times E}$ associated to a transition function measurable in time as defined in Definitions \ref{defMarkov} and \ref{DefFoncTrans}, and for any $(s,x)\in [0,T]\times E$, $\left(\Omega,\mathcal{F}^{s,x},(\mathcal{F}^{s,x}_t)_{t\in[0,T]},\mathbbm{P}^{s,x}\right)$ will denote the stochastic basis introduced in Definition \ref{CompletedBasis} and which fulfills the usual conditions.
%\begin{remark}
%All notions and results of this section extend to a time index equal to $\mathbbm{R}_+$.
%\end{remark}

The following notion of Martingale Problem comes from \cite{jacod79} Chapter XI.
\begin{definition}\label{firstMP}
Let $\chi$ be a family of stochastic processes defined on a filtered space $(\tilde{\Omega},\tilde{\mathcal{F}},(\tilde{\mathcal{F}}_t)_{t\in\mathbbm{T}})$. We say that a probability measure $\mathbbm{P}$ defined on $(\tilde{\Omega},\tilde{\mathcal{F}})$ solves the \textbf{martingale problem}
 associated to $\chi$ if under $\mathbbm{P}$ all elements of $\chi$ are in $\mathcal{M}_{loc}$.
We denote  $\mathcal{MP}(\chi)$ the set of probability measures solving this martingale problem. $\mathbbm{P}$ in $\mathcal{MP}(\chi)$ is said to be \textbf{extremal}  if there can not exist distinct probability measures $\mathbbm{Q},\mathbbm{Q}'$ in $\mathcal{MP}(\chi)$ and $\alpha\in]0,1[$ such that 
\\
$\mathbbm{P}=\alpha\mathbbm{Q}+(1-\alpha)\mathbbm{Q}'$. 
\end{definition}

We now  introduce a Martingale problem associated to an operator, following closely  the formalism of D.W. Stroock and S.R.S Varadhan in \cite{stroock}. We will see in Remark \ref{RPBM} that both Definitions \ref{firstMP} and \ref{MartingaleProblem} are closely related.
\begin{definition}\label{MartingaleProblem}
Let us consider  a domain 
$\mathcal{D}(a)\subset \mathcal{B}([0,T]\times E,\mathbbm{R})$ which is
a linear algebra; a linear operator 
$a:\mathcal{D}(a)\longrightarrow \mathcal{B}([0,T]\times E,\mathbbm{R})$ and
 a non-decreasing continuous function 
$V:[0,T]\rightarrow\mathbbm{R}_+$ starting at 0.
\\
We say that a set of probability measures $(\mathbbm{P}^{s,x})_{(s,x)\in [0,T]\times E}$ defined on $(\Omega,\mathcal{F})$ solves the {\bf martingale problem associated to}  $(\mathcal{D}(a),a,V)$ if, for  any 
\\
$(s,x)\in[0,T]\times E$, $\mathbbm{P}^{s,x}$ verifies
\begin{description}
%\label{MP}
\item{(a)} $\mathbbm{P}^{s,x}(\forall t\in[0,s], X_t=x)=1$;
\item{(b)} for every $\phi\in\mathcal{D}(a)$, $\left(t\longmapsto \phi(t,X_t)- \phi(s,x) - \int_s^t a(\phi)(r,X_r)dV_r \right)$, \\
$t \in [s,T]$, is  a cadlag $(\mathbbm{P}^{s,x},(\mathcal{F}_t)_{t\in[s,T]})$-local martingale.
\end{description}
We say that the Martingale Problem is \textbf{well-posed} if for any $(s,x)\in[0,T]\times E$, $\mathbbm{P}^{s,x}$ is the only probability measure satisfying 
those two properties.
% for all $\phi\in \mathcal{D}(a)$.
\end{definition}

\begin{remark}\label{RPBM}
In other words, $(\mathbbm{P}^{s,x})_{(s,x)\in [0,T]\times E}$ solves the  martingale problem associated to  $(\mathcal{D}(a),a,V)$ if and only if, for  any $(s,x)\in[0,T]\times E$, 
\\
$\mathbbm{P}^{s,x}\in \mathcal{MP}(\chi^{s,x})$ (see Definition \ref{firstMP}), where $\chi^{s,x}$ is the family of processes 
\\
$\left\{ t\mapsto\mathds{1}_{[s,T]}(t)\left(\phi(t,X_{t})-\phi(s,x)-\int_s^{t}a(\phi)(r,X_r)dV_r\right)\middle|\phi\in\mathcal{D}(a)\right\}$,
together with processes $\left\{t\mapsto\mathds{1}_{\{r\}}(t)(X_t-x)\middle|r\in[0,s]\right\}$.
%Indeed for some $r\in[0,s]$, $t\mapsto\mathds{1}_{\{r\}}(t)(X_t-x)$ is a cadlag local martingale iff $X_r=x$ a.s. so requiring that processes $t\mapsto\mathds{1}_{\{r\}}(t)(X_t-x)$ are cadlag local martingales for every $r\in[0,s]$ is equivalent to requiring that $X_r=x$ a.s. for every $r\in[0,s]$, and therefore since $X$ and $x$ are cadlag, it is equivalent to requiring item 
%%\eqref{MP}.
%a) in Definition \ref{MartingaleProblem}.
%\\
%$(\mathbbm{P}^{s,x})_{(s,x)\in [0,T]\times E}$ solves the well-posed martingale problem associated to  $(\mathcal{D}(a),a,V)$ if and only if, for  any $(s,x)\in[0,T]\times E$, $\mathbbm{P}^{s,x}$, $\mathcal{MP}(\chi^{s,x})=\{\mathbbm{P}^{s,x}\}$.
\end{remark} 

\begin{notation}\label{Mphi}
For every $(s,x)\in[0,T]\times E$ and $\phi\in\mathcal{D}(a)$, the process 
\\
$t\mapsto\mathds{1}_{[s,T]}(t)\left(\phi(t,X_{t})-\phi(s,x)-\int_s^{t}a(\phi)(r,X_r)dV_r\right)$ will be denoted $M[\phi]^{s,x}$.
\end{notation}
$M[\phi]^{s,x}$ is a cadlag $(\mathbbm{P}^{s,x},(\mathcal{F}_t)_{t\in[0,T]})$-local
 martingale which is equal to $0$ on $[0,s]$, and by Proposition \ref{ConditionalExp}, it is also a $(\mathbbm{P}^{s,x},(\mathcal{F}^{s,x}_t)_{t\in[0,T]})$-local martingale.
\\
\\
The following Hypothesis \ref{MPwellposed} is assumed for the rest of this section.
%%%  NOUVELLE HYPOTHESE
%% CITER LA NOTE SUR LES F. ADDITIVES
\begin{hypothesis}\label{MPwellposed}
The Markov class $(\mathbbm{P}^{s,x})_{(s,x)\in [0,T]\times E}$ solves a well-posed Martingale Problem associated to a triplet $(\mathcal{D}(a),a,V)$  in the sense 
of  Definition \ref{MartingaleProblem}.
\end{hypothesis}

The bilinear operator below was introduced (in the case of time-homogeneous operators) by J.P. Roth in potential analysis (see Chapter III in \cite{roth}),
 and popularized by P.A. Meyer 
 in the study of homogeneous Markov processes, see e.g. \cite{dellmeyerD} Chapter XV Comment 23 or \cite{jacod79} Remark 13.46. It has
finally become a fundamental tool in the study of Markov processes and semi-groups, 
 see for instance \cite{bakry}. It will be central in our work.
\begin{definition}
We set
\begin{equation} \label{D45}
\Gamma : \begin{array}{r c l}
    \mathcal{D}(a) \times \mathcal{D}(a)  & \rightarrow & \mathcal{B}([0,T]\times E) \\
    (\phi, \psi) & \mapsto & a(\phi\psi) - \phi a(\psi) - \psi a(\phi).
   \end{array}
\end{equation} 
The operator $\Gamma$ is 
called the \textbf{carr\'e du champs operator}.
\end{definition}

This operator will appear in the expression of the angular bracket of the local martingales
 that we have defined.

\begin{proposition}\label{bracketindomain}
For any $\phi\in \mathcal{D}(a)$ and $(s,x)\in[0,T]\times E$, $M[\phi]^{s,x}$ 
belongs to $\mathcal{H}^2_{0,loc}$.
Moreover, for any $(\phi, \psi)\in \mathcal{D}(a) \times \mathcal{D}(a)$ and $(s,x)\in[0,T]\times E,$ we have
\begin{equation*}
\langle M[\phi]^{s,x} , M[\psi]^{s,x} \rangle = \int_s^{\cdot} \Gamma(\phi,\psi)(r,X_r)dV_r,
\end{equation*}
on the interval $[s,T]$,
 in the stochastic basis $(\Omega,\mathcal{F}^{s,x},(\mathcal{F}^{s,x}_t)_{t\in[0,T]},\mathbbm{P}^{s,x}).$
\end{proposition}

\begin{proof}
We fix some $(s,x)\in[0,T]\times E$ and the associated probability $\mathbbm{P}^{s,x}$. For any $\phi,\psi$ in  $\mathcal{D}(a)$, by integration by parts 
 on $[s,T]$ we have
\begin{equation*}
\begin{array}{rcl}
&&M[\phi]^{s,x}M[\psi]^{s,x}\\
&=& \int_s^{\cdot}M[\phi]^{s,x}_{r^-}dM[\psi]^{s,x}_r +\int_s^{\cdot}M[\psi]^{s,x}_{r^-}dM[\phi]^{s,x}_r + [M[\phi]^{s,x},M[\psi]^{s,x}] \\
&= &\int_s^{\cdot}M[\phi]^{s,x}_{r^-}dM[\psi]^{s,x}_r +\int_s^{\cdot}M[\psi]^{s,x}_{r^-}dM[\phi]^{s,x}_r  + [\phi(\cdot,X_{\cdot}),\psi(\cdot,X_{\cdot})]  \\
&= &\int_s^{\cdot}M[\phi]^{s,x}_{r^-}dM[\psi]^{s,x}_r +\int_s^{\cdot}M[\psi]^{s,x}_{r^-}dM[\phi]^{s,x}_r+ \phi\psi(\cdot,X_{\cdot})\\
&&-\phi\psi(s,x)-\int_s^{\cdot}\phi(r^-,X_{r^-})d\psi(r,X_r)-\int_s^{\cdot}\psi(r^-,X_{r^-})d\phi(r,X_r). 
\end{array}
\end{equation*}

Since $\phi\psi$ belongs to $\mathcal{D}(a)$, we can use the decomposition 
of $\phi\psi(\cdot,X_{\cdot})$ given by (b) in Definition 
\ref{MartingaleProblem} 
 and
\begin{equation}\label{eqBracket}
\begin{array}{rcl}
&&M[\phi]^{s,x}M[\psi]^{s,x}\\
&=& \int_s^{\cdot}M[\phi]^{s,x}_{r^-}dM[\psi]^{s,x}_r +\int_s^{\cdot}M[\psi]^{s,x}_{r^-}dM[\phi]^{s,x}_r + \int_s^{\cdot}a(\phi\psi)(r,X_r)dV_r\\
&& + M^{s,x}[\phi\psi] -\int_s^{\cdot}\phi a(\psi)(r,X_r)dV_r - \int_s^{\cdot}\psi a(\phi)(r,X_r)dV_r \\
&&- \int_s^{\cdot}\phi(r^-,X_{r^-})dM^{s,x}[\psi]_r-\int_s^{\cdot}\psi(r^-,X_{r^-})dM^{s,x}[\phi]_r\\
 &=& \int_s^{\cdot}\Gamma(\phi,\psi)(r,X_r)dV_r + \int_s^{\cdot}M[\phi]^{s,x}_{r^-}dM[\psi]^{s,x}_r +\int_s^{\cdot}M[\psi]^{s,x}_{r^-}dM[\phi]^{s,x}_r \\
 &&+ M^{s,x}[\phi\psi]- \int_s^{\cdot}\phi(r^-,X_{r^-})dM^{s,x}[\psi]_r-\int_s^{\cdot}\psi(r^-,X_{r^-})dM^{s,x}[\phi]_r.
\end{array}
\end{equation}
Since $V$ is continuous, this implies that $M[\phi]^{s,x}M[\psi]^{s,x}$ is a special semi-martingale with bounded variation predictable part $\int_s^{\cdot}\Gamma(\phi,\psi)(r,X_r)dV_r$. In particular taking $\phi = \psi$, we have on $[s,T]$ that
$(M[\phi]^{s,x})^2=\int_s^{\cdot}\Gamma(\phi,\phi)(r,X_r)dV_r+N^{s,x}$,
where $N^{s,x}$ is some local martingale.
The first element in previous sum is  locally bounded
since it is a continuous process.  
The second one  is locally integrable as
 every  local martingale. Finally
  $\left(M[\phi]^{s,x}\right)^2$ is locally integrable, implying that
  $M[\phi]^{s,x}$ is in $\mathcal{H}^2_{0,loc}$. 
\\
\\
Let us come back to two given $\phi,\psi \in \mathcal{D}(a)$. Since we know that  $M[\phi]^{s,x}$, $M[\psi]^{s,x}$ belong to $\mathcal{H}^2_{0,loc}$ we can consider 
 $\langle M[\phi]^{s,x},M[\psi]^{s,x}\rangle$ which, by definition, is the unique predictable process with bounded variation such that 
\\
$M[\phi]^{s,x}M[\psi]^{s,x}-\langle M[\phi]^{s,x},M[\psi]^{s,x}\rangle$ is a local martingale. So necessarily, taking \eqref{eqBracket} into account,
$\langle M[\phi]^{s,x},M[\psi]^{s,x}\rangle=\int_s^{\cdot}\Gamma(\phi,\psi)(r,X_r)dV_r$.
\end{proof}
Taking $\phi = \psi$ in Proposition \ref{bracketindomain}, yields the following.
\begin{corollary}\label{H2Vloc}	
	For any $(s,x)\in[0,T]\times E$ and $\phi\in\mathcal{D}(a)$, $M[\phi]^{s,x}\in \mathcal{H}^{2,V}_{loc}$.
\end{corollary} 

% \begin{remark}\label{ortholoc}
% By Proposition \ref{OrthogonalSpaces}, it is clear that any element of  $\mathcal{H}^{2,\perp V}$ is strongly orthogonal to any element of $\mathcal{H}^{2,V}_{loc}$.
% \end{remark}

We now show that in our setup, $\mathcal{H}_0^2$ is always equal to $\mathcal{H}^{2,V}$. This can be seen as a generalization of Theorem 13.43 in \cite{jacod79}. 
%We  need for this a theorem proven by J. Jacod and M. Yor which states  (see e.g. Theorem 11.2 in \cite{jacod79}) the 
%following.
%\begin{theorem} \label{TC2}
%Let $\chi$ be a set of processes defined on some fixed filtered space $(\tilde{\Omega},\tilde{\mathcal{F}},(\tilde{\mathcal{F}}_t)_{t\in\mathbbm{T}})$.
%If $\mathbbm{P}\in\mathcal{MP}(\chi)$ then the following assertions are equivalent.
%
%\begin{enumerate}
%\item $\mathbbm{P}\in\mathcal{MP}_e(\chi)$;
%\item any $N\in\mathcal{H}^{\infty}_{0,loc}(\mathbbm{P})$  strongly orthogonal to all elements of $\chi$  is equal to zero, and $\tilde{\mathcal{F}}_0$ is $\mathbbm{P}$-trivial.
%\end{enumerate}
%\end{theorem}

\begin{proposition}\label{BoundOrthoMart}
Let $(s,x)\in[0,T]\times E$ and $\mathbbm{P}^{s,x}$ be fixed. If $N\in\mathcal{H}^{\infty}_{0,loc}$ is strongly orthogonal to $M[\phi]^{s,x}$ for all $\phi\in\mathcal{D}(a)$ then it  is necessarily equal to $0$.
\end{proposition}
\begin{proof}
In Hypothesis \ref{MPwellposed}, for any $(s,x)\in[0,T]\times E$ we have assumed that $\mathbbm{P}^{s,x}$ was the unique element of $\mathcal{MP}(\chi^{s,x})$, where $\chi^{s,x}$ was introduced in Remark \ref{RPBM}.
Therefore $\mathbbm{P}^{s,x}$ is extremal in $\mathcal{MP}(\chi^{s,x})$. So thanks to the Jacod-Yor Theorem (see e.g. Theorem 11.2 in \cite{jacod79}), we know that if an element $N$ of $\mathcal{H}^{\infty}_{0,loc}$ is strongly orthogonal to all the $M[\phi]^{s,x}$ then it is equal to zero. 
\end{proof}

\begin{proposition}\label{H2V=H2}
If Hypothesis \ref{MPwellposed} is verified then for any $(s,x)\in[0,T]\times E$, in the stochastic basis $\left(\Omega,\mathcal{F}^{s,x},(\mathcal{F}^{s,x}_t)_{t\in[0,T]},\mathbbm{P}^{s,x}\right)$, we have $\mathcal{H}_0^2=\mathcal{H}^{2,V}$.
\end{proposition}
\begin{proof}
We fix $(s,x)\in[0,T]\times E$. 
It is enough to show the inclusion 
$\mathcal{H}_0^2 \subset \mathcal{H}^{2,V}$.
 We start considering a bounded martingale $N\in\mathcal{H}^{\infty}_0$ and 
 showing that it belongs to $\mathcal{H}^{2,V}$. Since $N$ belongs to 
$\mathcal{H}^2_0$, we can consider the corresponding 
$N^V,N^{\perp V}$ in $\mathcal{H}^2_0$, appearing in the statement
of Proposition \ref{DecompoMart}.
% we can consider $N^V,N^{\perp V}$ in $\mathcal{H}^2_0$. 
We show below that $N^V$ and $N^{\perp V}$ are locally bounded, which will permit us to use  Jacod-Yor theorem on $N^{\perp V}$. 
\\
\\
Indeed, by Proposition \ref{DecompoMart}  there exists a predictable process $K$ such that 
\\
$N^V=\int_s^{\cdot}\mathds{1}_{\{K_r<1\}}dN_r$ and $N^{\perp V}=\int_s^{\cdot}\mathds{1}_{\{K_r=1\}}dN_r$. So if $N$ is bounded then it has bounded jumps; 
by previous way of characterizing $N^V$ and $N^{\perp V}$, their jumps can be expressed $(\Delta N^V)_t = \mathds{1}_{\{K_t<1\}}\Delta N_t$ and $(\Delta N^{\perp V})_t = \mathds{1}_{\{K_t=1\}}\Delta N_t$ (see Theorem 8 Chapter IV.3 in \cite{protter}), so they 
also have bounded jumps  which implies that they are locally bounded, see (2.4) in \cite{jacod79}.
\\ \\
So $N^{\perp V}$ is in $\mathcal{H}^{\infty}_{0,loc}$ and
by construction it belongs to $\mathcal{H}^{2,\perp V}$.
 Since 
by Corollary \ref{H2Vloc}, all the  $M[\phi]^{s,x}$ belong to
 $\mathcal{H}^{2,V}_{loc}$, then, by 
Proposition \ref{OrthogonalSpaces},
%Remark \ref{ortholoc},
  $N^{\perp V}$
is strongly orthogonal to all the $M[\phi]^{s,x}$.
Consequently, by Proposition \ref{BoundOrthoMart}, $N^{\perp V}$ is equal to zero. This shows that $N=N^V$ which by construction belongs to $\mathcal{H}^{2,V}$, and consequently that
 $\mathcal{H}^{\infty}_0\subset\mathcal{H}^{2,V}$,
which concludes the proof when $N$ is a bounded martingale.
\\
\\
 We can conclude by density arguments as follows.
Let $M\in\mathcal{H}^2_0$. 
For any integer $n\in\mathbbm{N}^*$, we denote by $M^{n}$
the martingale in $\mathcal{H}^{\infty}_0$
 defined as the cadlag version of  
$t\mapsto\mathbbm{E}^{s,x}[((-n)\vee M_{T}\wedge n)|\mathcal{F}_t]$.
Now 
  $\left(M_{T}^{n}-M_{T}\right)^2\underset{n\rightarrow \infty}{\longrightarrow}0$  a.s. and this sequence is bounded by $4M_{T}^2$ 
which is an integrable r.v. So by the dominated convergence theorem
 $\mathbbm{E}^{s,x}\left[\left(M_{T}^{n}-M_{T}\right)^2\right]\underset{n\rightarrow \infty}{\longrightarrow} 0$. Then by Doob's inequality, $\underset{t\in[0,T]}{\text{sup }} (M^{n}_t-M_t)\underset{n\rightarrow \infty}{\overset{L^2}{\longrightarrow}} 0$ meaning that $M^{n}\underset{n\rightarrow \infty}{\overset{\mathcal{H}^2}{\longrightarrow}}M$. Since $\mathcal{H}^{\infty}_0\subset\mathcal{H}^{2,V}$, then $M^{n}$ belongs to $\mathcal{H}^{2,V}$ for any $n\geq 0$. Moreover $\mathcal{H}^{2,V}$ is closed in $\mathcal{H}^{2}$, 
 since  by Proposition \ref{OrthogonalSpaces}, it is a sub-Hilbert space. Finally we have shown that  $M\in\mathcal{H}^{2,V}$.
\end{proof}

Since $V$ is continuous, it follows  in particular that every
 $(\mathbbm{P}^{s,x},(\mathcal{F}^{s,x}_t)_{t\in[0,T]})$-square integrable martingale has a continuous angular bracket. By localization, the same assertion holds for local square integrable martingales. 
\\
\\
We will now be interested in extending the domain $\mathcal{D}(a)$.
\\
\\
For any $(s,x)\in[0,T]\times E$ we define the positive bounded \textbf{potential measure} $U(s,x,\cdot)$ on $\left([0,T]\times E,\mathcal{B}([0,T])\otimes \mathcal{B}(E)\right)$ by 
\\
$U(s,x,\cdot):\begin{array}{rcl}
\mathcal{B}([0,T])\otimes \mathcal{B}(E)&\longrightarrow& [0,V_T]\\
A &\longmapsto& \mathbbm{E}^{s,x}\left[\int_s^{T} \mathds{1}_{\{(t,X_t)\in A\}}dV_t\right].
\end{array}$

\begin{definition}\label{zeropotential}
A Borel set $A\subset [0,T]\times E$ will be said to be
 {\bf of zero potential} if, for any $(s,x)\in[0,T]\times E$  we have  $U(s,x,A) = 0$.
\end{definition}
\begin{notation}\label{topo}
Let %$p\in\mathbbm{N}^*$.
$p > 0$. We introduce
\\
${\mathcal L}^p_{s,x} :={\mathcal L}^p(U(s,x,\cdot)) =\left\{ f\in \mathcal{B}([0,T]\times E,\mathbbm{R}):\, \mathbbm{E}^{s,x}\left[\int_s^{T} |f|^p(r,X_r)dV_r\right] < \infty\right\}$.
\\
That classical $\mathcal{L}^p$-space is equipped with the seminorm
\\
$\|\cdot\|_{p,s,x}:f\mapsto \left(\mathbbm{E}^{s,x}\left[\int_s^{T}|f(r,X_r)|^pdV_r\right]\right)^{\frac{1}{p}}$.
 We also introduce 
\\
${\mathcal L}^0_{s,x} :={\mathcal L}^0(U(s,x,\cdot)) = \left\{ f\in \mathcal{B}([0,T]\times E,\mathbbm{R}):\, \int_s^{T} |f|(r,X_r)dV_r < \infty\quad\mathbbm{P}^{s,x}\text{ a.s.}\right\}$.
\\
We then denote for any $p\in\mathbbm{N}$ 
\begin{equation}
\mathcal{L}^p_X =\underset{(s,x)\in[0,T]\times E}{\bigcap}{\mathcal L}^p_{s,x}.
\end{equation}
Let $\mathcal{N}$ be the linear sub-space of $\mathcal{B}([0,T]\times E,\mathbbm{R})$ containing all functions which are equal to 0, $U(s,x,\cdot)$ a.e. for every $(s,x)$.
\\
For any $p\in\mathbbm{N}$, we define  the quotient space $L^p_X = \mathcal{L}^p_X /\mathcal{N}$.
\\
If $p\in\mathbbm{N}^*$, $L^p_X$ can be equipped with the topology generated by the family of semi-norms $\left(\|\cdot\|_{p,s,x}\right)_{(s,x)\in[0,T]\times E}$ which makes it a separate locally convex topological vector space, 
see Theorem 5.76 in \cite{aliprantis}.
\end{notation}

\begin{proposition}\label{uniquenessupto}
Let $f$ and $g$ be in $\mathcal{L}^0_X$.
Then $f$ and $g$ are equal up to a set of zero potential if and only if for any $(s,x)\in[0,T]\times E$, the processes $\int_s^{\cdot}f(r,X_r)dV_r$ and $\int_s^{\cdot}g(r,X_r)dV_r$ are indistinguishable under $\mathbbm{P}^{s,x}$.
Of course in  this case $f$ and $g$ correspond to the same element of $L^0_X$.
\end{proposition}
\begin{proof}
Let  $\mathbbm{P}^{s,x}$ be fixed. 
Evaluating the total variation of $\int_s^{\cdot}(f-g)(r,X_r)dV_r$
 yields  that
 $\int_s^{\cdot}f(r,X_r)dV_r$ and $\int_s^{\cdot}g(r,X_r)dV_r$ are indistinguishable  if and only if 
 $\int_s^{T}|f-g|(r,X_r)dV_r=0$ a.s. Since that r.v. is non-negative,
 this is true if and only if 
$\mathbbm{E}^{s,x}\left[\int_s^{T} |f-g|(r,X_r)dV_r\right]=0$ and therefore if and only if $U(s,x,N)=0$, where
 $N$ is the Borel subset
of $[0,T] \times E$,
 defined by 
 $\left\{(t,y):\, f(t,y)\neq g(t,y)\right\}$.
This concludes the proof of Proposition \ref{uniquenessupto}.

\end{proof}

We can now define our notion of \textbf{extended generator}.
\begin{definition}\label{domainextended}
 We first define the \textbf{extended domain} $\mathcal{D}(\mathfrak{a})$ as the set functions $\phi\in\mathcal{B}([0,T]\times E,\mathbbm{R})$ for which there exists 
\\
$\psi\in\mathcal{B}([0,T]\times E,\mathbbm{R})$ such that under any $\mathbbm{P}^{s,x}$ the process
\begin{equation} \label{E45}
     \mathds{1}_{[s,T]}\left(\phi(\cdot,X_{\cdot}) - \phi(s,x) - \int_s^{\cdot}\psi(r,X_r)dV_r \right) 
\end{equation}
(which is not necessarily cadlag) has a cadlag modification in $\mathcal{H}^2_{0}$. 
\end{definition}

%Je propose de fusionner comme suit:

\begin{proposition}\label{uniquenesspsi}
	Let $\phi \in {\mathcal B}([0,T] \times E, {\mathbbm R}).$
%$\mathcal{D}(\mathfrak{a})$. 
% Then the function $\psi$ satisfying the above condition  is unique up to zero potential sets.
There is at most one (up to zero potential sets)  $\psi 
\in {\mathcal B}([0,T] \times E, {\mathbbm R})$ such that
under any $\mathbbm{P}^{s,x}$, the process defined in \eqref{E45} 
has a modification which belongs to ${\mathcal M}_{loc}$.
	\\
	If moreover $\phi\in\mathcal{D}(a)$, then $a(\phi)=\psi$ up to zero potential sets. In this case, according to Notation \ref{Mphi},
 for every $(s,x)\in[0,T]\times E$, 
  $M[\phi]^{s,x}$ is the $\mathbbm{P}^{s,x}$ cadlag modification
 in $\mathcal{H}^2_{0}$ of
	 $\mathds{1}_{[s,T]}\left(\phi(\cdot,X_{\cdot}) - \phi(s,x) - \int_s^{\cdot}\psi(r,X_r)dV_r \right)$.
\end{proposition}
\begin{proof}
	Let  $\psi^1$ and $\psi^2$ be two functions
%	verifying the condition imposed by Definition \ref{domainextended}.
%	We fix   
such that for any 
%$(s,x)$  and the related probability 
$\mathbbm{P}^{s,x}$, \\
	$\mathds{1}_{[s,T]}\left(\phi(\cdot,X_{\cdot}) - \phi(s,x) - 
	\int_s^{\cdot}\psi^i(r,X_r)dV_r \right), \quad i =1,2,$
	admits a cadlag modification which is a local martingale.
	% and
	%\\
	%$\mathds{1}_{[s,T]}\left(\phi(\cdot,X_{\cdot}) - \phi(s,x) - \int_s^{\cdot}\psi^2(r,X_r)dV_r \right),$
	% have both  $M^1,M^2$ as cadlag modifications,
	% which are  in  $\mathcal{H}^2_0$.
	Then, under a fixed $\mathbbm{P}^{s,x}$,
  $\phi(\cdot,X_{\cdot})$ has two cadlag modifications  which are therefore
	indistinguishable, and by uniqueness of the decomposition of special 
	semi-martingales,  $\int_s^{\cdot}\psi^1(r,X_r)dV_r$ and $\int_s^{\cdot}\psi^2(r,X_r)dV_r$ are indistinguishable on $[s,T]$. Since this is true under any $\mathbbm{P}^{s,x}$,
	the two functions are equal up to a zero-potential set because of Proposition \ref{uniquenessupto}.
	\\
	Concerning the second part of the statement,
let $\phi\in\mathcal{D}(a)\cap\mathcal{D}(\mathfrak{a})$.
        The result follows by Definition \ref{MartingaleProblem} and 
the uniqueness of the function $\phi$ established just before.
 % then under a fixed probability $\mathbbm{P}^{s,x}$, the process $\phi(\cdot,X_{\cdot})$ is a cadlag special semimartingale admitting on $[s,T]$ the decomposition given by the Martingale problem (see Definition \ref{MartingaleProblem}) and having a modification being another cadlag special semimartingale given by Definition \ref{domainextended}. The statement follows as above by uniqueness of the decomposition of a semimartingale.
\end{proof}

\begin{definition}\label{extended}
Let $\phi \in \mathcal{D}(\mathfrak{a})$ as in Definition
 \ref{domainextended}.
 We denote again  by $M[\phi]^{s,x}$, the unique
 cadlag version
 of the process \eqref{E45} in $\mathcal{H}^2_{0}$.
Taking Proposition \ref{uniquenessupto} into account, this will not
 generate  any ambiguity with respect
to Notation \ref{Mphi}.   
Proposition \ref{uniquenessupto}, also permits to define without ambiguity the operator  
% \\
\begin{equation*}
\mathfrak{a}:
\begin{array}{rcl}
\mathcal{D}(\mathfrak{a})&\longrightarrow& L^0_X\\
\phi &\longmapsto & \psi.
\end{array}
\end{equation*}
$\mathfrak{a}$ will be called the \textbf{extended generator}.
\end{definition}

We now want to extend the carr\'e du champs operator $\Gamma(\cdot,\cdot)$
 to $\mathcal{D}(\mathfrak{a})\times\mathcal{D}(\mathfrak{a})$.

\begin{proposition} \label{P321}
Let $\phi_1, \phi_2$ be in $\mathcal{D}(\mathfrak{a})$. 
There exists a (unique up to zero-potential sets) function in $\mathcal{B}([0,T]\times E,\mathbbm{R})$ which we will denote $\mathfrak{G}(\phi_1,\phi_2)$ such that under any $\mathbbm{P}^{s,x}$, 
$\langle M[\phi_1]^{s,x},M[\phi_2]^{s,x}\rangle=\int_s^{\cdot}\mathfrak{G}(\phi_1,\phi_2)(r,X_r)dV_r$ on $[s,T]$, up to indistinguishability.
\\
If moreover $\phi_1$ and $\phi_2$ belong to $\mathcal{D}(a)$, then $\Gamma(\phi_1,\phi_2)=\mathfrak{G}(\phi_1,\phi_2)$ up to zero potential sets.
\end{proposition} 
\begin{proof}
Let $\phi_1, \phi_2  \in \mathcal{D}(\mathfrak{a})$ according to 
Definition \ref{extended}. We take some representative of the classes
 $\mathfrak{a}(\phi_i)$ for $i=1,2$, still denoted by the same symbol
 and define the square integrable MAFs (see Definition \ref{DefAF})  $M[\phi_i]$ by
\begin{equation}
M[\phi_i]^t_u(\omega)=\left\{\begin{array}{l}
\phi_i(u,X_u(\omega))-\phi_i(t,X_t(\omega))-\int_t^u\mathfrak{a}(\phi_i)(r,X_r(\omega))dV_r\\
\quad\quad\text{ if }\int_t^u|\mathfrak{a}(\phi_i)|(r,X_r(\omega))dV_r<+\infty\\
0 \ \text{otherwise.}
\end{array}\right.
\end{equation}
Indeed, for every 
$(s,x)\in[0,T]\times E$,  $M[\phi_i]^{s,x}$
 is the cadlag version under $\mathbbm{P}^{s,x}$.
\\
The existence of   $\mathfrak{G}(\phi_1,\phi_2)$
 therefore derives from Proposition \ref{bracketMAFs}. 
By Proposition \ref{uniquenessupto} that
 function is determined up to a zero-potential set.
 \\
 The second statement holds because of Proposition \ref{bracketindomain}.
\end{proof}

%\begin{remark} \label{R423}
%$\mathfrak{G}$ therefore defines a bilinear operator from 
%$\mathcal{D}(\mathfrak{a})\times\mathcal{D}(\mathfrak{a})$ to $L^0_X$.
%\\
%If $\phi\in\mathcal{D}(a)\cap\mathcal{D}(\mathfrak{a})$, by Proposition \ref{bracketindomain} and the uniqueness stated in the above proposition,  $\Gamma(\phi,\phi)$ is in the class $\mathfrak{G}(\phi,\phi)$.
%% which
%% extends  $\Gamma$ in the following sense. If $\phi,\psi$ are in
%% $\mathcal{D}(a)$ and such that $M[\phi]^{s,x}$, $M[\psi]^{s,x}$ are square integrable for all $(s,x)\in[0,T]\times E$ (in particular if $\Gamma(\phi,\phi)$ and $\Gamma(\psi,\psi)$ belong to $\mathcal{L}^1_X$), then $\phi,\psi$ belong to $\mathcal{D}(\mathfrak{a})$. Since, by Proposition \ref{bracketindomain},
%%under $\mathbbm{P}^{s,x}$, on $[s,T]$, we have
%%\\
%%$\langle M[\phi]^{s,x},M[\psi]^{s,x}\rangle= \int_s^{\cdot}\Gamma(\phi,\psi)(r,X_r)dV_r$, then it is clear that $\Gamma(\phi,\psi)$ is in the class $\mathfrak{G}(\phi,\psi)$.
%\end{remark}

\begin{definition}\label{extendedgamma}
The bilinear operator $\mathfrak{G}:\mathcal{D}(\mathfrak{a})\times\mathcal{D}(\mathfrak{a})\longmapsto L^0_X$ will be called  the \textbf{extended carr\'e du champs operator}.
\end{definition}
According to Definition \ref{domainextended}, we do not have
necessarily  $\mathcal{D}(a)\subset\mathcal{D}(\mathfrak{a})$,
however we have the following.
\begin{corollary}\label{RExtendedClassical} 
If $\phi\in\mathcal{D}(a)$ and $\Gamma(\phi,\phi)\in\mathcal{L}^1_X$,
 then $\phi\in\mathcal{D}(\mathfrak{a})$ and $(a(\phi),\Gamma(\phi,\phi))=(\mathfrak{a}(\phi),\mathfrak{G}(\phi,\phi))$ up to zero potential sets.
\end{corollary}
\begin{proof} \
Given some $\phi\in\mathcal{D}(a)$, by Definition \ref{domainextended},
if for every  $(s,x)\in[0,T]\times E$, $M[\phi]^{s,x}$ is square integrable,
 then $\phi\in\mathcal{D}(\mathfrak{a})$. By Proposition \ref{bracketindomain},
for every  $(s,x)\in[0,T]\times E$
 $M[\phi]^{s,x}$ is a $\mathbbm{P}^{s,x}$  square integrable if
and only if $\Gamma(\phi,\phi)\in\mathcal{L}^1_X$. So 
the statement holds because of Propositions \ref{uniquenesspsi} and \ref{P321}.
\end{proof}

\section{Pseudo-PDEs and associated Markovian BSDEs with no driving martingale}\label{SecPDE}

In this section, we still consider $T\in\mathbbm{R}_+^*$, a Polish space $E$ and the associated canonical space $\left(\Omega,\mathcal{F},(X_t)_{t\in[0,T]},(\mathcal{F}_t)_{t\in[0,T]}\right)$, see Definition \ref{canonicalspace}. We also 
consider a  canonical Markov class $(\mathbbm{P}^{s,x})_{(s,x)\in[0,T]\times E}$ and assume the following for the rest of the Section.
\begin{hypothesis}\label{HypX}
The transition function of $(\mathbbm{P}^{s,x})_{(s,x)\in[0,T]\times E}$ is measurable in time (see Definitions \ref{defMarkov} and \ref{DefFoncTrans}) and $(\mathbbm{P}^{s,x})_{(s,x)\in[0,T]\times E}$ solves a well-posed martingale problem associated to a triplet $(\mathcal{D}(a),a,V)$, see Definition \ref{MartingaleProblem} and Hypothesis \ref{MPwellposed}.
\end{hypothesis}
We will investigate here  a specific type of BSDE with no driving martingale $BSDE(\xi,\hat{f},V)$ which we will call \textbf{of Markovian type}, 
or {\bf Markovian} BSDE, 
in the following sense. The process $V$  will be the (deterministic)   function $V$ introduced in
  Definition \ref{MartingaleProblem},  the final condition $\xi$ will only depend on the final value of the canonical process $X_T$ and the  randomness of the driver $\hat{f}$  at time $t$ will only appear via the value at time $t$ of the forward process $X$. 
Given a function 
\\
$f:[0,T] \times E \times {\mathbb R} \times {\mathbb R} \rightarrow {\mathbb R}$,
we will set ${\hat f}(t,\omega, y,z) = f(t,X_t(\omega), y, z)$
for $t \in [0,T], \omega \in \Omega, y, z \in {\mathbb R}$.
\\
\\
That BSDE will be connected with the deterministic problem below.
\begin{notation}
From now on, we fix some $g\in\mathcal{B}(E,\mathbbm{R})$ and 
\\
$f\in\mathcal{B}([0,T]\times E\times\mathbbm{R}\times\mathbbm{R},\mathbbm{R})$.
\end{notation}

\begin{definition}\label{MarkovPDE}
We will call \textbf{Pseudo-Partial Differential Equation} (in short Pseudo-PDE) the following equation with final condition:
\begin{equation}\label{PDE}
\left\{
\begin{array}{rccc}
 a(u)(t,x) + f\left(t,x,u(t,x),\sqrt{\Gamma(u,u)(t,x)}\right)&=&0& \text{ on } [0,T]\times E   \\
 u(T,\cdot)&=&g.& 
\end{array}\right.
\end{equation}
We will say that $u$ is a \textbf{classical solution} of the Pseudo-PDE if it belongs to $\mathcal{D}(a)$ and verifies \eqref{PDE}.
\end{definition}
\begin{notation}
Equation \eqref{PDE} will be denoted $Pseudo-PDE(f,g)$.
\end{notation}

%To be able to perform the connection between forward BSDEs and
%\\
%$Pseudo-PDE(f,g)$,
%we will  assume some generalized 
%moments conditions on $X$, and some growth conditions
% on the functions $(f,g)$. Those will be related to 
%two functions $\zeta,\eta \in\mathcal{B}(E,\mathbbm{R}_+)$.
%\begin{hypothesis}\label{HypMom}
%The Markov class will be said \textbf{to verify}  $H^{mom}(\zeta,\eta)$ if 
%\begin{enumerate}
%\item for any $(s,x)\in[0,T]\times E$, $\mathbbm{E}^{s,x}[\zeta^2(X_T)]$ is finite;
%\item for any $(s,x)\in[0,T]\times E$, $\mathbbm{E}^{s,x}\left[\int_0^T\eta^2(X_r)dV_r\right]$ is finite.
%\end{enumerate}
%If $(E,\|\cdot\|)$  is a separable Banach space, it is often useful
%to choose 
%\\
%$\zeta:x\mapsto\|x\|^p$, $\eta:x\mapsto\|x\|^q$ for some $p,q\in\mathbbm{R}_+$,
%see \cite{paper2}. 
%In that context we will write $H^{mom}(p,q)$ instead of  $H^{mom}(\zeta,\eta)$.
%\end{hypothesis}

For the rest of this section, we will also assume that $f,g$ verify the following.
\vfill \eject
\begin{hypothesis}\label{Hpq}
\begin{itemize}
\item $g(X_T)$ is $L^2$ under $\mathbb{P}^{s,x}$ for every $(s,x)\in[0,T]\times E$;
\item $t\longmapsto f(t,X_t,0,0)\in\mathcal{L}^2_X$;
\item there exist $K^Y,K^Z>0$ such that for all $(t,x,y,y',z,z')$,
\begin{equation}
|f(t,x,y,z)-f(t,x,y',z')|\leq K^Y|y-y'|+K^Z|z-z'|.
\end{equation}
\end{itemize}
\end{hypothesis}

%\begin{hypothesis}\label{Hpq}
%A couple of functions  $f\in\mathcal{B}([0,T]\times E\times\mathbbm{R}\times\mathbbm{R},\mathbbm{R})$ and $g\in\mathcal{B}(E,\mathbbm{R})$ will be said \textbf{to verify}   $H(\zeta,\eta)$ if
%there exist positive  constants $K^Y,K^Z,C,C'$ such that
%\begin{enumerate}
%\item $\forall x: \quad |g(x)|\leq C(1+\zeta(x))$,
%\item $\forall (t,x):\quad  |f(t,x,0,0)|\leq C'(1+\eta(x))$,
%\item $\forall (t,x,y,y',z,z'):\quad  |f(t,x,y,z)-f(t,x,y',z')|\leq K^Y|y-y'|+K^Z|z-z'|$.
%\end{enumerate}
%%If  $(E,\|\cdot\|)$  is a separable Banach space
%%and $\zeta:x\mapsto\|x\|^p$, $\eta:x\mapsto\|x\|^q$ for some
%% $p,q\in\mathbbm{R}_+$, we will write $H(p,q)$ instead of  $H(\zeta,\eta)$.
%\end{hypothesis}

With the equation $Pseudo-PDE(f,g)$, we will associate the family of BSDEs with no driving martingale indexed by $(s,x)\in[0,T]\times E$ and defined on the interval $[0,T]$ and in the stochastic basis $\left(\Omega,\mathcal{F}^{s,x},(\mathcal{F}^{s,x}_t)_{t\in[0,T]},\mathbbm{P}^{s,x}\right)$, given by
\begin{equation}\label{BSDE}
Y^{s,x}_t = g(X_T) + \int_t^T f\left(r,X_r,Y^{s,x}_r,\sqrt{\frac{d\langle M^{s,x}\rangle}{dV}}(r)\right)dV_r  -(M^{s,x}_T - M^{s,x}_t).
\end{equation}

\begin{notation} \label{N55bis}
Equation \eqref{BSDE} will be denoted $BSDE^{s,x}(f,g)$.
It corresponds to $BSDE(g(X_T), \hat f, V)$ with $\mathbbm P:= \mathbbm 
P^{s,x}$. 
\end{notation}

\begin{remark}\label{MarkovBSDEsol}.
\begin{enumerate}
\item If Hypothesis \ref{Hpq} is verified  then Hypothesis \ref{HypBSDE}
 is verified for \eqref{BSDE}.
 By Theorem \ref{uniquenessBSDE}, for any $(s,x)$,  $BSDE^{s,x}(f,g)$ has a unique solution, in the sense of Definition \ref{firstdefBSDE}.
\item 
Even if  the underlying process $X$ admits no (even generalized)  moments,  given 
  a couple  $(f,g)$ such that $f(\cdot,\cdot,0,0)$ and $g$ are bounded, the considerations of this section still apply. In particular  the connection   between the
 $BSDE^{s,x}(f,g)$ and the corresponding $Pseudo-PDE(f,g)$ still exists.
% In particular, if $E$ is finite
%dimensional then, the underlyng process $X$ is allowed not to have any moments.
\end{enumerate}
\end{remark}
%
%For the rest of this section, the positive functions $\zeta,\eta$ and
% the functions $(f,g)$ appearing in $Pseudo-PDE(f,g)$ will be  fixed,  and we will  assume that the Markov class verifies $H^{mom}(\zeta,\eta)$ and that
% $(f,g)$ verify $H(\zeta,\eta)$.
 
\begin{notation} \label{N55}
From now on, $(Y^{s,x}, M^{s,x})$
will always denote the (unique) solution of $BSDE^{s,x}(f,g)$.
\end{notation}
 
%
%\begin{remark}\label{deterministic}
%Let $(s,x)\in[0,T]\times E$ be fixed. We know (see Proposition \ref{Ftrivial}) that if $t<s$, $\mathcal{F}_t$ is $\mathbbm{P}^{s,x}$-trivial.  So since $Y^{s,x}$ and $M^{s,x}$ are adapted, they are deterministic on $[0,s[$. Moreover since $M^{s,x}$ belongs to $\mathcal{H}^2_0$, then it is equal to zero on $[0,s[$. 
%\\
%We will not be interested in the value of $Y^{s,x}$ before $s$. However, we will later show that $Y^{s,x}_s$ is also deterministic. 
%However, already at this stage,  it is interesting to realize that on $[0,s]$, $Y^{s,x}$ is almost surely equal to the solution of the deterministic integral equation
%\begin{equation*}
%Y^{s,x}_t = Y^{s,x}_s + \int_s^tf(r,x,Y^{s,x}_r,0)dV_r, \ t \in [0,s].
%\end{equation*}
%So,  on $[0,s]$,  $Y^{s,x}$ is almost surely to a deterministic function absolutely continuous with respect to $V$ and solving the ODE (parametrized by $x$)
%\begin{equation*}
%\frac{dY^{s,x}}{dV}(t) =  -f(t,x,Y^{s,x}_t,0)dV_t.
%\end{equation*}
%\end{remark}

The goal of our work is to understand if and how the solutions of equations  
 $BSDE^{s,x}(f,g)$ produce a solution of $Pseudo-PDE(f,g)$  and 
reciprocally.
\\
\\  
We will start  by showing that if $Pseudo-PDE(f,g)$ has a classical solution,
 then this one provides solutions to the associated $BSDE^{s,x}(f,g)$.
\begin{proposition}\label{classicimpliesBSDE}
Let $u$ be a classical solution of $Pseudo-PDE(f,g)$ verifying $\Gamma(u,u)\in\mathcal{L}^1_X$.
%
%
%We assume that there exists $u\in\mathcal{D}(a)$ such that
%\begin{equation}\label{EqclassicimpliesBSDE}
%\left\{
%\begin{array}{rccc}
% a(u)(s,x) + f(s,x,u(s,x),\sqrt{\Gamma(u,u)(s,x)})&=&0& \text{ on } [0,T]\times E   \\
% u(T,\cdot)&=&g,& 
%\end{array}\right.
%\end{equation} 
%and we also assume the existence of a positive $C>0$ such that 
%for every $(s,x)\in[0,T]\times E$, $\sqrt{\Gamma(u,u)(s,x)}\leq C(1+\eta(x))$.
Then, for any  $(s,x)\in [0,T]\times E$, if $M[u]^{s,x}$ is as in Notation \ref{Mphi},
% and $(Y^{s,x},M^{s,x})$  is the unique solution of $BSDE^{s,x}(f,g)$,
we have that $(u(\cdot,X_{\cdot}),M[u]^{s,x})$ and  $(Y^{s,x},M^{s,x}_{\cdot}-M^{s,x}_s)$  are ${\mathbb P}^{s,x}$-indistinguishable on $[s,T]$.

\end{proposition}

\begin{proof}
Let $(s,x)$ be fixed. Since $u\in\mathcal{D}(a)$, 
 the martingale problem in the sense of Definition
 \ref{MartingaleProblem} and  
\eqref{PDE} imply  that, on $[s,T]$, under $\mathbbm{P}^{s,x}$
\begin{equation*}
    \begin{array}{rcl}
     &&u(\cdot,X_{\cdot})\\
     &=& u(T,X_T)-\int_{\cdot}^T a(u)(r,X_r)dV_r -(M[u]^{s,x}_T-M[u]^{s,x}_{\cdot})\\
     &=& g(X_T)+\int_{\cdot}^T f\left(r,X_r,u(r,X_r),\sqrt{\Gamma(u,u)(r,X_r)}\right) -(M[u]^{s,x}_T-M[u]^{s,x}_{\cdot})\\
     &=&g(X_T) + \int_{\cdot}^T f\left(r,X_r,Y_r,\sqrt{\frac{d\langle M[u]^{s,x}\rangle}{dV}}(r)\right)dV_r  -(M[u]^{s,x}_T-M[u]^{s,x}_{\cdot}),
    \end{array}
\end{equation*}
where the latter equality comes from Proposition \ref{bracketindomain}.
Since $\Gamma(u,u)\in\mathcal{L}^1_X$
it follows that 
$\mathbbm{E}^{s,x}\left[\langle M[u]^{s,x}\rangle_T\right]=\mathbbm{E}^{s,x}\left[\int_s^T\Gamma(u,u)(r,X_r)dV_r\right]<\infty$.
This means that 
$M[u]^{s,x}\in\mathcal{H}^2_0$, so by Lemma \ref{LED+Pext} 
$(u(\cdot,X_{\cdot}),M[u]^{s,x}_{\cdot})$ and 
 $(Y^{s,x},M^{s,x}_{\cdot}-M^{s,x}_s)$ are indistinguishable on $[s,T]$.
\end{proof}

We will now adopt the converse point of view, 
 starting with the solutions of the equations $BSDE^{s,x}(f,g)$.
Below we will show  
 that there exist Borel functions $u$ and $v\geq 0$ such 
that for any $(s,x)\in[0,T]\times E$, for all $t\in[s,T]$, $Y^{s,x}_t=u(t,X_t)$ $\mathbbm{P}^{s,x}$-a.s., and $\frac{d\langle M^{s,x}\rangle}{dV}=v^2(\cdot,X_{\cdot})\quad dV\otimes d\mathbbm{P}^{s,x}$ a.e. on $[s,T]$. 
This will be the object of Theorem
\ref{Defuv}, whose an analogous formulation  exists in the Brownian framework, see e.g.
 Theorem 4.1 in \cite{el1997backward}. We start with a lemma.

\begin{lemma} \label{L41}  
%Let $\tilde{f}\in\mathcal{B}([0,T]\times E,\mathbbm{R})$ be such that for any 
%\\
%$(s,x)\in [0,T]\times E$, $\tilde{f}(\cdot,X_{\cdot})\mathds{1}_{[s,T]}$ belongs to $\mathcal{L}^2(dV\otimes d\mathbbm{P}^{s,x})$.
Let $\tilde{f}\in\mathcal{L}^2_X$.
Let, for any 
$(s,x)\in[0,T]\times E$, 
 $(\tilde Y^{s,x},\tilde M^{s,x})$ be the (unique by Theorem \ref{uniquenessBSDE} and Remark \ref{BSDESmallInt}) solution of 
\begin{equation*}
\tilde Y^{s,x}_t = g(X_T) + \int_t^T \tilde f\left(r,X_r\right)dV_r  -(\tilde M^{s,x}_T - \tilde M^{s,x}_t),\quad t\in[s,T],
\end{equation*}
in $\left(\Omega,\mathcal{F}^{s,x},(\mathcal{F}^{s,x}_t)_{t\in[0,T]},\mathbbm{P}^{s,x}\right)$.
  Then there exist two functions $u$ and $v\geq 0$ in $\mathcal{B}([0,T]\times E,\mathbbm{R})$ such that for any $(s,x)\in[0,T]\times E$
\begin{equation*}
\left\{\begin{array}{r}
     \forall t\in [s,T]: \tilde Y^{s,x}_t = u(t,X_t)\quad \mathbbm{P}^{s,x}\text{a.s.}  \\
     \frac{d\langle \tilde M^{s,x}\rangle}{dV}=v^2(\cdot,X_{\cdot})\quad dV\otimes d\mathbbm{P}^{s,x}\text{ a.e. on }[s,T].
\end{array}\right.
\end{equation*}
\end{lemma}
\begin{proof}
We set  $u:(s,x)\mapsto \mathbbm{E}^{s,x}\left[g(X_T) + \int_s^T \tilde{f}\left(r,X_r\right)dV_r\right]$ 
which is Borel by Proposition \ref{Borel} and Lemma \ref{LemmaBorel}. Therefore 
by \eqref{Markov3} in Remark \ref{Rfuturefiltration},
 for a fixed $t \in[s,T]$  
$\mathbb {P}^{s,x}$- a.s. we have
\begin{equation*}
\begin{array}{rcl}
    u(t,X_t) &=& \mathbbm{E}^{t,X_t}\left[g(X_T) + \int_t^T \tilde{f}\left(r,X_r\right)dV_r\right]\\
     &=& \mathbbm{E}^{s,x}\left[g(X_T) + \int_t^T \tilde{f}\left(r,X_r\right)dV_r\middle|\mathcal{F}_t\right]\\
     &=& \mathbbm{E}^{s,x}\left[\tilde Y^{s,x}_t+(\tilde M^{s,x}_T-\tilde M^{s,x}_t)|\mathcal{F}_t\right]\\
     &=& \tilde Y^{s,x}_t,
\end{array}
\end{equation*}
since $\tilde M^{s,x}$ is a martingale and $\tilde Y^{s,x}$ is adapted. Then  the square integrable AF (see Definition \ref{DefAF}) defined by 
\begin{equation}
M^t_{t'}(\omega)=\left\{\begin{array}{l}
u(t',X_{t'}) - u(t,X_{t}) + \int_{t}^{t'}\tilde{f}(r,X_r(\omega))dV_r\\
\quad\quad\text{ if }\int_{t}^{t'}|\tilde{f}|(r,X_r(\omega))dV_r<+\infty\\
0 \text{ otherwise}
\end{array}\right.
\end{equation}
is a MAF whose cadlag version under $\mathbbm{P}^{s,x}$
is  $\tilde M^{s,x}$.  
The existence of the function
$v$ follows setting $v = \sqrt{k}$ in Proposition \ref{bracketMAFs}.

\end{proof}

We now define the Picard iterations associated to the contraction defining the solution of the BSDE associated with  $BSDE^{s,x}(f,g)$.

\begin{notation}\label{DefPicard}
For a fixed $(s,x)\in [0,T]\times E$, $\Phi^{s,x}$ will denote the 
contraction $\Phi: L^2(dV\otimes d\mathbbm{P}^{s,x})\times\mathcal{H}^2_0$
introduced in  Notation \ref{contraction} with respect to 
the BSDE  associated with  $BSDE^{s,x}(f,g)$, see Notation \ref{N55}
% whose fixed point
% defines the solution of $BSDE^{s,x}(f,g)$, see Definition \ref{contraction}. 
In the sequel we will not distinguish between a couple
$(\dot Y,M)$ in $L^2(dV\otimes d\mathbbm{P}^{s,x})\times\mathcal{H}^2_0$ 
and $(Y,M)$, where $Y$ is the reference cadlag process of $\dot Y$,
according to Definition \ref{defYM}. We then convene the following.
\begin{enumerate}
\item $(Y^{0,s,x},M^{0,s,x}):=(0,0)$;
\item $\forall k\in\mathbbm{N}^*:(Y^{k,s,x},M^{k,s,x}):=\Phi^{s,x}(Y^{k-1,s,x},M^{k-1,s,x})$,
\end{enumerate}
meaning that for $k\in\mathbbm{N}^*$, $(Y^{k,s,x},M^{k,s,x})$ is the solution of the BSDE
\begin{equation}\label{defYk}
Y^{k,s,x} = g(X_T) + \int_{\cdot}^T f\left(r,X_r,Y^{k-1,s,x},\sqrt{\frac{d\langle M^{k-1,s,x}\rangle}{dV}}(r)\right)dV_r  -(M^{k,s,x}_T - M^{k,s,x}_{\cdot}).
\end{equation}
\end{notation}

The  processes $(Y^{k,s,x},M^{k,s,x})$
will be called the \textbf{Picard iterations} of  $BSDE^{s,x}(f,g)$

\begin{proposition} \label{P511}
For each $k\in\mathbbm{N}$, there exist functions $u_k$ and $v_k\geq 0$ in
 $\mathcal{B}([0,T]\times E,\mathbbm{R})$ such that for every $(s,x)\in[0,T]\times E$
\begin{equation}\label{defuk}
\left\{\begin{array}{l}
     \forall t\in [s,T]: Y^{k,s,x}_t = u_k(t,X_t) \quad \mathbbm{P}^{s,x}a.s.  \\
     \frac{d\langle M^{k,s,x}\rangle}{dV}=v^2_k(\cdot,X_{\cdot})\quad dV\otimes d\mathbbm{P}^{s,x}\text{ a.e. on }[s,T].
\end{array}\right.
\end{equation}
\end{proposition}
\begin{lemma}\label{ModifImpliesdV}
Let $(s,x)\in[0,T]\times E$ be fixed and let $\phi,\psi$ be two measurable processes. If $\phi$ and $\psi$ are $\mathbbm{P}^{s,x}$-modifications of each other, then they are equal $dV\otimes d\mathbbm{P}^{s,x}$ a.e.
\end{lemma}

\begin{proof}
Since for any $t\in[0,T]$, $\phi_t=\psi_t$ $\mathbbm{P}^{s,x}$ a.s.  we can write by Fubini's theorem
$\mathbbm{E}^{s,x}\left[\int_0^{T}\mathds{1}_{\phi_t\neq\psi_t}dV_t\right] = \int_0^{T}\mathbbm{P}^{s,x}(\phi_t\neq\psi_t)dV_t = 0$.
\end{proof} 
\begin{prooff}. of Proposition \ref{P511}.
\\
We proceed by induction on $k$.
It is clear that $(u_0,v_0)=(0,0)$ verifies the assertion for $k=0$. 
\\
Now let us assume that functions $u_{k-1}$, $v_{k-1}$ exist, for some integer 
$k \ge 1$, 
verifying \eqref{defuk} for $k$ replaced with $k-1$.
\\
We fix $(s,x)\in[0,T]\times E$. By Lemma \ref{ModifImpliesdV},
  $(Y^{k-1,s,x},Z^{k-1,s,x})=(u_{k-1},v_{k-1})(\cdot,X_{\cdot})$ $dV\otimes \mathbbm{P}^{s,x}$ a.e. on [s,T]. Therefore by \eqref{defYk}, on $[s,T]$
\\
$Y^{k,s,x} = g(X_T) + \int_{\cdot}^T f\left(r,X_r,u_{k-1}(r,X_r),v_{k-1}(r,X_r)\right)dV_r  -(M^{k,s,x}_T - M^{k,s,x}_{\cdot})$.
\\
 Since $\Phi^{s,x}$ maps $L^2(dV\otimes d\mathbbm{P}^{s,x})\times \mathcal{H}^2_0$ into itself (see Definition \ref{contraction}), obviously all 
the Picard iterations 
belong to  
$L^2(dV\otimes d\mathbbm{P}^{s,x})\times \mathcal{H}^2_0$. 
In particular,
 $Y^{k-1,s,x}$ and $\sqrt{\frac{d\langle M^{k-1,s,x}\rangle}{dV}}$ belong to
$\mathcal{L}^2(dV\otimes d\mathbbm{P}^{s,x})$. So, 
 by recurrence assumption 
on  $ u_{k-1}$ and  $ v_{k-1}$, it follows that
%$u_{k-1}(\cdot,X_{\cdot})\mathds{1}_{[s,T]}$ and 
%$v_{k-1}(\cdot,X_{\cdot})\mathds{1}_{[s,T]}$ belong to 
% $\mathcal{L}^2(dV\otimes d\mathbbm{P}^{s,x})$. 
 $u_{k-1}$ and $v_{k-1}$ belong to $\mathcal{L}^2_X$.
% Combining $H^{mom}(\zeta,\eta)$ and the growth condition of $f$ in $H(\zeta,\eta)$,  
% $f(\cdot,X_{\cdot},0,0)$ also 
%belongs to  
%$\mathcal{L}^2(dV\otimes d\mathbbm{P}^{s,x})$.
%$f(\cdot,\cdot,0,0)$ also belongs to $\mathcal{L}^2_X$.
% Therefore thanks to the Lipschitz conditions on $f$ assumed in $H(\zeta,\eta)$,
Therefore, using the assumptions $f$ in Hypothesis \ref{Hpq},  $f(\cdot,\cdot,u_{k-1},v_{k-1}) \in \mathcal{L}^2_X$.
% $f\left(\cdot,X_{\cdot},u_{k-1}(\cdot,X_{\cdot}),v_{k-1}(\cdot,X_{\cdot})\right)\mathds{1}_{[s,T]}$ is in 
%$\mathcal{L}^2(dV\otimes d\mathbbm{P}^{s,x})$.
The existence of $u_k$ and $v_k$ now comes from Lemma \ref{L41} applied to $\tilde f:=f(\cdot,\cdot,u_{k-1},v_{k-1})$. This establishes the induction step
 for a general $k$ and allows to conclude the proof.

\end{prooff}

%\begin{remark} \label{R512}
%For any $k\in\mathbbm{N}^*$ we have
%\begin{enumerate}
%\item $u_k\in\mathcal{D}(\mathfrak{a})$;
%\item $v^2_k = \mathfrak{G}(u_k,u_k)$;
%\item $\mathfrak{a}(u_k)= -f(\cdot,\cdot,u_{k-1},v_{k-1})$.
%\end{enumerate}
%Indeed for any $(s,x)\in [0,T]\times E$,  under $\mathbbm{P}^{s,x}$
% for $t\in[s,T]$, we have
%\\
%$u_k(t,X_t)-u_k(s,x)= -\int_s^tf(\cdot,\cdot,u_{k-1},v_{k-1})(r,X_r)dV_r +(M^{k,s,x}_t-M^{k,s,x}_s)$ a.s. and we have $  \frac{d\langle M^{k,s,x}\rangle}{dV} = v^2_k(\cdot,X_{\cdot})\text{  } dV\otimes d\mathbbm{P}^{s,x}$  a.e. on $[s,T]$.
%\\
%So from Definition \ref{extended}  $u_k\in\mathcal{D}(\mathfrak{a})$ and $\mathfrak{a}(u_k)= -f(\cdot,\cdot,u_{k-1},v_{k-1})$ and by Definition \ref{extendedgamma}, $v^2_k = \mathfrak{G}(u_k,u_k)$, which shows the statement. \\
%
%\end{remark}
%Remark \ref{R512} shows a first  link between the BSDE, 
%the martingale problem introduced in Hypothesis \ref{MartingaleProblem} and the
% Pseudo-PDE with extended operators.

Now we intend  to pass to the limit in $k$. For any $(s,x)\in[0,T]\times E$, we
 have seen in Proposition \ref{ProofContraction} that $\Phi^{s,x}$ is a contraction in $\left(L^2(dV\otimes d\mathbbm{P}^{s,x})\times\mathcal{H}^2_0,\|\cdot\|_{\lambda}\right)$ for some $\lambda>0$, so we know that the sequence $(Y^{k,s,x},M^{k,s,x})$ converges to $(Y^{s,x},M^{s,x})$ in this topology.
\\
The proposition below also shows an a.e. corresponding convergence,
adapting the techniques of
 Corollary 2.1 in \cite{el1997backward}.
%% PAS TROP IMPORTANT MAIS LA PROPOSITION QUI SUIT EST
%%% VRAIE POUR UNE BSDE GENERALE
\begin{proposition}\label{cvdt}
For any $(s,x)\in[0,T]\times E$, $Y^{k,s,x}\underset{k\rightarrow \infty}{\longrightarrow} Y^{s,x}\quad dV\otimes d\mathbbm{P}^{s,x}$ a.e.  and $\sqrt{\frac{d\langle M^{k,s,x}\rangle}{dV}}\underset{k\rightarrow \infty}{\longrightarrow}\sqrt{\frac{d\langle M^{s,x}\rangle}{dV}}\quad dV\otimes d\mathbbm{P}^{s,x}$ a.e.

%if we set $Z^{k,s,x}=\sqrt{\frac{d\langle M^{k,s,x}\rangle}{dV}}$ and $Z^{s,x}=\sqrt{\frac{d\langle M^{s,x}\rangle}{dV}}$ then $Z^{k,s,x}\longrightarrow Z^{s,x}$   $dV\otimes d\mathbbm{P}^{s,x}$ a.e.
\end{proposition}
\begin{proof}
We fix $(s,x)$ and the associated probability. In this proof, all superscripts
 $s,x$  are dropped.
We set $Z^{k}=\sqrt{\frac{d\langle M^{k}\rangle}{dV}}$ and $Z^{ }=\sqrt{\frac{d\langle M^{ }\rangle}{dV}}$.
 By Proposition \ref{ProofContraction},  
there exists $\lambda>0$ such that for any $k\in\mathbbm{N}^*$
\begin{equation*}
\begin{array}{rl}
&\mathbbm{E}\left[\int_0^Te^{-\lambda V_r}|Y^{k+1}_r-Y^{k}_r|^2dV_r + \int_0^Te^{-\lambda V_r}d\langle M^{k+1}-M^{k}\rangle_r\right]\\
\leq &\frac{1}{2}\mathbbm{E}\left[\int_0^Te^{-\lambda V_r}|Y^{k}_r-Y^{k-1}_r|^2dV_r + \int_0^Te^{-\lambda V_r}d\langle M^{k}-M^{k-1}\rangle_r\right],
\end{array}
\end{equation*}
therefore
\begin{equation} \label{E514}
\begin{array}{rl}
 &\underset{k\geq 0}{\sum}\mathbbm{E}\left[\int_0^Te^{-\lambda V_r}|Y^{k+1}_r-Y^{k}_r|^2dV_r\right] + \mathbbm{E}\left[\int_0^Te^{-\lambda V_r}d\langle M^{k+1}-M^{k}\rangle_r\right]\\
 \leq& \underset{k\geq 0}{\sum}\frac{1}{2^k}\left(\mathbbm{E}\left[\int_0^Te^{-\lambda V_r}|Y^{1 }_r|^2dV_r\right] + \mathbbm{E}\left[\int_0^Te^{-\lambda V_r}d\langle M^{1}\rangle_r\right]\right)\\
 <& \infty.
\end{array}
\end{equation}
Thanks to \eqref{CSRadonNikodymDeriv} and \eqref{E514} we have 
\\
$\sum_{k\geq 0}\left(\mathbbm{E}\left[\int_0^Te^{-\lambda V_r}|Y^{k+1}_r-Y^{k}_r|^2dV_r\right] + \mathbbm{E}\left[\int_0^Te^{-\lambda V_r}|Z^{k+1}_r-Z^{k}_r|^2dV_r\right]\right)<\infty$.
So by Fubini's theorem we have 
\\
$\mathbbm{E}\left[\int_0^Te^{-\lambda V_r}\left(\sum_{k\geq 0}(|Y^{k+1}_r-Y^{k}_r|^2+|Z^{k+1}_r-Z^{k}_r|^2)\right)dV_r\right]<\infty$.
\\
\\
Consequently  the sum 
$\sum_{k\geq 0}\left(|Y^{k+1}_r(\omega)-Y^{k}_r(\omega)|^2 + |Z_r^{k+1}(\omega)-Z^{k}_r(\omega)|^2\right)$
is finite on a set of full $dV\otimes d\mathbbm{P}$ measure. So on this set of full measure, the sequence $(Y^{k+1}_t(\omega),Z^{k+1}_t(\omega))$ converges, and the limit is necessarily equal to $(Y^{ }_t(\omega),Z^{ }_t(\omega))$ $dV\otimes d\mathbbm{P}$ a.e. because of the $L^2(dV\otimes d\mathbbm{P})$ convergence 
that we have mentioned in the lines before the statement of the
present Proposition \ref{cvdt}. 

\end{proof}

\begin{theorem}\label{Defuv}
There exist two functions $u$ and $v\geq 0$ in 
\\
$\mathcal{B}([0,T]\times E,\mathbbm{R})$ such that for every $(s,x)\in[0,T]\times E$,
%$\mathbbm{P}^{s,x}$
\begin{equation} \label{E418}
\left\{\begin{array}{l}
     \forall t\in[s,T]: Y^{s,x}_t = u(t,X_t)  \quad \mathbbm{P}^{s,x} \text{ a.s.} \\
     \frac{d\langle M^{s,x}\rangle}{dV}=v^2(\cdot,X_{\cdot})\quad 
dV\otimes d  \mathbbm{P}^{s,x} \text{ a.e. on }[s,T].
\end{array}\right.
\end{equation}
\end{theorem}
\begin{proof}
We  set $\bar{u}:=\underset{k\in\mathbbm{N}}{\text{limsup }}u_k$,
 in the sense that for any $(s,x)\in[0,T]\times E$, 
\\
${\bar u}(s,x)= \underset{k\in\mathbbm{N}}{\text{limsup }}u_k(s,x)$ and $v:=\underset{k\in\mathbbm{N}}{\text{limsup }}v_k$. $\bar{u}$ and $v$ are Borel functions. We know by Propositions \ref{P511}, \ref{cvdt} and Lemma \ref{ModifImpliesdV} that for every $(s,x)\in[0,T]\times E$
\begin{equation*}
\left\{\begin{array}{rcl}
     u_k(\cdot,X_{\cdot})&\underset{k\rightarrow\infty}{\longrightarrow}& Y^{s,x}  \quad dV\otimes d\mathbbm{P}^{s,x}\text{ a.e. on }[s,T]\\
     v_k(\cdot,X_{\cdot})&\underset{k\rightarrow\infty}{\longrightarrow}& Z^{s,x}  \quad dV\otimes d\mathbbm{P}^{s,x}\text{ a.e. on }[s,T],
\end{array}\right.
\end{equation*}
where $Z^{s,x} := \sqrt{\frac{d\langle M^{s,x}\rangle}{dV}}$.
Therefore, for some fixed $(s,x)\in[0,T]\times E$ and on the set of full $dV\otimes d\mathbbm{P}^{s,x}$ measure on which these convergences hold we have 
\begin{equation}\label{E420}
\left\{\begin{array}{r}
     \bar{u}(t,X_t(\omega))=\underset{k\in\mathbbm{N}}{\text{limsup }}u_k(t,X_t(\omega))=\underset{k\in\mathbbm{N}}{\text{lim }}u_k(t,X_t(\omega)) = Y^{s,x}_t(\omega)  \\
     v(t,X_t(\omega))=\underset{k\in\mathbbm{N}}{\text{limsup }}v_k(t,X_t(\omega))=\underset{k\in\mathbbm{N}}{\text{lim  }}v_k(t,X_t(\omega)) = Z^{s,x}_t(\omega).
\end{array}\right.
\end{equation}
This shows in particular the existence of $v$ and
 the validity of the second line of \eqref{E418}. \\

It remains to show the existence of $u$ so that 
the first line of \eqref{E418} holds.
% but we would like to go further with the $Y$ part.
 Thanks to the $dV\otimes d\mathbbm{P}^{s,x}$ equalities concerning $v$ and $\bar{u}$ stated in 
\eqref{E420}, under every $\mathbbm{P}^{s,x}$ we actually have
\begin{equation} \label{E421}
Y^{s,x} = g(X_T) + \int_{\cdot}^T f\left(r,X_r,\bar{u}(r,X_r),v(r,X_r)\right)dV_r  -(M^{s,x}_T - M^{s,x}_{\cdot}).
\end{equation}
Now \eqref{E421} can be considered as a BSDE where the driver does
not depend on $y$ and $z$.
For any $(s,x)\in[0,T]\times E$, $Y^{s,x}$ and $Z^{s,x}$ belong to $\mathcal{L}^2(dV\otimes d\mathbbm{P}^{s,x})$, then by \eqref{E420}, so do $\bar{u}(\cdot,X_{\cdot})\mathds{1}_{[s,T]}$ and $v(\cdot,X_{\cdot})\mathds{1}_{[s,T]}$, meaning that $\bar{u}$ and $v$ belong to $\mathcal{L}^2_X$.  Using the two assumptions made on $f$ in Hypothesis \ref{Hpq},   $f(\cdot,\cdot,\bar{u},v)$ also belongs to $\mathcal{L}^2_X$. We can therefore apply Lemma \ref{L41} to 
$\tilde{f}=f(\cdot,\cdot,\bar{u},v)$, and conclude to the existence of a Borel function $u$ such that for every $(s,x)\in[0,T]\times E$, 
 $Y^{s,x}$ is on $[s,T]$ a $\mathbbm{P}^{s,x}$-version of $u(\cdot,X_{\cdot})$. 

\end{proof}

%\begin{remark}
%In particular, $Y^{s,x}_s=u(s,x)$ is deterministic and 
%\\
%$M^{s,x}_s = Y^{s,x}_s -Y^{s,x}_0 +\int_0^sf\left(r,X_r,Y^{s,x}_r,0\right)dV_r$ is also
% deterministic and it  is therefore equal to 0 since $M^{s,x}\in\mathcal{H}^2_0$,
%by Remark \ref{deterministic}.
%\end{remark}
%
\begin{remark}
Since $\bar{u}(\cdot,X_{\cdot})=Y^{s,x}=u(\cdot,X_{\cdot})$ $dV\otimes d\mathbbm{P}^{s,x}$ a.e. for every $(s,x)\in[0,T]\times E$, one can remark that $u=\bar{u}$ up to a zero potential set, and in particular that $u\in\mathcal{L}^2_X$ since $\bar{u}$ does.
\\
Moreover, for any $(s,x)\in[0,T]\times E$, the stochastic convergence 
\\
$(Y^{k,s,x},M^{k,s,x})\underset{k\rightarrow\infty}{\xrightarrow{L^2(dV\otimes d\mathbbm{P}^{s,x})\times \mathcal{H}^2}}(Y^{s,x},M^{s,x})$ now has the 
functional counterpart 
$\left\{
\begin{array}{rcl}
u_k&\xrightarrow{\|\cdot\|_{2,s,x}}&u\\
v_k&\xrightarrow{\|\cdot\|_{2,s,x}}&v,
\end{array}\right.$
which yields $\left\{
\begin{array}{rcl}
u_k&\overset{L^2_X}{\longrightarrow}&u\\
v_k&\overset{L^2_X}{\longrightarrow}&v,
\end{array}\right.$
where we recall that the locally convex topological space $L^2_X$
was 
introduced in Notation \ref{topo}.
\end{remark}

\begin{corollary}\label{uvBSDE}
 For any $(s,x)\in[0,T]\times E$  and for any $t\in[s,T]$, the couple of functions $(u,v)$ obtained in Theorem \ref{Defuv} verifies $\mathbbm{P}^{s,x}$ a.s.
\begin{equation*}
u(t,X_t) = g(X_T) + \int_t^T f\left(r,X_r,u(r,X_r),v(r,X_r)\right)dV_r  -(M^{s,x}_T - M^{s,x}_t),
\end{equation*}
where $M^{s,x}$ denotes the martingale part of the unique solution of $BSDE^{s,x}(f,g)$.
\end{corollary}
\begin{proof}
The corollary follows from Theorem \ref{Defuv} and Lemma \ref{ModifImpliesdV}. 
\end{proof}

We now introduce now a probabilistic notion of solution for $Pseudo-PDE(f,g)$.
\begin{definition} \label{D417}
A function $u: [0,T] \times E \rightarrow {\mathbbm R}$
will be said 
to be a {\bf martingale solution} of $Pseudo-PDE(f,g)$ if 
% solve $Pseudo-PDE(f,g)$ in the \textbf{martingale sense} if 
$u\in\mathcal{D}(\mathfrak{a})$  and
	\begin{equation}\label{PDEextended}
	\left\{\begin{array}{rcl}
	\mathfrak{a}(u)&=& -f(\cdot,\cdot,u,\sqrt{\mathfrak{G}(u,u)})\\
	u(T,\cdot)&=&g.
	\end{array}\right.
	\end{equation}
\end{definition}
\begin{remark} \label{R417}
The first equation of \eqref{PDEextended} holds in $L^0_X$, hence up to a zero potential set. The second one is a pointwise equality.
\end{remark}

\begin{proposition}\label{CoroClassic}
A classical solution $u$ of $Pseudo-PDE(f,g)$ such that 
$\Gamma(u,u)\in\mathcal{L}^1_X$,  is also a martingale solution.	
\\
Conversely, if $u$ is a martingale solution of $Pseudo-PDE(f,g)$  
%\\$u\in
belonging to $\mathcal{D}(a)$, then $u$ is a classical solution of $Pseudo-PDE(f,g)$ up to a zero-potential set, meaning that the first equality of \eqref{PDE} holds up to a set of zero potential. 
%Version Francesco:	
%A classical solution $u$ of $Pseudo-PDE(f,g)$ 
%such that 
%\\
%$\Gamma(u,u) \in \mathcal{L}^1_X$
%is a martingale solution.

\end{proposition}
\begin{proof} 
Let $u$ be a classical solution of $Pseudo-PDE(f,g)$ verifying 
\\$\Gamma(u,u)\in\mathcal{L}^1_X$, 
% to \eqref{PDE} 
 Definition \ref{MarkovPDE}   and Corollary \ref{RExtendedClassical} imply 
that $u\in\mathcal{D}(\mathfrak{a})$, 
\\$u(T,\cdot)=g$, and the equalities up to zero potential sets
\begin{equation}
\mathfrak{a}(u) = a(u)
=-f(\cdot,\cdot,u,\Gamma(u,u))
=-f(\cdot,\cdot,u,\mathfrak{G}(u,u)),
\end{equation}
which shows that $u$ is a martingale solution.
%Concerning the second statement, let $u$ be a  martingale solution of $Pseudo-PDE(f,g)$ with $u\in\mathcal{D}(a)$. 
%Taking \eqref{PDEextended} and \ref{RExtendedClassical} into account, we have 
% $u(T,\cdot)=g$, and the equalities up to zero potential sets
% \begin{equation}
%  a(u)= \mathfrak{a}(u)
% =-f(\cdot,\cdot,u,\mathfrak{G}(u,u)=-f(\cdot,\cdot,u,\Gamma(u,u)).
%\end{equation}
Similarly, the second statement follows by Definition \ref{D417}
and again Corollary \ref{RExtendedClassical}.

%Version Francesco:
%The proof is a consequence of Proposition \ref{classicimpliesBSDE},
%Theorem \ref{Defuv} and Corollary \ref{uvBSDE}.
\end{proof}

%%%%  THEOREME OU PROP? LE VRAI THEOREME EST LE SUIVANT
\begin{theorem} \label{RMartExistence}
%Let $(\mathbbm{P}^{s,x})_{(s,x)\in[0,T]\times E}$ be a Markov class associated to a transition function measurable 
%in time (see Definitions \ref{defMarkov} and \ref{DefFoncTrans}) which
%fulfills Hypothesis \ref{MPwellposed},
%i.e. it is a solution of a well-posed martingale problem associated with
%the triplet $(\mathcal{D}(a),a,V)$.
%Moreover we suppose  Hypothesis $H^{mom}(\zeta,\eta)$ for some positive
% $\zeta,\eta$. Let $\mathfrak{a}$, $\mathfrak{G}$ be the extended operators defined in Definitions \ref{extended} and \ref{extendedgamma}. Let $(f,g)$ be a couple verifying $H(\zeta,\eta)$.
Assume Hypothesis \ref{HypX} and \ref{Hpq} and  let $(u,v)$ be the functions defined in Theorem \ref{Defuv}. 

Then $u\in\mathcal{D}(\mathfrak{a})$, $v^2 = \mathfrak{G}(u,u)$ and
$u$ is a martingale solution of $Pseudo-PDE(f,g)$.
% $u$ solves $Pseudo-PDE(f,g)$ in the \textbf{martingale sense}.
\end{theorem}

\begin{proof}
For any $(s,x)\in[0,T]\times E$, by Corollary \ref{uvBSDE}, for $t \in [s,T]$,  we have
\\
$u(t,X_t)-u(s,x) = -\int_s^tf(r,X_r,u(r,X_r),v(r,X_r))dV_r +(M^{s,x}_t-M^{s,x}_s)\quad \mathbbm{P}^{s,x}$ a.s.
so by Definition \ref{extended},  $u\in \mathcal{D}(\mathfrak{a})$, $\mathfrak{a}(u)= -f(\cdot,\cdot,u,v)$ and 
\\
$M[u]^{s,x}=M^{s,x}_{\cdot}-M^{s,x}_s$.
\\
Moreover by Theorem \ref{Defuv} we have $\frac{d\langle M^{s,x}\rangle}{dV}=v^2(\cdot,X_{\cdot})$ $dV\otimes d\mathbbm{P}^{s,x}$ a.e. on $[s,T]$, 
so by Proposition \ref{P321} 
%by Proposition \ref{uniquenessupto}
 it follows $v^2 = \mathfrak{G}(u,u)$ and therefore,
the $L^2_X$ equality
\\
$\mathfrak{a}(u)= -f(\cdot,\cdot,u,\sqrt{\mathfrak{G}(u,u)})$,
which establishes the first line of \eqref{PDEextended}.
\\
Concerning the second line, we have for any $x\in E$,
\\
%%the last item, for any $x\in E$, we also have under $\mathbbm{P}^{T,x}$ the a.s. equalities 
$u(T,x)=u(T,X_T)=g(X_T)=g(x)$ $\mathbbm{P}^{T,x}$ a.s. so $u(T,\cdot)=g$ (in the deterministic pointwise sense). 
\end{proof}
% \begin{remark}
% The equality $\mathfrak{a}(u)= -f(\cdot,\cdot,u,\sqrt{\mathfrak{G}(u,u)})$ 
% takes place in $L^2_X$.
% \end{remark}
%So in the most general setup, the function $u$ constructed by
% the $BSDE^{s,x}(f,g)$ is a martingale solution of  $Pseudo-PDE(f,g)$.
% in the martingale sense.
% However, a priori, the extended operators 
%have  no analytical meaning.

%%%%%%% . FAUT-IL EN PARLER ICI? N'EST
%PAS SUFFISANT DANS L'INTRODUTION?
% In  the companion paper  \cite{paper2}
% we will show that $u$ also solves $Pseudo-PDE(f,g)$ in
%  a more specific analytical mild sense.
%However,
 % The notion  of martingale solution for  Pseudo-PDE 
% stated above  has  two interesting features. First,
%   any classical solution (see Proposition \ref{CoroClassic}) 
% of
% $Pseudo-PDE(f,g)$ is  a martingale solution,
%  secondly
We conclude the section  with  Theorem  \ref{P418} 
which states that  the previously constructed martingale solution of $Pseudo-PDE(f,g)$
%\eqref{PDEextended}
 is unique.

%
%so by 
%Corollary \ref{uvBSDE} and Definition \ref{domainextended}, it is clear that $u'\in\mathcal{D}(\mathfrak{a})$ with $\mathfrak{a}(u')=-f(\cdot,\cdot,u,v)$, where $u,v$ are the functions built in Theorem \ref{Defuv}.  We know by definition of $u$ that for any $(s,x)\in[0,T]\times E$,
%$Y^{s,x}$ is also a $\mathbbm{P}^{s,x}$-version of $u(\cdot, X_{\cdot})$ on $[s,T]$.
%So  $u'(\cdot,X_{\cdot})$ and $u(\cdot,X_{\cdot})$ are 
%$\mathbbm{P}^{s,x}$-modifications on $[s,T]$;
% by Lemma \ref{ModifImpliesdV} and Proposition \ref{uniquenessupto}, $u=u'$ up to a zero potential set. Moreover by Proposition \ref{bracketindomain}, under any $\mathbbm{P}^{s,x}$,
%\begin{equation*}
% \int_s^{\cdot}\Gamma(u',u')(r,X_r)dV_r=\langle M[u']^{s,x}\rangle=\langle M^{s,x}\rangle = \int_s^{\cdot} v^2(r,X_r)dV_r,
%\end{equation*}
%so $v^2=\Gamma(u',u')$ up to a zero potential set, and $u'$ solves $Pseudo-PDE(f,g)$  in the martingale sense. 

%So under a certain growth condition on the carr\'e du champs operator, the only possible classical solution is $u$.

\begin{theorem} \label{P418}
%The problem \eqref{PDEextended} admits a  unique solution.
Under Hypothesis \ref{HypX} and \ref{Hpq}, $Pseudo-PDE(f,g)$ admits a unique martingale solution.
\end{theorem}
\begin{proof}
Existence has been the object of Theorem \ref{RMartExistence}.

Let $u$ and $u'$ be two elements of $\mathcal{D}(\mathfrak{a})$ solving \eqref{PDEextended} and let 
$(s,x)\in[0,T]\times E$ be fixed. By Definition \ref{domainextended} and Remark
\ref{BSDESmallInt}, the process $u(\cdot,X_{\cdot})$ (respectively $u'(\cdot,X_{\cdot})$) under $\mathbbm{P}^{s,x}$ admits 
 a cadlag modification 
 $U^{s,x}$ (respectively $U'^{s,x}$) on $[s,T]$,
 which is a special semi-martingale with decomposition 
\begin{equation} \label{E531}
\begin{array}{rcl}
U^{s,x} &=& u(s,x) + \int_s^{\cdot} \mathfrak{a}(u)(r,X_r)dV_r + M[u]^{s,x} \\
&=&   u(s,x) - \int_s^{\cdot} f\left(r,X_r,u(r,X_r),\sqrt{\mathfrak{G}(u,u)}(r,X_r)\right)dV_r + M[u]^{s,x} \\
&=&	u(s,x) - \int_s^{\cdot} f\left(r,X_r,U^{s,x},\sqrt{\mathfrak{G}(u,u)}(r,X_r)\right)dV_r + M[u]^{s,x}, 
\end{array}
\end{equation} 
 where the third equality of \eqref{E531} comes from Lemma \ref{ModifImpliesdV}.
Similarly we have   $U'^{s,x}=u'(s,x) - \int_s^{\cdot} f\left(r,X_r,U'^{s,x},\sqrt{\mathfrak{G}(u',u')}(r,X_r)\right)dV_r + M[u']^{s,x}$).
\\
The processes $M[u]^{s,x}$ and $M[u']^{s,x}$  (introduced at Definition \ref{extended}) belong to
 $\mathcal{H}^2_0$; by
Proposition \ref{P321},
 $\langle M[u]^{s,x}\rangle =\int_s^{\cdot}\mathfrak{G}(u,u)(r,X_r)dV_r$ (respectively 
 \\
  $\langle M[u']^{s,x}\rangle =\int_s^{\cdot}\mathfrak{G}(u',u')(r,X_r)dV_r$). 
Moreover since 
\\
$u(T,\cdot)=u'(T,\cdot)=g$, then  $u(T,X_T)=u'(T,X_T)=g(X_T)$ a.s. then 
the couples  $(U^{s,x}, M[u]^{s,x})$ and $(U'^{s,x}, M[u']^{s,x})$ both 
verify the equation (with respect to $\mathbbm P^{s,x}$).
\begin{equation} \label{EBSDEweaker}
Y_{\cdot} = g(X_T)+\int_{\cdot}^Tf\left(r,X_r,Y_r,\sqrt{\frac{d\langle M\rangle}{dV}}(r)\right)dV_r - (M_T-M_{\cdot})
\end{equation}
on $[s,T]$.
\\
Even though we do not have a priori information  on the square
 integrability of $U^{s,x}$ and $U'^{s,x}$, we know that  $M[u]^{s,x}$ and $M[u']^{s,x}$ are 
in $\mathcal{H}^2$ and equal to zero at time $s$, and that $U^{s,x}_s$ and $U'^{s,x}_s$ are deterministic so $L^2$. By  Lemma \ref{LED+Pext} and
the fact that $(U^{s,x},M[u]^{s,x})$ and $(U'^{s,x},M[u']^{s,x})$ 
solve the BSDE in the weaker sense \eqref{EBSDEweaker}, it is sufficient to 
%\eqref{EqLedPext}.
conclude that both
solve $BSDE^{s,x}(f,g)$ on $[s,T]$. 
By  Theorem \ref{uniquenessBSDE} and Remark \ref{BSDESmallInt} 
the two couples are $\mathbb {P}^{s,x}$-indistinguishable.
This implies that $u(\cdot,X_{\cdot})$ and $u'(\cdot,X_{\cdot})$ are 
$\mathbb {P}^{s,x}$-modifications one of the other on $[s,T]$. In particular, considering their values at time $s$, we have $u(s,x)=u'(s,x)$. We therefore have $u'=u$.

\end{proof}
\begin{corollary} 
There is at most one classical solution $u$ of $Pseudo-PDE(f,g)$ 
such that $\Gamma(u,u) \in {\mathcal L}^1_X$.
\end{corollary}
\begin{proof}
The proof follows from Proposition \ref{CoroClassic} and Theorem \ref{P418}.
\end{proof}

\section{Applications}
\label{SUpcoming}

In \cite{paper2} which is the continuation of the present paper,
 several examples are  studied. The examples below
fit in the framework of Section \ref{SecMProcess}.
\\
A first typical example is the setup of jump diffusions as in the formalism D.W. Stroock in \cite{stroock1975diffusion}. These are Markov processes which solve a Martingale problem associated to an operator of the type 
\begin{equation*}
\begin{array}{rcl}
a(\phi) &=&\partial_t\phi + \frac{1}{2}\underset{i,j\leq d}{\sum} (\sigma\sigma^\intercal)_{i,j}\partial^2_{x_ix_j}\phi + \underset{i\leq d}{\sum} \mu_i\partial_{x_i}\phi \\
&&+\int\left(\phi(\cdot,\cdot+y)-\phi(\cdot,y)-\frac{1}{1+\|y\|^2}\underset{i\leq d}{\sum}y_i\partial_{x_i}\phi\right)K(\cdot,\cdot,dy),
\end{array}
\end{equation*}
where $\mu$ is a bounded Borel function with values in $\mathbbm{R}^d$ and $\sigma$ is a continuousBorel function with values in 
$GL_d(\mathbbm{R})$, the set of invertible matrices of size $d$. 
$K$ is a L\'evy kernel.
%, meaning that for every $(t,x)\in[0,T]\times \mathbbm{R}^d$, $K(t,x,\cdot)$ is a $\sigma$-finite measure on $\mathbbm{R}^d\backslash\{0\}$ verifying
%\\
%$\int \frac{\| y\|^2}{1+\| y\|^2}K(t,x,dy)<\infty$ and for every Borel set $A\in\mathcal{B}(\mathbbm{R}^d\backslash\{0\})$,
%\\
%$(t,x)\longmapsto \int_A \frac{\| y\|^2}{1+\| y\|^2}K(t,x,dy)$ is Borel.
\\
\\
We also study Markov processes associated to a large class of pseudo-differential operators 
%of type $q(\cdot,D)$ where
%\begin{equation}
%q(\cdot,D)(\phi):x\longmapsto\frac{1}{(2\pi)^{\frac{d}{2}}}\int_{\mathbbm{R}^d} e^{i(x,\xi)}q(x,\xi)\hat{\phi}(\xi)d\xi.
%\end{equation}
with the formalism of N. Jacob in \cite{jacob2005pseudo}.
% Here $\hat{\phi}$ denotes the Fourier transform of $\phi$.
A typical example of equation considered is
\begin{equation}\label{PDEsymbol}
\left\{
\begin{array}{lcl}
\partial_tu-(-\Delta)^{\frac{\alpha}{2}}u = f(\cdot,\cdot,u,\sqrt{\Gamma^{\alpha}(u,u)}) \,\text{on }[0,T]\times\mathbbm{R}^d\\
u(T,\cdot)=g.
\end{array}\right.
\end{equation}
Here, the fractional Laplace operator $(-\Delta)^{\frac{\alpha}{2}}$ is given for some $\alpha\in]0,2[$ by $\phi\longmapsto  c_{\alpha}PV\int_{\mathbbm{R}^d} \frac{(\phi(\cdot+y)-\phi)}{\| y\|^{d+\alpha}}dy$ where $c_{\alpha}$ is some positive constant and $PV$ denotes the principal value operator. 
\begin{equation}
\Gamma^{\alpha}(\phi,\phi)=c_{\alpha}PV\int_{\mathbbm{R}^d} \frac{(\phi(\cdot,\cdot+y)-\phi)^2}{\| y\|^{d+\alpha}}dy
\end{equation}
is the corresponding Carré du champ. The forward process of the corresponding BSDEs is the $\alpha$-stable Levy process.
\\
\\
An other example of application is given by solutions of SDEs with distributional drift, which are studied in \cite{frw1}.  These  permit to tackle semilinear parabolic PDEs with distributional drift of type 
\begin{equation}\label{PDEdistri}
\left\{
\begin{array}{l}
 \partial_tu + \frac{1}{2}\sigma^2\partial^2_x u + b'\partial_xu +f(\cdot,\cdot,u,\sigma|\partial_xu|)=0\quad\text{ on }[0,T]\times\mathbbm{R}\\
 u(T,\cdot) = g,
\end{array}\right.
\end{equation}
where $b$ is only a continuous function, hence $b'$ is a distribution.
\\
\\
Finally, examples in non Euclidean state spaces are given with the study of diffusions in a compact differential manifold $M$. A typical example is the Brownian motion in a Riemannian manifold. The equation considered is then of type
\begin{equation}\label{PDEmanifold}
\left\{
\begin{array}{l}
 \partial_tu + \Delta_M u +f(\cdot,\cdot,u,\|\nabla_Mu\|_2)=0\quad\text{ on }[0,T]\times M\\
 u(T,\cdot) = g,
\end{array}\right.
\end{equation}
where $\Delta_M$  is the Laplace-Beltrami operator and $\nabla_M$ is the gradient in local coordinates. More general equations are considered in \cite{paper2}.

\begin{appendices}

\section{Markov classes}\label{A1}

We recall in this Appendix some basic definitions and results concerning Markov processes. For a complete study of homogeneous Markov processes, one may consult \cite{dellmeyerD}, concerning non-homogeneous Markov classes, our reference was chapter VI of \cite{dynkin1982markov}. Some results are only stated, but the advised reader may consult \cite{paper1preprint1} and \cite{paperAF} in which all announced results are carefully proven.

The first definition refers to the  canonical space that one can find in \cite{jacod79}, see paragraph 12.63.
\begin{notation}\label{canonicalspace}
In the whole section  $E$ will be a fixed  Polish  space (a separable completely metrizable topological space), and $\mathcal{B}(E)$  its Borel $\sigma$-field. $E$ will be called the \textbf{state space}. 
\\
\\
We consider $T\in\mathbbm{R}^*_+$. We denote $\Omega:=\mathbbm{D}(E)$ the Skorokhod space of functions from $[0,T]$ to $E$  right-continuous  with left limits and continuous at time $T$ (e.g. cadlag). For any $t\in[0,T]$ we denote the coordinate mapping $X_t:\omega\mapsto\omega(t)$, and we introduce on $\Omega$ the $\sigma$-field  $\mathcal{F}:=\sigma(X_r|r\in[0,T])$. 
\\
\\
On the measurable space $(\Omega,\mathcal{F})$, we introduce the measurable \textbf{canonical process}
\begin{equation*}
X:
\begin{array}{rcl}
(t,\omega)&\longmapsto& \omega(t)\\ \relax
([0,T]\times \Omega,\mathcal{B}([0,T])\otimes\mathcal{F}) &\longrightarrow & (E,\mathcal{B}(E)),
\end{array}
\end{equation*}
and the right-continuous filtration $(\mathcal{F}_t)_{t\in[0,T]}$ where $\mathcal{F}_t:=\underset{s\in]t,T]}{\bigcap}\sigma(X_r|r\leq s)$ if $t<T$, and $\mathcal{F}_T:= \sigma(X_r|r\in[0,T])=\mathcal{F}$.
\\
\\
$\left(\Omega,\mathcal{F},(X_t)_{t\in[0,T]},(\mathcal{F}_t)_{t\in[0,T]}\right)$ will be called the \textbf{canonical space} (associated to $T$ and $E$).
\\
\\
For any $t \in [0,T]$ we denote $\mathcal{F}_{t,T}:=\sigma(X_r|r\geq t)$, and
for any $0\leq t\leq u<T$ we will denote
$\mathcal{F}_{t,u}:= \underset{n\geq 0}{\bigcap}\sigma(X_r|r\in[t,u+\frac{1}{n}])$.
\end{notation}
%We recall that since $E$ is Polish, then $\mathbbm{D}(E)$ can be equipped with a Skorokhod distance which makes it a Polish metric space (see Theorem 5.6 in chapter 3 of \cite{EthierKurz}, and for which the Borel $\sigma$-field is $\mathcal{F}$ (see Proposition 7.1 in chapter 3 of \cite{EthierKurz}). This in particular implies that $\mathcal{F}$ is separable, as the Borel $\sigma$-field of a separable metric space.

\begin{remark}
Previous definitions and all the notions of this Appendix,
 extend to a time interval equal to $\mathbbm{R}_+$ or replacing the Skorokhod space with the Wiener space of continuous functions from $[0,T]$ (or $\mathbbm{R}_+$) to $E$.
% but since our goal is to work on a finite time interval, we will not consider this situation.
\end{remark}

\begin{definition}\label{Defp}
The function 
\begin{equation*}
    p:\begin{array}{rcl}
        (s,x,t,A) &\longmapsto& p(s,x,t,A)   \\ \relax
        [0,T]\times E\times[0,T]\times\mathcal{B}(E) &\longrightarrow& [0,1], 
    \end{array}
\end{equation*}
will be called \textbf{transition function} if, for any $s,t$ in $[0,T]$, $x\in E$,  $A\in \mathcal{B}(E)$, it verifies

\begin{enumerate}
\item $x \mapsto p(s,x,t,A)$ is Borel,
\item $B \mapsto p(s,x,t,B)$ is a probability measure on $(E,\mathcal{B}(E))$,
\item if $t\leq s$ then $p(s,x,t,A)=\mathds{1}_A(x)$,
\item if $s<t$, for any $u>t$, $\int_{E} p(s,x,t,dy)p(t,y,u,A) = p(s,x,u,A)$.
\end{enumerate}
\end{definition}
The latter statement is the well-known \textbf{Chapman-Kolmogorov equation}.

\begin{definition}\label{DefFoncTrans}
A transition function $p$ for which  the first item is reinforced 
supposing that $(s,x)\longmapsto p(s,x,t,A)$ is Borel for any $t,A$,
 will be said  \textbf{measurable in time}.

\end{definition}
%\begin{remark} \label{RDefFoncTrans}
% Let $p$ be a transition function which is measurable in time.
%By approximation by step functions, one can easily show that,
%  for any Borel function $\phi$ from $E$ to $\mathbbm{R}$ then
%$(s,x)\mapsto \int \phi(y)p(s,x,t,dy)$ is Borel, provided
%previous integral  makes sense.
%In this paper we will only consider transition functions which are measurable in time.
%\end{remark}
\begin{definition}\label{defMarkov}
A \textbf{canonical Markov class} associated to a transition function $p$ is a set of probability measures $(\mathbbm{P}^{s,x})_{(s,x)\in[0,T]\times E}$ defined on the measurable space 
$(\Omega,\mathcal{F})$ and verifying for any $t \in [0,T]$ and $A\in\mathcal{B}(E)$
\begin{equation}\label{Markov1}
\mathbbm{P}^{s,x}(X_t\in A)=p(s,x,t,A),
\end{equation}
and for any $s\leq t\leq u$
\begin{equation}\label{Markov2}
\mathbbm{P}^{s,x}(X_u\in A|\mathcal{F}_t)=p(t,X_t,u,A)\quad \mathbbm{P}^{s,x}\text{ a.s.}
\end{equation}
%\eqref{Markov2} will  be called the \textbf{Markov property}.
\end{definition}
\begin{remark}\label{Rfuturefiltration}
	Formula 1.7 in Chapter 6 of \cite{dynkin1982markov} states
	that for any $F\in \mathcal{F}_{t,T}$ yields
	\begin{equation}\label{Markov3}
	\mathbbm{P}^{s,x}(F|\mathcal{F}_t) = \mathbbm{P}^{t,X_t}(F)=\mathbbm{P}^{s,x}(F|X_t)\,\text{  }\,  \mathbbm{P}^{s,x} \text{a.s.}
	\end{equation}
	Property  \eqref{Markov3}  will  be called 
	\textbf{Markov property}.
\end{remark}
For the rest of this section, we are given a canonical Markov class $(\mathbbm{P}^{s,x})_{(s,x)\in[0,T]\times E}$ which transition function is measurable in time.

\begin{proposition}\label{Borel}
For any event $F\in \mathcal{F}$,  
$(s,x)\longmapsto \mathbbm{P}^{s,x}(F)$ is Borel.
% For any random variable $Z$ such that all these expectations exist and are
For any random variable $Z$, if the function $(s,x)\longmapsto \mathbbm{E}^{s,x}[Z]$ 
is well-defined (with possible values in $[-\infty, \infty]$),
then it is Borel. 
\end{proposition}

\begin{lemma}\label{LemmaBorel}
Let $V$ be a continuous  non-decreasing function on $[0,T]$ and 
\\
$f\in\mathcal{B}([0,T]\times E)$ be such that for every $(s,x)$, $\mathbbm{E}^{s,x}[\int_s^{T}|f(r,X_r)|dV_r]<\infty$, then 
\\
$(s,x)\longmapsto \mathbbm{E}^{s,x}[\int_s^{T}f(r,X_r)dV_r]$ is Borel.
\end{lemma}

\begin{definition}\label{CompletedBasis}
For any $(s,x)\in[0,T]\times E$ we will consider the  $(s,x)$-\textbf{completion} $\left(\Omega,\mathcal{F}^{s,x},(\mathcal{F}^{s,x}_t)_{t\in[0,T]},\mathbbm{P}^{s,x}\right)$ of the stochastic basis $\left(\Omega,\mathcal{F},(\mathcal{F}_t)_{t\in[0,T]},\mathbbm{P}^{s,x}\right)$ by defining $\mathcal{F}^{s,x}$ as  the $\mathbbm{P}^{s,x}$-completion of $\mathcal{F}$ , by extending $\mathbbm{P}^{s,x}$ to $\mathcal{F}^{s,x}$ and finally by defining  $\mathcal{F}^{s,x}_t$ as the $\mathbbm{P}^{s,x}$-closure of $\mathcal{F}_t$ for every $t\in[0,T]$. 
\end{definition}
We remark that, for any $(s,x)\in[0,T]\times E$, $\left(\Omega,\mathcal{F}^{s,x},(\mathcal{F}^{s,x}_t)_{t\in[0,T]},\mathbbm{P}^{s,x}\right)$ is a stochastic basis fulfilling the usual conditions. 

Proposition 3.13 in \cite{paperAF} states the following.
\begin{proposition}\label{ConditionalExp} Let $(s,x)\in[0,T]\times E$ be fixed, $Z$ be a random variable and $t\in[s,T]$, then 
$\mathbbm{E}^{s,x}[Z|\mathcal{F}_t]=\mathbbm{E}^{s,x}[Z|\mathcal{F}^{s,x}_t]$ $\mathbbm{P}^{s,x}$ a.s.
\end{proposition}

We recall here Definition 4.1 in \cite{paperAF}.
\begin{definition}\label{DefAF} We denote $\Delta:=\{(t,u)\in[0,T]^2|t\leq u\}$.
 On $(\Omega,\mathcal{F})$, we define a \textbf{non-homogeneous Additive Functional} (shortened AF) as a random-field indexed by $\Delta$ 
 $A:=(A^t_u)_{(t,u)\in\Delta}$,
with values in $\mathbbm{R}$,  
 verifying the two following conditions.
\begin{enumerate}
\item For any $(t,u)\in\Delta$, $A^t_u$ is $\mathcal{F}_{t,u}$-measurable;
\item for any $(s,x)\in[0,T]\times E$, there exists a real cadlag $\mathcal{F}^{s,x}$-adapted process $A^{s,x}$ (taken equal to zero on $[0,s]$ by convention) such that for any $x\in E$ and $s\leq t\leq u$, $A^t_u = A^{s,x}_u-A^{s,x}_t \,\text{  }\, \mathbbm{P}^{s,x}$ a.s.
\end{enumerate}
%We denote by $A^t$ the ($\mathcal{F}_{t,\cdot}$-adapted) process
% $u\mapsto A^t_u$ indexed by $[t,T]$. For any $(s,x)\in [0,t]\times E$, $A^{s,x}_{\cdot}-A^{s,x}_t$ 
% is a $\mathbbm{P}^{s,x}$-version  of $A^t$ on $[t,T]$.
%\\
$A^{s,x}$ will be called the \textbf{cadlag version of $A$ under} $\mathbbm{P}^{s,x}$.
% since $A$ is not a process.
\\
\\
An AF will be called a \textbf{non-homogeneous square integrable Martingale Additive Functional} (shortened square integrable MAF) if under any $\mathbbm{P}^{s,x}$ its cadlag version is a square integrable martingale.
\end{definition}

A immediate consequence of Proposition 4.17 in \cite{paperAF} is the following.
\begin{proposition}\label{bracketMAFs}
Given an increasing continuous function $V$, if in every stochastic basis $\left(\Omega,\mathcal{F}^{s,x},(\mathcal{F}^{s,x}_t)_{t\in[0,T]},\mathbbm{P}^{s,x}\right)$, we have $\mathcal{H}^2_0=\mathcal{H}^{2,V}$, then we can state the following.
\\
Let $M$ and $M'$ be two square integrable MAFs and let $M^{s,x}$ (respectively $M'^{s,x}$) be the cadlag version of $M$ (respectively $M'$) under a fixed $\mathbbm{P}^{s,x}$. There exists a Borel function $k\in\mathcal{B}([0,T]\times E,\mathbbm{R})$  such that for any $(s,x)\in[0,T]\times E$, $\langle M^{s,x},M'^{s,x}\rangle =  \int_s^{\cdot}k(r,X_r)dV_r$.
\\
In particular if $M$ is a square integrable MAF and $M^{s,x}$ its cadlag version under a fixed $\mathbbm{P}^{s,x}$,  there exists a Borel function $k\in\mathcal{B}([0,T]\times E,\mathbbm{R})$ (which can be taken positive) such that for any $(s,x)\in[0,T]\times E$, $\langle M^{s,x}\rangle =  \int_s^{\cdot}k(r,X_r)dV_r$.
\end{proposition}

\section{Technicalities related to Section \ref{S2}}\label{SC}

\begin{prooff} \\ of Proposition \ref{Decomposition}. Since we have $dA \ll dA + dB$ in the sense of stochastic measures with
$A,B$ predictable,  there
exists a predictable positive process $K$  such that $A = \int_0^{\cdot} K_sdA_s + \int_0^{\cdot} K_sdB_s$
up to indistinguishability, see Proposition I.3.13 in \cite{jacod}. Now there exists a $\mathbbm{P}$-null set $\mathcal{N}$ such that for any $\omega\in\mathcal{N}^c$ we have $0\leq \int_0^\cdot K_s(\omega)dB_s(\omega) =
	\int_0^\cdot (1-K_s(\omega))dA_s(\omega)$ ,
	so $K(\omega)\leq 1$ $dA(\omega)$ a.e. on $\mathcal{N}^c$. Therefore if we set  $E(\omega) = \{t: K_t(\omega) = 1\}$ and $F(\omega) = \{t: K_t(\omega) < 1\}$ then $E(\omega)$ and $F(\omega)$ are disjoint Borel sets and $dA(\omega)$ has all its mass in $E(\omega)\cup F(\omega)$ so we can decompose $dA(\omega)$ within these two sets. 
	\\
	\\
	We therefore define the processes $A^{\perp B} =  \int_0^{\cdot} \mathds{1}_{\{K_s = 1\}}dA_s$ and ; $A^B = \int_0^{\cdot} \mathds{1}_{\{K_s < 1\}}dA_s$.
	$A^{\perp B}$ and $A^B$ are both in $\mathcal{V}^{p,+}$, and $A= A^{\perp B}+A^B$. 
In particular the  (stochastic) measures $dA^{\perp B}$ and $dA^B$ fulfill
$dA^{\perp B}(\omega)(G) = dA(\omega)(E(\omega) \cap G)$ and
$dA^B(\omega)(G) = dA(\omega)(F(\omega) \cap G)$.
\\
We remark $dA^{\perp B} \bot dB$ in the sense of stochastic measures. 
Indeed, fixing $\omega \in \mathcal{N}^c$,
for  $t\in E(\omega)$, $K_t(\omega)=1$, so 
$\int_{E(\omega)} dA(\omega) = \int_{E(\omega)} dA(\omega) + \int_{E(\omega)} dB(\omega)$ implying that $\int_{E(\omega)} dB(\omega) = 0$.
% So $dA^{\perp B} \bot dB$ in the sense of stochastic measures 
Since for any $\omega\in\mathcal{N}^c$, $dB(\omega)\left(E(\omega)\right)=0$ while $dA^{\perp B}(\omega)$ has all its mass in $E(\omega)$, which gives
 this first result.
\\
\\
Now let us prove  $dA^B\ll dB$ in the sense of stochastic measure.
\\
Let $\omega\in\mathcal{N}^c$, and let $G\in\mathcal{B}([0,T])$, such that $\int_GdB(\omega)=0$. Then
\begin{equation*}
	\begin{array}{rcl}
	\int_GdA^B(\omega) &=& \int_{G\cap F(\omega)}dA(\omega) \\
	&=& \int_{G\cap F(\omega)}K(\omega)dA(\omega) + \int_{G\cap F(\omega)}K(\omega)dB(\omega) \\
	&=& \int_{G\cap F(\omega)}K(\omega)dA(\omega).
	\end{array}
\end{equation*}
	So $\int_{G\cap F(\omega)}(1-K(\omega))dA(\omega)=0$, but $(1-K(\omega))>0$ on $F(\omega)$. 
	\\
	So $dA^B(\omega)(G) = 0$. Consequently for every $\omega\in\mathcal{N}^c$, $dA^B(\omega)\ll dB(\omega)$ and so that $dA^B\ll dB$.
	\\
	Now, since $K$ is positive and $K(\omega)\leq 1$ $dA(\omega)$ a.e. for almost all $\omega$, we can replace $K$ by $K\wedge 1$ which is still positive predictable, without changing the associated stochastic measures $dA^B,dA^{\perp B}$; therefore we can
 consider that $K_t(\omega)\in[0,1]$ for all $(\omega,t)$.
	\\
	We remark that for   $\mathbbm{P}$ almost all $\omega$
	the decomposition  $A^{\perp B}$ and $A^B$ is unique
	because of the corresponding   uniqueness  of the decomposition 
	in the Lebesgue-Radon-Nikodym theorem for each fixed $\omega\in\mathcal{N}^c$. 
	
	Since $dA^B \ll dB$, again by Proposition I.3.13 in \cite{jacod}, 
	there exists a predictable positive process that we will call $\frac{dA}{dB}$ 
	such that $A^B = \int_0^\cdot \frac{dA}{dB}dB$ and which is only unique up to $dB\otimes d\mathbbm{P}$ null sets. 
\end{prooff}

 \begin{proposition}\label{CS}
 	Let $M$ and $M'$ be two local martingales in $\mathcal{H}^2_{loc}$ and  let 
 	\\
 	$V\in\mathcal{V}^{p,+}$. We have $\frac{d\langle M\rangle}{dV}\frac{d\langle M'\rangle}{dV} -\left(\frac{d\langle M,M'\rangle}{dV}\right)^2 \geq 0\quad dV\otimes d\mathbbm{P}$ a.e.
 \end{proposition}
 
 \begin{proof}
 	Let $x\in\mathbbm{Q}$. Since 
 	$\langle M+xM'\rangle$ is
 	an increasing process starting at zero, then
 	by Proposition \ref{Decomposition}, we have $\frac{d\langle M+xM'\rangle}{dV} \geq 0\quad dV\otimes d\mathbbm{P}$ a.e. 
 	\\
 	By the linearity property stated
 	in Proposition \ref{linearity}, we have
 	\\
 	$0\leq \frac{d\langle M+xM'\rangle}{dV} = \frac{d\langle M\rangle}{dV} + 2x\frac{d\langle M,M'\rangle}{dV}+x^2\frac{d\langle M'\rangle}{dV}$ $dV\otimes d\mathbbm{P}$ a.e.
 	Since $\mathbbm{Q}$ is countable, there exists a $dV\otimes d\mathbbm{P}$-null set $\mathcal{N}$ such that for 
$(\omega,t)\notin\mathcal{N}$ and $x\in\mathbbm{Q}$,
 	\\
 	$\frac{d\langle M\rangle}{dV}(\omega,t) + 2x\frac{d\langle M,M'\rangle}{dV}(\omega,t)+x^2\frac{d\langle M'\rangle}{dV}(\omega,t) \geq 0$.
 	By continuity of polynomes, this holds for any $x\in\mathbbm{R}$.
 Expressing the discriminant of this polynome, we deduce that  $4\left(\frac{d\langle M,M'\rangle}{dV}(\omega,t)\right)^2 - 4\frac{d\langle M\rangle}{dV}(\omega,t)\frac{d\langle M'\rangle}{dV}(\omega,t) \leq 0$ for all
 $(\omega,t)\notin\mathcal{N}$. 
 \end{proof}

 \begin{prooff}  \\ of Proposition \ref{DecompoMart}. Since the angular bracket $\langle M\rangle$ of a square integrable martingale
$M$  always belongs to 
 	$\mathcal{V}^{p,+}$, by Proposition \ref{Decomposition}, we can 
consider the
 	processes $\langle M\rangle^V$ and $\langle M\rangle^{\perp V}$;
in particular  there
 	exists a predictable process $K$ with values in $[0,1]$ such that $\langle M\rangle^V = \int_0^{\cdot}\mathds{1}_{\{K_s<1\}}d\langle M\rangle_s$ and  $\langle M\rangle^{\perp V}  =   \int_0^{\cdot}\mathds{1}_{\{K_s=1\}}d\langle M\rangle_s$.
 	\\
 	\\
 	We can then set $M^V = \int_0^{\cdot}\mathds{1}_{\{K_s<1\}}dM_s$ and $M^{\bot V}= \int_0^{\cdot}\mathds{1}_{\{K_s=1\}}dM_s$ which are well-defined because $K$ is predictable, and therefore $\mathds{1}_{\{K_t<1\}}$ and $\mathds{1}_{\{K_t=1\}}$ are also predictable. $M^V,M^{\bot V}$ belong to $\mathcal{H}_0^2$ because their
 	angular brackets are both  bounded by $\langle M\rangle_T\in L^1$. Since $K$ takes values in $[0,1]$, we have 
 	\\
 	$M^V+M^{\perp V} = \int_0^{\cdot}\mathds{1}_{\{K_s<1\}}dM_s+\int_0^{\cdot}\mathds{1}_{\{K_s=1\}}dM_s = M$;
 	\\
 	$\langle M^V \rangle = \int_0^{\cdot}\mathds{1}_{\{K_s<1\}}d\langle M\rangle_s = \langle M \rangle^V$;
 	$\langle M^{\perp V} \rangle = \int_0^{\cdot}\mathds{1}_{\{K_s=1\}}d\langle M\rangle_s = \langle M\rangle^{\perp V}$
 	\\
 	and $\langle M^V, M^{\bot V}\rangle = \int_0^{\cdot}\mathds{1}_{\{K_s<1\}}\mathds{1}_{\{K_s=1\}}d\langle M\rangle_s=0$. 
 \end{prooff}

\begin{prooff} \\ of Proposition \ref{OrthogonalSpaces}.
	We start by remarking that for any $M_1,M_2$ in $\mathcal{H}^2_0$, a consequence of Kunita-Watanabe's decomposition (see Theorem 4.27 in \cite{jacod}) is that $dVar(\langle M_1,M_2\rangle)\ll d\langle M_1\rangle$ and $dVar(\langle M_1,M_2\rangle)\ll d\langle M_2\rangle$.
\\
\\
Now, let  $M_1$ and $M_2$ be in $\mathcal{H}^{2,V}$. We have
$dVar(\langle M_1,M_2\rangle)\ll  d\langle M_1\rangle\ll dV$. So since $\langle M_1+M_2\rangle=\langle M_1\rangle+2\langle M_1,M_2\rangle+\langle M_2\rangle$, then 
	$d\langle M_1+M_2\rangle \ll dV$ which shows that $\mathcal{H}^{2,V}$ is a vector space.
	\\
	\\
	If $M_1$ and $M_2$ are in $\mathcal{H}^{2,\perp V}$, then since $dVar(\langle M_1,M_2\rangle)\ll  d\langle M_1\rangle$ we can write $Var(\langle M_1,M_2\rangle)=\int_0^{\cdot}\frac{dVar(\langle M_1,M_2\rangle)}{d\langle M_1\rangle}d\langle M_1\rangle$ which is almost surely singular with respect to $dV$ since $M_1$
	belongs to $ \mathcal{H}^{2,\perp V}$. 
	% it is supported in support of $d\langle M_1\rangle$.
	So, by the bilinearity of the angular bracket
	$\mathcal{H}^{2,\perp V}$ is also a vector space.
	\\
	\\
	Finally if $M_1\in\mathcal{H}^{2,V}$ and $M_2\in \mathcal{H}^{2,\perp V}$ then $dVar(\langle M_1,M_2\rangle)\ll  d\langle M_1\rangle\ll dV$ but we 
also have seen  that if $d\langle M_2\rangle$ is singular to $dV$ 
then so is $dVar(\langle M_1,M_2\rangle)  \ll d\langle M_2\rangle$.
	\\
	For fixed $\omega$, a measure being simultaneously dominated and singular 
	with respect to to $dV(\omega)$ is necessarily the null measure, so $dVar(\langle M_1,M_2\rangle)=0$ as a stochastic measure. Therefore $M_1$ and $M_2$ are strongly orthogonal, which implies in particular that $M_1$ and $M_2$ are orthogonal in $\mathcal{H}^2_0$.
	\\
	So we have shown that $\mathcal{H}^{2,V}$ and $\mathcal{H}^{2,\perp V}$ are orthogonal sublinear-spaces of $\mathcal{H}^2_0$ but we also know that $\mathcal{H}^2_0 = \mathcal{H}^{2,V}+\mathcal{H}^{2,\perp V}$ thanks to Proposition \ref{DecompoMart}, so 
	\\
	$\mathcal{H}^2_0 = \mathcal{H}^{2,V}\oplus^{\perp}\mathcal{H}^{2,\perp V}$.
	This implies that $\mathcal{H}^{2,V}=(\mathcal{H}^{2,\perp V})^{\perp}$ and 
	\\
	$\mathcal{H}^{2,\perp V}=(\mathcal{H}^{2, V})^{\perp}$ and therefore that these spaces are closed. So they are sub-Hilbert spaces. 
	%By the characterization
	%of the scalar product in $\mathcal{H}^{2}$, 
	We also have shown that they were strongly orthogonal spaces, 
	in the sense that any
	$M^1\in\mathcal{H}^{2,V}$, $M^2\in\mathcal{H}^{2,\perp V}$ are strongly orthogonal.
By localization the strong orthogonality property also extends
to  $M^1\in\mathcal{H}^{2,V}_{loc}$, $M^2\in\mathcal{H}^{2,\perp V}_{loc}$. 
\end{prooff}

%
%
%We recall here a classical notion of martingale theory.
%\begin{definition}
%	Let $p\in[1,\infty[$, a subset $\mathcal{H}\subset\mathcal{H}^p$ will be called a \textbf{stable sub-space} if it is a closed sublinear space such that
%	for any $M\in\mathcal{H}$, any event $A\in\mathcal{F}_0$ and any stopping time $\tau$ then $\mathds{1}_A M^{\tau}\in\mathcal{H}$.
%\end{definition}
%
%\begin{proposition}
%	$\mathcal{H}^{2, V}$ and $\mathcal{H}^{2, \perp V}$ are stable subspaces of $\mathcal{H}^2$.
%\end{proposition}
%\begin{proof}
%	Since by Proposition \ref{OrthogonalSpaces}, $\mathcal{H}^{2, V}$ and $\mathcal{H}^{2, \perp V}$ are closed sub-linear spaces of $\mathcal{H}^2$, then Proposition 4.3 in \cite{jacod79} states that
%the result will follow 
% if we  show that for any $M$ in $\mathcal{H}^{2, V}$ (respectively in $\mathcal{H}^{2, \perp V}$) and $H$ in $L^2(M)$ then $\int_0^{\cdot} HdM$ is in $\mathcal{H}^{2, V}$ (respectively in $\mathcal{H}^{2, \perp V}$). So let $M \in \mathcal{H}^{2, V}$ (respectively in $\mathcal{H}^{2, \perp V}$) and $H$ in $L^2(M)$, then $\langle \int_0^{\cdot} HdM\rangle = \int_0^{\cdot} H^2 d\langle M\rangle$ therefore if $d\langle M\rangle$ is dominated by $dV$ (respectively singular to $dV$), so is $d\langle \int_0^{\cdot} HdM\rangle$.
%\end{proof}

\end{appendices}
{\bf ACKNOWLEDGEMENTS.} The authors are grateful to Andrea Cosso
for stimulating discussions. The research of the first named author
was provided by a PhD fellowship (AMX) of the Ecole Polytechnique. 

%\newpage
\bibliographystyle{plain}
\bibliography{../../biblioPhDBarrasso_bib/biblioPhDBarrasso}

\def\polhk#1{\setbox0=\hbox{#1}{\ooalign{\hidewidth
  \lower1.5ex\hbox{`}\hidewidth\crcr\unhbox0}}}
  \def\polhk#1{\setbox0=\hbox{#1}{\ooalign{\hidewidth
  \lower1.5ex\hbox{`}\hidewidth\crcr\unhbox0}}} \def\cprime{$'$}
  \def\polhk#1{\setbox0=\hbox{#1}{\ooalign{\hidewidth
  \lower1.5ex\hbox{`}\hidewidth\crcr\unhbox0}}}
\begin{thebibliography}{10}

\bibitem{aliprantis}
C.~D. Aliprantis and K.~C. Border.
\newblock {\em Infinite-dimensional analysis}.
\newblock Springer-Verlag, Berlin, second edition, 1999.
\newblock A hitchhiker's guide.

\bibitem{bakry}
D.~Bakry, I.~Gentil, and M.~Ledoux.
\newblock {\em Analysis and geometry of Markov diffusion operators}, volume
  348.
\newblock Springer Science \& Business Media, 2013.

\bibitem{BSDEmildPardouxBally}
V.~Bally, E.~Pardoux, and L.~Stoica.
\newblock Backward stochastic differential equations associated to a symmetric
  {M}arkov process.
\newblock {\em Potential Anal.}, 22(1):17--60, 2005.

\bibitem{BandiniBSDE}
E.~Bandini.
\newblock Existence and uniqueness for backward stochastic differential
  equations driven by a random measure.
\newblock {\em Electronic Communications in Probability}, 20(71):1--13, 2015.

\bibitem{barles1997backward}
G.~Barles, R.~Buckdahn, and E.~Pardoux.
\newblock Backward stochastic differential equations and integral-partial
  differential equations.
\newblock {\em Stochastics: An International Journal of Probability and
  Stochastic Processes}, 60(1-2):57--83, 1997.

\bibitem{barles1997sde}
G.~Barles and E.~Lesigne.
\newblock S{DE}, {BSDE} and {PDE}.
\newblock In {\em Backward stochastic differential equations ({P}aris,
  1995--1996)}, volume 364 of {\em Pitman Res. Notes Math. Ser.}, pages 47--80.
  Longman, Harlow, 1997.

\bibitem{paper1preprint1}
A.~Barrasso and F.~Russo.
\newblock Backward {S}tochastic {D}ifferential {E}quations with no driving
  martingale, {M}arkov processes and associated {P} seudo {P}artial
  {D}ifferential {E}quations.
\newblock 2017.
\newblock Preprint, hal-01431559, v1.

\bibitem{paper2}
A.~Barrasso and F.~Russo.
\newblock Backward {S}tochastic {D}ifferential {E}quations with no driving
  martingale, {M}arkov processes and associated {P}seudo {P}artial
  {D}ifferential {E}quations. part {II}: Decoupled mild solutions and examples.
\newblock 2017.
\newblock Preprint, hal-01505974.

\bibitem{paperAF}
A.~Barrasso and F.~Russo.
\newblock A note on time-dependent additive functionals.
\newblock {\em Communications on Stochastic Analysis}, 11 no 3:313--334, 9
  2017.

\bibitem{bismut}
J.M. Bismut.
\newblock Conjugate convex functions in optimal stochastic control.
\newblock {\em J. Math. Anal. Appl.}, 44:384--404, 1973.

\bibitem{sant}
R.~Carbone, B.~Ferrario, and M.~Santacroce.
\newblock Backward stochastic differential equations driven by c\`adl\`ag
  martingales.
\newblock {\em Teor. Veroyatn. Primen.}, 52(2):375--385, 2007.

\bibitem{cretarola}
C.~Ceci, A.~Cretarola, and F.~Russo.
\newblock B{SDE}s under partial information and financial applications.
\newblock {\em Stochastic Process. Appl.}, 124(8):2628--2653, 2014.

\bibitem{Confortola}
F.~Confortola, M.~Fuhrman, and J.~Jacod.
\newblock Backward stochastic differential equation driven by a marked point
  process: an elementary approach with an application to optimal control.
\newblock {\em Ann. Appl. Probab.}, 26(3):1743--1773, 2016.

\bibitem{dellmeyer75}
C.~Dellacherie and P.-A. Meyer.
\newblock {\em Probabilit\'es et potentiel}, volume~A.
\newblock Hermann, Paris, 1975.
\newblock Chapitres I {\`a} IV.

\bibitem{dellmeyerB}
C.~Dellacherie and P.-A. Meyer.
\newblock {\em Probabilit\'es et potentiel. {C}hapitres {V} \`a {VIII}}, volume
  1385 of {\em Actualit\'es Scientifiques et Industrielles [Current Scientific
  and Industrial Topics]}.
\newblock Hermann, Paris, revised edition, 1980.
\newblock Th{\'e}orie des martingales. [Martingale theory].

\bibitem{dellmeyerD}
C.~Dellacherie and P.-A. Meyer.
\newblock {\em Probabilit\'es et potentiel. {C}hapitres {XII}--{XVI}}.
\newblock Publications de l'Institut de Math\'ematiques de l'Universit\'e de
  Strasbourg [Publications of the Mathematical Institute of the University of
  Strasbourg], XIX. Hermann, Paris, second edition, 1987.
\newblock Th{\'e}orie des processus de Markov. [Theory of Markov processes].

\bibitem{dynkin1982markov}
E.~B. Dynkin.
\newblock {\em Markov processes and related problems of analysis}, volume~54 of
  {\em London Mathematical Society Lecture Note Series}.
\newblock Cambridge University Press, Cambridge-New York, 1982.

\bibitem{el1997backward}
N.~El~Karoui, S.~Peng, and M.~C. Quenez.
\newblock Backward stochastic differential equations in finance.
\newblock {\em Mathematical finance}, 7(1):1--71, 1997.

\bibitem{frw1}
F.~Flandoli, F.~Russo, and J.~Wolf.
\newblock Some {SDE}s with distributional drift. {I}. {G}eneral calculus.
\newblock {\em Osaka J. Math.}, 40(2):493--542, 2003.

\bibitem{fuhrman2005generalized}
M.~Fuhrman and G.~Tessitore.
\newblock Generalized directional gradients, backward stochastic differential
  equations and mild solutions of semilinear parabolic equations.
\newblock {\em Appl. Math. Optim.}, 51(3):279--332, 2005.

\bibitem{fuosta}
M.~Fukushima, Y.~Oshima, and M.~Takeda.
\newblock {\em Dirichlet forms and symmetric Markov processes}, volume~19 of
  {\em de Gruyter Studies in Mathematics}.
\newblock Walter de Gruyter \& Co., Berlin, 1994.

\bibitem{jacob2005pseudo}
N.~Jacob.
\newblock {\em Pseudo Differential Operators \& Markov Processes: Markov
  Processes And Applications}, volume~3.
\newblock Imperial College Press, 2005.

\bibitem{jacod79}
J.~Jacod.
\newblock {\em Calcul stochastique et probl\`emes de martingales}, volume 714
  of {\em Lecture Notes in Mathematics}.
\newblock Springer, Berlin, 1979.

\bibitem{jacod}
J.~Jacod and A.~N. Shiryaev.
\newblock {\em Limit theorems for stochastic processes}, volume 288 of {\em
  Grundlehren der Mathematischen Wissenschaften [Fundamental Principles of
  Mathematical Sciences]}.
\newblock Springer-Verlag, Berlin, second edition, 2003.

\bibitem{laachir}
I.~Laachir and F.~Russo.
\newblock B{SDE}s, c\`adl\`ag martingale problems, and orthogonalization under
  basis risk.
\newblock {\em SIAM J. Financial Math.}, 7(1):308--356, 2016.

\bibitem{qian}
G.~Liang, T.~Lyons, and Zh. Qian.
\newblock Backward stochastic dynamics on a filtered probability space.
\newblock {\em Ann. Probab.}, 39(4):1422--1448, 2011.

\bibitem{Pardoux}
{\'E}.~Pardoux.
\newblock Backward stochastic differential equations and viscosity solutions of
  systems of semilinear parabolic and elliptic {PDE}s of second order.
\newblock In {\em Stochastic analysis and related topics, {VI} ({G}eilo,
  1996)}, volume~42 of {\em Progr. Probab.}, pages 79--127. Birkh\"auser
  Boston, Boston, MA, 1998.

\bibitem{parpen90}
{\'E}.~Pardoux and S.~Peng.
\newblock Adapted solution of a backward stochastic differential equation.
\newblock {\em Systems Control Lett.}, 14(1):55--61, 1990.

\bibitem{pardoux1992backward}
{\'E}.~Pardoux and S.~Peng.
\newblock Backward stochastic differential equations and quasilinear parabolic
  partial differential equations.
\newblock In {\em Stochastic partial differential equations and their
  applications ({C}harlotte, {NC}, 1991)}, volume 176 of {\em Lecture Notes in
  Control and Inform. Sci.}, pages 200--217. Springer, Berlin, 1992.

\bibitem{PardouxRascanu}
E.~Pardoux and A.~R{\u{a}}{\c{s}}canu.
\newblock {\em Stochastic differential equations, backward {SDE}s, partial
  differential equations}, volume~69 of {\em Stochastic Modelling and Applied
  Probability}.
\newblock Springer, Cham, 2014.

\bibitem{peng1991probabilistic}
S.~Peng.
\newblock Probabilistic interpretation for systems of quasilinear parabolic
  partial differential equations.
\newblock {\em Stochastics Stochastics Rep.}, 37(1-2):61--74, 1991.

\bibitem{protter}
P.~E. Protter.
\newblock {\em Stochastic integration and differential equations}, volume~21 of
  {\em Applications of Mathematics (New York)}.
\newblock Springer-Verlag, Berlin, second edition, 2004.
\newblock Stochastic Modelling and Applied Probability.

\bibitem{roth}
J.P. Roth.
\newblock Op\'erateurs dissipatifs et semi-groupes dans les espaces de
  fonctions continues.
\newblock {\em Ann. Inst. Fourier (Grenoble)}, 26(4):ix, 1--97, 1976.

\bibitem{stroock1975diffusion}
D.~W. Stroock.
\newblock Diffusion processes associated with {L}\'evy generators.
\newblock {\em Z. Wahrscheinlichkeitstheorie und Verw. Gebiete},
  32(3):209--244, 1975.

\bibitem{stroock}
D.~W. Stroock and S.~R.~S. Varadhan.
\newblock {\em Multidimensional diffusion processes}.
\newblock Classics in Mathematics. Springer-Verlag, Berlin, 2006.
\newblock Reprint of the 1997 edition.

\bibitem{ZhuMildBSDE}
R.~Zhu.
\newblock B{SDE} and generalized {D}irichlet forms: the finite-dimensional
  case.
\newblock {\em Infin. Dimens. Anal. Quantum Probab. Relat. Top.},
  15(4):1250022, 40, 2012.

\end{thebibliography}
%\bibliography{biblioPhDBarrasso}

\end{document}